\documentclass[11pt]{amsart}
\usepackage[hmargin=2.5cm,vmargin=2.5cm]{geometry}

\usepackage{amssymb}
\usepackage{amsmath}
\usepackage{amsfonts}
\usepackage{amsthm}
\usepackage{stmaryrd}
\usepackage[all]{xy}
\usepackage{mathrsfs}
\usepackage{graphicx}
\usepackage{hyperref}
\usepackage{color}
\usepackage{multirow}
\usepackage{extarrows}
\usepackage{amscd}
\usepackage{scalerel}
\usepackage{stackengine}
\usepackage{bbm}
\usepackage{mathtools}

\usepackage{mathdots}

\usepackage{cite}
\usepackage{pictexwd,dcpic}

\usepackage[titletoc]{appendix}

\numberwithin{equation}{section}
\newtheorem{theorem}{Theorem}[section]

\newtheorem{conjecture}[theorem]{Conjecture}
\newtheorem{corollary}[theorem]{Corollary}
\newtheorem{lemma}[theorem]{Lemma}
\newtheorem{proposition}[theorem]{Proposition}

\newtheorem{hypo}[theorem]{Hypothesis}

\theoremstyle{definition}

\newtheorem{convention}[theorem]{Convention}
\newtheorem{definition}[theorem]{Definition}

\newtheorem{remark}[theorem]{Remark}

\newcommand{\ra}{\rightarrow}
\newcommand{\lra}{\longrightarrow}

\def\AAA{\mathbb{A}}

\def\CC{\mathbb{C}}

\def\GG{\mathbb{G}}

\def\QQ{\mathbb{Q}}

\def\ZZ{\mathbb{Z}}

\def\calA{\mathcal{A}}

\def\calH{\mathcal{H}}

\def\calJ{\mathcal{J}}

\def\calM{\mathcal{M}}

\def\calR{\mathcal{R}}
\def\calS{\mathcal{S}}
\def\calT{\mathcal{T}}
\def\calU{\mathcal{U}}
\def\calV{\mathcal{V}}

\def\gothf{\mathfrak{f}}

\def\goths{\mathfrak{s}}

\def\gothN{\mathfrak{N}}

\def\scrH{\mathscr{H}}

\def\scrS{\mathscr{S}}

\def\scrU{\mathscr{U}}
\def\scrV{\mathscr{V}}
\def\scrW{\mathscr{W}}

\def\wt{\widetilde}
\def\wh{\widehat}
\def\Lgp{\prescript{L}{}} %Langlands dual group
\def\alg_k{\AAA\mathbbm{l}\mathbbm{g}_{/k}}

\DeclareMathOperator{\im}{Im}

\DeclareMathOperator{\Tr}{Tr}% trace
\DeclareMathOperator{\disc}{disc}% discriminant

\DeclareMathOperator{\Hom}{Hom}

\DeclareMathOperator{\Ind}{Ind}

\DeclareMathOperator{\Irr}{Irr}

\DeclareMathOperator{\Herm}{Herm}% hermitian mapping
\DeclareMathOperator{\Isom}{Isom}% Isomorphism
\DeclareMathOperator{\supp}{supp}

\begin{document}

\title{Theta correspondence and Arthur packets}
\author{Rui Chen \and Jialiang Zou}

%\tableofcontents

\begin{abstract}
In spirit of \cite{MR3866889}, we have established an Arthur's multiplicity formula for even orthogonal or unitary groups with Witt index less than or equal to one. In that multiplicity formula, some local packets defined using the stable range theta lifts are involved. In this paper, we prove that at non-Archimedean places, the definition of the local packets involved in that multiplicity formula is independent of the choice of the dual-pairs used in their construction. Moreover, at those places where the groups are quasi-split, we prove that the local packets involved are the same as the local $A$-packets defined by Arthur/ Mok.
\end{abstract}

\maketitle

\section{Introduction}
In this section we briefly explain our main motivation.\\

Let $F$ be a number field, $\AAA$ be the Adele ring of $F$, and $G$ be a reductive group over $F$. A central question in representation theory is to study the unitary representation $L^2(G(F)\backslash G(\AAA))$. By some classical works, this question is reduced to study the ``discrete part'' of $L^2(G(F)\backslash G(\AAA))$, which is usually denoted by $L_{disc}^2(G)$. When $G$ is a quasi-split classical group, this question is already very well-studied by Arthur \cite{MR3135650} and Mok \cite{MR3338302}. Basically, they decomposed $L_{disc}^2(G)$ into some summands called ``near equivalence classes'', and describe each of these summands using the so called ``Arthur's multiplicity formula''. For non quasi-split groups, Arthur has proposed some strategies to attack this question using the trace formula. For the case of inner forms of unitary groups, Kaletha-M\'{\i}nguez-Shin-White have established some partial results in \cite{kaletha2014endoscopic}. Also, for the case of certain inner forms of classical groups, Ta\"{\i}bi has proved the Arthur's multiplicity formula for the near equivalence classes satisfying some conditions in \cite{MR3908767}. However, except these works, it seems that the general results for non quasi-split groups are not known.\\

On the other hand, the theta correspondence has provided a tool to ``transfer'' results between dual-pairs. In particular, one may ``transfer'' results for a quasi-split group to a possibly non quasi-split group using the theta correspondence. This idea has been used in many papers. In our previous paper \cite{CZ2020AMFPIF}, we have established some results for non quasi-split even orthogonal or unitary groups. Now we briefly recall them.

\subsection{Weak lift, multiplicity preservation and a multiplicity formula}
We continue to let $F$ be a number field. Let $(G, H)$ be a reductive dual pair over $F$ in stable range and $H$ is the larger group, that is, the split rank of $H$ is greater than the dimension of the standard representation of $G$ (we shall make it more precisely later). In some early works of Howe \cite{MR777342} and J-S. Li \cite{MR1448215}, they defined the so-called ``low rank representations'' of classical groups, and showed that theses representations can be characterized by the theta lift from some smaller groups. These works suggest the possibility to study the automorphic spectrum of $G$ in terms of $H$, using the theta correspondence between $(G,H)$. This idea was first exploited by Gan-Ichino \cite{MR3866889}. In their work, they put $G=Mp(2n)$, and $H=SO(2r+1)$, such that $r>2n$. By computing some unramified theta lifts and partial $L$-functions, they attached to each near equivalence class of $G$ an elliptic $A$-parameter (i.e. showed the existence of the weak lift to $GL(2n)$ through the standard representation of the dual group); they also observed that, for those automorphic representations $\pi$ of $G$ with tempered elliptic $A$-parameters, any automorphic realization of $\pi$ must be cuspidal, which implies that
\begin{equation}\label{Intro.Multi.Preserve}
  m_{disc}(\pi)=m_{disc}\left(\theta^{abs}(\pi)\right),
\end{equation}
where $m_{disc}$ means the multiplicity in the automorphic discrete spectrum, and $\theta^{abs}(\pi)$ means the abstract theta lift of $\pi$ to $H$ (i.e. restricted tensor product of the local theta lift at each local place). Combining this with some knowledge on the local theta lift, they proved the Arthur's multiplicity formula for the tempered part of the automorphic discrete spectrum of $Mp(2n)$.\\

In our previous paper \cite{CZ2020AMFPIF}, we put $G$ to be an even orthogonal or unitary group (not necessarily quasi-split), and $H$ be a symplectic or quasi-split unitary group according to $G$. We established the same results as in \cite{MR3866889}. Besides, we also observed that the multiplicity preservation (i.e. equality (\ref{Intro.Multi.Preserve})) also holds if the Witt index of $G$ is less than or equal to one. Hence we obtained a description for the full automorphic discrete spectra of those even orthogonal or unitary groups, by ``pulling back'' the Arthur's multiplicity formula for $H$ to $G$ through the theta lift. For each local place of $F$, we defined the so-called ``$\theta$-packets'' of $G$ to be the theta lift of certain $A$-packets of $H$ (see Section \ref{Def.Theta-Pack}), and ``glue'' them together to get some global packets; we showed that the automorphic discrete spectrum of $G$ can be decomposed according to these global packets (see Theorem \ref{AMF.Anisotropic}). We should emphasize here that, essentially our results in this paper are independent of the results in \cite{CZ2020AMFPIF}; we use \cite{CZ2020AMFPIF} only as one of the motivations for the results shown in this paper.\\

Notice that at each place of $F$, the definition of the ``$\theta$-packets'' is purely local; also, although we assumed that the Witt index of $G$ is less or equal to one, the localization of $G$ at local places could be quite general. It makes sense to ask: are these ``$\theta$-packets'' dependent on the choice of $H$? What are these ``$\theta$-packets''? It turns out that, at least at non-Archimedean places, $\theta$-packets are independent of the choice of $H$; moreover, at those places where our $G$ is quasi-split, we show that the $\theta$-packets are the same as the $A$-packets defined by Arthur and Mok. These are the main results in this paper. Indeed, these questions are already asked in the study of local theta correspondence, known as the ``Adams conjecture''.

\subsection{Adams conjecture}
Now let $F$ be a local field of characteristic $0$. In his paper \cite{MR1021501} Section 4, Adams proposed the following conjecture, which describes the local theta lift in terms of $A$-parameters:
\begin{conjecture}\label{Intro.Adams.Conjecture}
Let $(G,H)$ be a reductive dual pair, such that the dimension of the standard representation of $\Lgp G$ is not greater than that of $\Lgp H$. Let $\psi$ be a local $A$-parameter for $G$. Then one can write down a local $A$-parameter $\theta(\psi)$ for $H$ in terms of $\psi$ explicitly, such that
\begin{enumerate}
  \item[(A)] For an irreducible representation $\pi\in\Pi_\psi^A(G)$, its theta lift $\theta(\pi)$ lies in the corresponding local $A$-packet $\Pi_{\theta(\psi)}^A(H)$ if it is non-zero;
  \item[(B)] If we suppose further that $(G,H)$ is in the stable range, then the theta lift between $(G',H)$ provides us a bijection
    \[
      \theta:\bigsqcup\Pi_\psi^A(G')\lra\Pi_{\theta(\psi)}^A(H),
    \]
    where $G'$ runs over all pure inner forms of $G$.
\end{enumerate}
\end{conjecture}
Our results in this paper can be more or less regarded as a refined version of Conjecture \ref{Intro.Adams.Conjecture} (B) here: we not only prove the bijection between packets as (multi) sets, we also show the consistency of ``labelings'', i.e. the characters of component groups attached to the representations inside the packets. However, in our results, we use the terminology ``$\theta$-packets'' rather than ``$A$-packets'', due to two reasons: 
\begin{itemize}
\item for quasi-split groups, our results imply the two terminologies are the same, so there is no harm; but for non quasi-split groups, at present $A$-packets (under the framework of Arthur, both locally and gloablly, cf. \cite{MR3135650} Chapter 9, or \cite{kaletha2014endoscopic} Chapter 1.6 and 1.7) are not avaliable, so we use $\theta$-packets for substitutions;
\item as explained in the last subsection, a motivation of studying this question is to obtain some results for non quasi-split groups from quasi-split groups via theta lifts; motivated by this purpose and taking the results of \cite{CZ2020AMFPIF} into account, it is natural to use the $\theta$-packets.
\end{itemize}
We should also mention to readers that M{\oe}glin has done many wonderful works on Conjeture \ref{Intro.Adams.Conjecture} in \cite{MR2906916}, based on her explicit construction of local $A$-packets for classical groups. Next we briefly recall her results.

\subsection{Some results of M{\oe}glin}
For a classical group $G$ over some $p$-adic field and a local $A$-parameter $\psi$ of $G$, M{\oe}glin has construct a packet $\Pi_\psi^M(G)$ consisting of irreducible unitary representations using the techniques of Jacquet modules. To be more precise, she defined a collection of unitary representations
\[
	\pi\left(\psi,\underline{t},\underline{\eta}\right),
\]
where the parameter $(\underline{t},\underline{\eta})$ runs over some set $\Sigma_\psi(G)$ (which can be writen down explicitly according to $\psi$). These representations are either zero or irreducible, and the packet $\Pi_\psi^M(G)$ is simply the collection of all non-zero guys. When $G$ is a symplectic group or quasi-split orthogonal group, it was proved in \cite{MR3679701} that $\Pi_\psi^M(G)=\Pi_\psi^A(G)$. To distinguish various notions, we shall call the packet $\Pi_\psi^M(G)$ ``$M$-packet''.\\

Now, let $(G,H)$ be an orthogonal-symplectic dual pair such that the dimension of the standard representation of $\Lgp G$ is not greater than that of $\Lgp H$, and $\psi$ be a local $A$-parameter for the group $G$. Then, for each $(\underline{t},\underline{\eta})\in\Sigma_\psi(G)$, she constructed a pair $(\wt{\underline{t}},\wt{\underline{\eta}})\in\Sigma_{\theta(\psi)}(H)$, where $\theta(\psi)$ is the local $A$-parameter for $H$ predicted by Conjecture \ref{Intro.Adams.Conjecture}. Under some technical assumptions on the local $A$-parameter $\psi$, she asserted that
\[
	\theta\left(\pi\left(\psi,\underline{t},\underline{\eta}\right)\right)=\pi\left(\theta(\psi),\wt{\underline{t}},\wt{\underline{\eta}}\right)
\]
if the LHS and the RHS are both non-zero (see \cite{MR2906916} 5.2 Th\'eor\`eme). She also showed some non-vanishing criterion for the representation $\pi\left(\psi,\underline{t},\underline{\eta}\right)$ along the proof. Hence in some sense, she has proved Conjecture \ref{Intro.Adams.Conjecture} (A) under her assumptions. Her method is purely local. To prove these results, she mastered the Kudla's filtration and Jacquet modules very carefully. Moreover, she pointed out that Conjecture \ref{Intro.Adams.Conjecture} (A) is generally not true by giving some counter-examples.\\

If we suppose further that $(G,H)$ is in the stable range, then M{\oe}glin's technical assumptions are automatically satisfied. From her results, one can easily check the following:
\begin{itemize}
	\item the assignment $(\underline{t},\underline{\eta})\mapsto(\wt{\underline{t}},\wt{\underline{\eta}})$ induces a bijection
		\[
			\theta:\bigsqcup\Sigma_\psi(G')\longrightarrow\Sigma_{\theta(\psi)}(H),
		\]
		where $G'$ runs over all pure inner forms of $G$;
	\item the representation $\pi\left(\psi,\underline{t},\underline{\eta}\right)$ is non-zero if and only if $\pi\left(\theta(\psi),\wt{\underline{t}},\wt{\underline{\eta}}\right)$ is non-zero.
\end{itemize}
It follows that Conjecture \ref{Intro.Adams.Conjecture} (B) holds (with ``$A$-packets'' replaced by ``$M$-packets'').\\

However, we still want to look for an independent proof of Conjecture \ref{Intro.Adams.Conjecture} (B), due to the following reasons:
\begin{itemize}
\item As we have explained, one of the motivations of studying this question is to obtain some results for non quasi-split groups from quasi-split groups via theta lifts. Hence we want to look for an approach which is free of using results from non quasi-split groups. 
\item Except for for the bijectivity, for our purpose, we also need to show the consistency of the ``labelings''. To pass from M{\oe}glin's parametrization to that of Arthur, one still needs to do some computions following \cite{MR3679701}.
\end{itemize}
In the next subsection, we briefly describe the idea of our approach.

\subsection{Idea of the proof}
Our idea of the proof is very simple: we use global methods as much as possible. For quasi-split classical groups, the Arthur's multiplicity formula implies that any localization of an irreducible unitary representation occuring in the automorphic discrete spectrum lies in a local $A$-packet. One can image that, if the Arthur's multiplicity formula has been established for all classical groups, then Conjecture \ref{Intro.Adams.Conjecture} (B) should simply follows from the combination of some easy computations at unramified places and the Arthur's multiplicity formula. From this point of view, many of our lemmas/ propositions in this paper indeed reduce to appropriately globalize a (local) representation. However, since the Arthur's multiplicity formula for general non quasi-split groups has not been established yet, we still need to appeal to M{\oe}glin's explicit constructions of $p$-adic local $A$-packets to deal with some cases. But we should emphasize that our approach only rely on her results for quasi-split groups. As for the ``labelings'', we shall use the intertwining relation to interpret the ``labelings'' of a local $A$-packet as some representation-theoretical quantities. Then we can compute the ``labelings'' using the same techniques as in \cite{MR3573972}.\\ 

For Archimedean places, we also expect that Conjecture \ref{Intro.Adams.Conjecture} (B) holds. Indeed, combining results in \cite{cossutta2009theta} and \cite{MR3947270}, one can conclude that Conjecture \ref{Intro.Adams.Conjecture} (B) holds for unitary dual-pairs when the $A$-parameter $\psi$ is Adams-Johnson. We will not consider Archimedean places in this paper.\\

Now we give a summary of the layout of this paper. We formulate the main theorems in Section 2, taking the chance to recall some preliminaries. After doing some preparation work in Section 3 and recalling some results from other papers that we will use in Section 4, we prove our first result (independence of $\theta$-packets as sets on the choice of some data used in their construction) in Section 5, and we also prove some complementary results in Section 6. Then in Section 7 we recall the local intertwining relation by Arthur, and state an alternative version of it. Finally in Section 8, we prove the local intertwining relation for non quasi-split groups using some techniques developed by Gan-Ichino, and finish the proof of our main results; after that, we briefly summarize some expected and known properties of $\theta$-packets. 

\section*{Acknowledgments}
We would like to thank our supervisor Wee Teck Gan for many useful advices. We would also like to give a special thanks to Wen-Wei Li, Colette M{\oe}glin, and Bin Xu for answering our naive questions. We thank Atsushi Ichino, Alberto M\'{\i}nguez, and Lei Zhang for giving us many useful suggestions. We also thank Caihua Luo and Chuijia Wang for helpful discussions. The second author is supported by an MOE Graduate Research Scholarship.

\section{Statement of main results}\label{Statements.Main.Results}
We first recall some notations from \cite{CZ2020AMFPIF}. Let $F$ be a local or global field, and $E$ be either $F$ or a quadratic field extension of $F$. Let
\[
  c=\begin{cases}
    \textit{the identity of }F\quad &\textit{if }E=F;\\
    \textit{the non-trivial element in }Gal(E/F)\quad &\textit{if }[E:F]=2.
  \end{cases}
\]
In the case $[E:F]=2$, we denote by $\omega_{E/F}$ the quadratic character of $F^\times$ (or $F^\times\backslash\AAA^\times$ if $F$ is global, and similarly in later paragraph) by class field theory, and we fix a trace zero element $\delta\in E^\times$. Let $V=V_{(n)}$ be a finite dimensional vector space over $E$ equipped with a non-degenerate Hermitian $c$-sesquilinear form
\[
  \langle\cdot,\cdot\rangle_V:V\times V\lra E.
\]
We consider the following three cases:
\[
  \begin{cases}
    \textit{Case $O$: } &\textit{$E=F$ and $\dim V=2n$};\\
    \textit{Case $U_0$: } &\textit{$[E:F]=2$ and $\dim V=2n$};\\
    \textit{Case $U_1$: } &\textit{$[E:F]=2$ and $\dim V=2n-1$}.\\
  \end{cases}
\]
where $n\geq 0$ is an integer (we require $n\geq 1$ in Case $U_1$). Sometimes when we want to deal with Case $U_0$ and Case $U_1$ at the same time, we shall simply write ``Case $U$''. Let $G=G(V)$ be the group of elements $g$ in $GL(V)$ such that
\[
  \langle gv,gw\rangle_V=\langle v,w\rangle_V\quad\textit{for }v,w\in V.
\]
If $n=0$, we interpret $G$ as the trivial group. In Case $O$, we let
\[
  \chi_V:F^\times\lra\CC^\times\quad\left(\textit{or}\quad\chi_V:F^\times\backslash\AAA^\times\lra\CC\textit{ if $F$ is global }\right)
\]
be the quadratic character associated to the discriminant of $V$ by class field theory. We set
\[
  \varepsilon_0=\begin{cases}
              1\quad&\textit{Case $O$};\\
              0\quad&\textit{Case $U$}.
          \end{cases}
\]
All pure inner forms of $G=G(V)$ arise in the form $G'=G(V')$ for some space $V'$. When $F$ is a local field, all these spaces $V'$ are classified by some invariants. We briefly describe this classification.

\noindent\underline{\textit{When $F$ is non-Archimedean:}}
\begin{itemize}
\item In Case $O$, these $V'$ are orthogonal spaces with the same dimension and discriminant as $V$. There are exactly two of these spaces, distinguished by their (normalized) Hasse-Witt invariant $\epsilon(V)$ (cf. \cite{MR770063} page 80--81). We shall denote by $V^+$ the one with Hasse-Witt invariant $+1$, and by $V^-$ the one with Hasse-Witt invariant $-1$. Since $V^+$ has the maximal possible Witt index, $V^+$ must be isometric to
\[
  V^+\simeq V_{(d,c)}+\calH^{n-1}
\]
for some $d,c\in F$, where 
\[
  V_{(d,c)}=F[X]/(X^2-d)
\]
is an $2$-dimensional vector space over $F$ equipped with the quadratic form
\[
  a+bX\longmapsto c\cdot(a^2-b^2d),
\]
and $\calH$ is the (orthogonal) hyperbolic plane. We fix such a tuple $(d,c)$ and the isometry, and we shall say that $V^+$ is of type $(d,c)$. Notice that the choice of the tuple $(d,c)$ is not unique.
\item In Case $U$, these $V'$ are Hermitian spaces with the same dimension as $V$. There are also exactly two of these spaces, distinguished by their sign $\epsilon(V)=\omega_{E/F}(\disc V)$. We shall denote by $V^+$ the one with sign $+1$, and by $V^-$ the one with sign $-1$.
\end{itemize}
\underline{\textit{When $F$ is real:}}
\begin{itemize}
\item In this situation, such spaces $V'$ are classified by their signatures $(p,q)$ (satisfying certain conditions). Similar to the non-Archimedean case, in Case $O$, we shall denote by $V^+$ the space with the same dimension, same discriminant as $V$ and with Hasse-Witt invariant $+1$, such that $G(V^+)$ is a quasi-split pure inner form of $G$; and in Case $U$, we shall denote by $V^+$ the space with the same dimension as $V$ and with sign $+1$, such that $G(V^+)$ is a quasi-split pure inner form of $G$.
\end{itemize}
\underline{\textit{When $F$ is complex:}}
\begin{itemize}
\item There is only one such space up to isometry with given dimension, and we shall denote it by $V^+$.
\end{itemize}
When $F$ is a global field, the local-global principle for orthogonal or Hermitian spaces implies that, whenever we are given a collection of local spaces $\{V'_v\}_v$ for all places $v$ of $F$, as long as these local spaces satisfy some ``coherent'' conditions, there will be a space $V'$ over $F$, such that the localization of $V'$ at each place $v$ is isometry to $V'_v$ (see \cite{MR770063} page 225 Theorem 6.10, or page 377 Theorem 6.9). Given $V$ and $G=G(V)$, we let $V^+$ be the space such that for each place $v$ of $F$, $V_v^+$ is (isometry to) the space we have defined in the local situations, i.e. $(V^+)_v\simeq (V_v)^+$.\\

In all cases above, $G^*=G(V^+)$ is quasi-split, and we shall refer it as the quasi-split pure inner form of $G$.
\begin{convention}
In later proofs of our results, we will often use the Arthur's multiplicity formula for quasi-split classical groups. When we say something like ``$V$ is a space such that $G=G(V)$ is quasi-split'', this should be understood as $V=V^+$, and $G$ is the quasi-split pure inner form of itself.
\end{convention}

Let $W=W_{(r)}$ be an 
\[
  \begin{cases}
    2r\textit{-dimensional}\quad &\textit{Case $O$};\\
    (2r+1)\textit{-dimensional}\quad &\textit{Case $U_0$};\\
    2r\textit{-dimensional}\quad &\textit{Case $U_1$}\\
  \end{cases}
\]
vector space over $E$ equipped with a non-degenerate skew-Hermitian $c$-sesquilinear form
\[
  \langle\cdot,\cdot\rangle_W:W\times W\lra E,
\]
such that $W$ is split (in Case $U_0$ we require that the anisotropic kernel of $W$ is the $1$-dimensional skew-Hermitian space represented by $\delta$). Let $H=H(W)$ be the group of elements $h$ in $GL(W)$ such that
\[
  \langle hv,hw\rangle_W=\langle v,w\rangle_W\quad\textit{for }v,w\in W.
\]
The pair $(G,H)$ is then an example of a reductive dual-pair. When $F$ is a local field, we fix a non-trivial additive character $\psi_F$ of $F$, and pick up a pair of characters $(\chi_V,\chi_W)$ of $E^\times$ as follows
\[
  \chi_V=\begin{cases}
    \textit{the quadratic character associated to }V\quad &\textit{Case $O$};\\
    \textit{a character of }E^\times\textit{ such that }\chi_V|_{F^\times}=\omega_{E/F}^{\dim V}\quad &\textit{Case $U$}.
  \end{cases}
\]
\[
  \chi_W=\begin{cases}
    \textit{the trivial character of }F^\times\quad &\textit{Case $O$};\\
    \textit{a character of }E^\times\textit{ such that }\chi_W|_{F^\times}=\omega_{E/F}^{\dim W}\quad &\textit{Case $U$}.
  \end{cases}
\]
When $F$ is a global field, we fix a non-trivial additive character $\psi_F$ of $F\backslash\AAA$, and also characters $(\chi_V,\chi_W)$ of $E^\times\backslash\AAA_E^\times$ similar to the local case. With respect to this tuple of auxiliary data $(\psi_F,\chi_V,\chi_W)$, one can consider the theta lift between $(G,H)$.

\subsection{Theta lifts}
Assume $F$ is local for a moment. With respect to the non-trivial additive character $\psi_F$ of $F$ and the auxliary data $(\chi_V,\chi_W)$, one can define the Weil representation $\omega$ of $G\times H$. For any irreducible representation $\pi$ of $G$, the maximal $\pi$-isotypic quotient of $\omega$ is of the form 
\begin{equation*}
\pi\boxtimes\Theta(\pi)
\end{equation*}
for some smooth representation $\Theta(\pi)$ of $H$ of finite length. Then by the Howe duality \cite{MR985172}, \cite{MR1159105}, \cite{MR3502978}, \cite{MR3454380}, the maximal semi-simple quotient $\theta(\pi)$ of $\Theta(\pi)$ is either zero or irreducible. Similarly, for any irreducible representation $\sigma$ of $H$, we can define $\Theta(\sigma)$ and $\theta(\sigma)$.\\

Suppose next that $F$ is a number field. Fix a non-trivial additive character $\psi_F$ of $F\backslash\AAA$, and also characters $(\chi_V,\chi_W)$. Let $\pi=\otimes_v\pi_v$ be an abstract irreducible representation of $G(\AAA)$ (i.e. a collection of local irreducible representations $\pi_v$ of $G(F_v)$ for all places $v$ of $F$, such that $\pi_v$ is unramified for almost all $v$). At each place $v$ of $F$, we can form the local theta lift $\theta(\pi_v)$ with respect to $(\psi_{F,v},\chi_{V,v},\chi_{W,v})$. Assume that they are all non-vanishing. Then $\theta(\pi_v)$ is irreducible for all $v$ and is unramified for almost all $v$. Hence we may define an abstract irreducible representation
\begin{equation*}
\theta^{abs}(\pi)=\bigotimes_v\theta(\pi_v)
\end{equation*}
of $H(\AAA)$. We call $\theta^{abs}(\pi)$ the abstract theta lift of $\pi$ to $H(\AAA)$.

\subsection{Unitary representations of low rank}\label{J-S.Li.Low.rk}
The notion of rank for unitary representations was first introduced by Howe \cite{MR777342} in the case of symplectic groups and was extended to the case of classical groups by J-S. Li \cite{MR1008803}. Following \cite{MR1008803}, we say that an irreducible unitary representation of $H=H\left(W_{(r)}\right)$ is of low rank if its rank is less than $r$. Such representations are obtained by theta lifts as follows.\\

Let $F$ be a local field. Assume $\dim V<r$. In particular, the reductive dual pair $(G,H)$ is in the stable range. Then for any irreducible representation $\pi$ of $G$, its theta lift $\theta(\pi)$ to $H$ is non-vanishing. Moreover, if $\pi$ is unitary, then by \cite{MR1001840}, so is $\theta(\pi)$. In \cite{MR1008803}, J-S. Li showed that:
\begin{theorem}\label{J-S.Li.Low.rk.Local}
The theta lift provides a bijection
\begin{equation*}
\begin{array}{c}
\displaystyle{\bigsqcup_V}\Irr_{unit}G(V)\times\Big\{\textit{Characters of }E^1\Big\}\\
\Bigg\updownarrow\\
\Big\{\textit{Irreducible unitary representations of }H\textit{ of rank }\dim V\Big\}.
\end{array}
\end{equation*}
where the disjoint union runs over all vector space $V$ over $E$ with fixed dimension, and equipped with a non-degenerate Hermitian $c$-sesquilinear form (in Case $O$ we interpret $E^1$ as the trivial group). The map sends a pair $(\pi,\chi)$ in the first set to a representation $\theta(\pi)\otimes\chi$ of $H$, where we regard $\chi$ as a character of $H$ via the determinant map.
\end{theorem}

This result has a global analog. Let $F$ be a number field and $\sigma=\otimes_v\sigma_v$ an irreducible unitary representation of $H(\AAA)$ which occurs as a subrepresentation of $\calA(H)$, where $\calA(H)$ is the space of automorphic forms of $H$. Then, by \cite{MR1001840} and \cite{MR1448215}, we have: 
\begin{theorem}\label{J-S.Li.Low.rk.Global}
\begin{enumerate}
\item The following are equivalent:
        \begin{itemize}
            \item $\sigma$ is of rank $\dim V$;
            \item $\sigma_v$ is of rank $\dim V$ for all $v$;
            \item $\sigma_v$ is of rank $\dim V$ for some $v$.
        \end{itemize}
\item Suppose that $\sigma$ satisfies the above equivalent conditions. Then, there exists an unique $G=G(V)$ together with an abstract representation $\pi=\otimes_v\pi_v$ of $G(\AAA)$, and an automorphic character $\chi$ of $E^1(\AAA)$, such that
\begin{equation*}
\sigma\simeq\theta^{abs}(\pi)\otimes\chi.
\end{equation*}
\end{enumerate}
\end{theorem}

Finally, we recall another result of J-S. Li, which allows us to lift square-integrable automorphic representations of $G(\AAA)$ to $H(\AAA)$. For any irreducible representation $\pi$ of $G(\AAA)$, we define its multiplicities $m(\pi)$ and $m_{disc}(\pi)$ by
\begin{align*}
m(\pi)&=\dim\Hom_{G(\AAA)}\big(\pi,\calA(G)\big);\\
m_{disc}(\pi)&=\dim\Hom_{G(\AAA)}\big(\pi,\calA^2(G)\big),
\end{align*}
where $\calA^2(G)=\calA(G)\cap L^2_{disc}(G)$. Obviously, $m_{disc}(\pi)\leq m(\pi)$. Likewise, if $\sigma$ is an irreducible representation of $H(\AAA)$, we have its multiplicities $m(\sigma)$ and $m_{disc}(\sigma)$. By \cite{MR1448215}, we have
\begin{theorem}
Assume that $\dim V<r$. Let $\pi$ be an irreducible unitary representation of $G(\AAA)$ and $\theta^{abs}(\pi)$ its abstract theta lift to $H(\AAA)$. Then we have
\begin{equation*}
m_{disc}(\pi)\leq m_{disc}(\theta^{abs}(\pi))\leq m(\theta^{abs}(\pi))\leq m(\pi).
\end{equation*}
\end{theorem}

\subsection{Local and global classifications}
We briefly recall some terminologies and results from \cite{MR3135650} (also \cite{MR3708200}), \cite{MR3338302}, and some other papers.\\

First let $F$ be a local field of characteristic $0$. A local $A$-parameter for the group $G$ is a homomorphism
\[
  \psi: L_F\times SL_2(\CC)\lra \Lgp{G},
\]
where $L_F$ is the Weil-Deligne group of $F$. If there is no further explanations, we will assume that the image of the Weil group under a local $A$-parameter is bounded by default. By composing this homomorphism with the standard representation of $\Lgp{G}$, we can regard a local $A$-parameter as a (conjugate) self-dual representation of $L_E\times SL_2(\CC)$ with certain parity. We denote by $\Psi(G)$ the set of local $A$-parameters for $G$. Following Arthur, we define
\begin{align*}
S_\psi&=Cent(\im\psi,\wh{G}),\\
\calS_\psi&=\pi_0(S_\psi),\\
\overline{\calS_\psi}&=\calS_\psi/\langle z_\psi\rangle,
\end{align*}
where $z_\psi$ is the image of $-I\in \widehat{G}$ in $\calS_\psi$. We shall call $\calS_\psi$ or $\overline{\calS_\psi}$ the component group associated to the local $A$-parameter $\psi$. If we write $\psi$ as
\[
  \psi=\sum_{i\in I_\psi} m_i\cdot\psi_i,
\]
where each $\psi_i=\phi_i\boxtimes S_{d_i}$ is an irreducible representation of $L_E\times SL_2(\CC)$, and $I_\psi$ is the index set of this summation, then as explicated in \cite{MR3202556} Section 8, $\calS_\psi$ has an explicit description in the form
\[
  \calS_\psi=\prod_{i\in I'_\psi}\left(\ZZ/2\ZZ\right)a_i,
\]
where on the RHS, the product runs over the subset $I'_\psi$ of $I_\psi$ containing all $i\in I_\psi$ such that $\psi_i$ is of the same parity as $\psi$; each element in the canonical basis $\{a_i\}$ of $\calS_\psi$ corresponds to such a $\psi_i$. Under this identification, we have
\[
  z_\psi=\sum_{i\in I'_\psi}m_i\cdot a_i,
\]
and $\overline{\calS_\psi}=\calS_\psi/\langle z_\psi\rangle$; where again, the summation on the RHS runs over all $i$ such that $\psi_i$ is of the same parity as $\psi$.\\

When $V=V^+$, i.e. $G=G^*$ is quasi-split, thanks to Arthur and Mok, we can talk about the local $A$-packet $\Pi_\psi^A(G^*)$ associated to the $A$-parameter $\psi$: this is a finite (multi) set of irreducible unitary representations of $G^*$, together with a map to the Pontryagin dual of the component group
\[
  \calJ^A_\scrW:\Pi_\psi^A(G^*)\lra\widehat{\overline{\calS_\psi}}.
\]
This map depends on the choice of a Whittaker datum $\scrW$ of $G^*$. The local $A$-packet $\Pi_\psi^A(G^*)$ can be also regarded as a representation of $\overline{\calS_\psi}\times G^*$ by setting
\[
  \Pi_{\psi}^A(G^*)=\bigoplus_\pi \calJ^A_{\scrW}(\pi)\boxtimes\pi,
\]
where the summation on the RHS runs over all irreducible unitary representations of $G^*$ in $\Pi_{\psi}^A(G^*)$. Sometimes we shall adopt this point of view without any further explanation.\\

Similarly, one can define the local $A$-parameter for the group $H$, and for a local $A$-parameter $\psi_H$ of $H$, one can define and describe the component group $\overline{\calS_{\psi_H}}$ in the same manner. Again, according to Arthur and Mok's works, we can talk about the local $A$-packet $\Pi_{\psi_H}^A(H)$ associated to the $A$-parameter $\psi_H$, which can be regarded as a representation of $\overline{\calS_{\psi_H}}\times H$ by setting
\[
  \Pi_{\psi_H}^A(H)=\bigoplus_\sigma \calJ^A_{\scrW'}(\sigma)\boxtimes\sigma,
\]
where $\scrW'$ is a Whittaker datum of $H$, and the summation on the RHS runs over all irreducible unitary representations of $H$ in $\Pi_{\psi_H}^A(H)$.\\

Now we turn to the global classifications. Let $F$ be a number field. Two irreducible representations $\pi=\otimes_v\pi_v$ and $\pi'=\otimes_v\pi'_v$ of $G(\AAA)$ are said to be nearly equivalent if $\pi_v$ and $\pi'_v$ are equivalent for almost all places $v$ of $F$. The decomposition of automorphic discrete spectrum of $G$ into near equivalence classes will be expressed in terms of elliptic $A$-parameters. Recall that an elliptic $A$-parameter for $G$ is nothing but a formal finite sum 
\begin{equation}\label{paradecomp}
\psi=\sum_i\rho_i\boxtimes S_{d_i},
\end{equation}
where
\begin{itemize}
\item $\rho_i$ is an irreducible (conjugate) self-dual cuspidal automorphic representation of $GL_{n_i}(\AAA_E)$;
\item $S_{d_i}$ is the $d_i$-dimensional irreducible representation of $SL_2(\CC)$;
\item $\sum_i n_id_i=\dim V$;
\item If $d_i$ is odd, then $\rho_i$ is 
  \[
    \begin{cases}
      \textit{orthogonal}\quad &\textit{Case $O$};\\
      \textit{conjugate symplectic }\quad &\textit{Case $U_0$};\\
      \textit{conjugate orthogonal }\quad &\textit{Case $U_1$}.
    \end{cases}
  \]
\item If $d_i$ is even, then $\rho_i$ is 
  \[
    \begin{cases}
      \textit{symplectic}\quad &\textit{Case $O$};\\
      \textit{conjugate orthogonal }\quad &\textit{Case $U_0$};\\
      \textit{conjugate symplectic }\quad &\textit{Case $U_1$}.
    \end{cases}
  \]
\item If $(\rho_i,d_i)=(\rho_j,d_j)$, then $i=j$;
\item In Case $O$, if we denote the central character of $\rho_i$ by $\omega_i$, then
  \[
    \prod_i\omega_i^{d_i}=\chi_V.
  \]
\end{itemize}
If further $d_i=1$ for all $i$, then we say that $\psi$ is generic. We denote the set of all elliptic $A$-parameters by $\Psi_{ell}(G)$. For each place $v$ of $F$, let  
\begin{equation*}
\psi_v:L_{F_v}\times SL_2(\CC)\lra\Lgp{G_v}
\end{equation*}
be the localization of $\psi$ at $v$. Here
\begin{equation*}
L_{F_v}=\begin{cases}
\textit{the Weil group of }F_v\quad&\textit{if }v\textit{ is Archimedean};\\
\textit{the Weil-Deligne group of }F_v\quad&\textit{if }v\textit{ is non-Archimedean},\\
\end{cases}
\end{equation*}
and we use the local Langlands correspondence for the general linear groups. We associate to it an $L$-parameter $\phi_{\psi_v}:L_{F_v}\ra\Lgp{G_v}$ by
\begin{equation*}
\phi_{\psi_v}(w)=\psi_v\left(w,\left(\begin{array}{cc}
                                    {|w|^\frac{1}{2}} & {} \\
                                    {} & {|w|^{-\frac{1}{2}}}
                                 \end{array}\right)\right).
\end{equation*}
We have
\begin{theorem}\label{Decompose.NEC}
There exists a decomposition
\begin{equation*}
L^2_{disc}(G)=\bigoplus_{\psi\in\Psi_{ell}(G)} L^2_\psi(G),
\end{equation*}
where $L^2_\psi(G)$ is a full near equivalence class of irreducible representations $\pi$ in $L^2_{disc}(G)$ such that the $L$-parameter of $\pi_v$ is $\phi_{\psi_v}$ for almost all places $v$ of $F$.
\end{theorem}

When $G=G^*$ is quasi-split, a further decomposition of each near equivalence class is avaliable, known as the Arthur's multiplicity formula for $G^*$. Fix a global Whittaker datum $\scrW$ of $G^*$. Given an elliptic $A$-parameter $\psi$, we define the global packet $\Pi_\psi^A(G^*)$ associated to $\psi$ as the restricted tensor product of the local $A$-packets
\begin{align*}
\Pi_\psi^A(G^*)&=\otimes'_v\Pi_{\psi_v}^A(G^*_v)\\
&=\{\pi=\otimes'_v\pi_v~|~\pi_v\in\Pi_{\psi_v}^A(G^*_v),~\pi_v\textit{ unramified with the $L$-parameter $\phi_{\psi_v}$ for almost all }v\}.
\end{align*}
We then have a map
\begin{align*}
\calJ_{\scrW}^A:\Pi_\psi^A(G^*)&\lra\wh{\overline{\calS_\psi}},\\
\pi&\longmapsto\calJ^A_{\scrW}(\pi),\\
\calJ^A_{\scrW}(\pi)(x)&\coloneqq\prod_v\calJ_{\scrW_v}^A(\pi_v)(x_v),
\end{align*}
where $x\in\calS_\psi$ and $x_v$ is the localization of $x$ at $v$. We can also define the so-called canonical sign character $\epsilon_\psi\in\wh{\calS_\psi}$ following \cite{MR3135650} page 47, or \cite{MR3338302} page 29. We put
\begin{equation*}
\Pi_\psi^A(G^*,\epsilon_\psi)=\left\{\pi\in\Pi_\psi^A(G^*)~|~\calJ_{\scrW}^A(\pi)=\epsilon_\psi\right\}.
\end{equation*}
Then the main global Theorems in \cite{MR3135650} and \cite{MR3338302} assert that
\begin{theorem}\label{AMF.QS}
Let $\psi$ be an elliptic $A$-parameter for $G^*$. Then we have the decomposition
\begin{equation*}
L^2_\psi(G^*)=\bigoplus_{\pi\in\Pi_\psi^A(G^*,\epsilon_\psi)}\pi.
\end{equation*}
\end{theorem}

Similarly, one can define the global $A$-parameters for the group $H$, and for a global elliptic $A$-parameter $\psi_H$ of $H$, one can define the associated global $A$-packets by gluing local $A$-packets in the same manner. Again, according to Arthur and Mok's works, Theorem \ref{Decompose.NEC} and Theorem \ref{AMF.QS} also holds for $H$.

\subsection{Remarks on Whittaker data}\label{Whittaker.Data}
Since the local or global classification of both $G^*$ and $H$ depend on the choices of Whittaker data on $G^*$ and $H$, and the theta lift also depends on the choice of an additive character, we need to choose these data in a compatible way. We now briefly describe the way we choose these data.\\

Let $F$ be a local or global field. Firstly we fix a non-trivial additive character $\psi_F$ of $F$ (or $F\backslash\AAA$ if $F$ is global). Next we fix an Whittaker datum $\scrW$ of $G^*$ as follows.\\

In Case $O$, $G^*=G(V^+)$ is an even orthogonal groups. As explicated at the beginning of this section, we fix an isometry
\[
  V^+\simeq V_{(d,c)}+\calH^{n-1}
\]
for some $d,c\in F$, where $\calH$ is the (orthogonal) hyperbolic plane. The images of $1,X\in F[X]$ in $V_{(d,c)}$ are denoted by $e,e'$, respectively. For $1\leq k\leq n-1$, we write the $k$-th hyperbolic plane $\calH=Fv_k+Fv_k^*$ with 
\[
  \langle v_k,v_k\rangle_V=\langle v_k^*,v_k^*\rangle_V=0\quad\textit{and}\quad \langle v_k,v_k^*\rangle_V=1,
\]
and we set
\[
  X_k=Fv_1+\cdots+Fv_k\quad\textit{and}\quad X_k^*=Fv_1^*+\cdots+Fv_k^*.
\]
We denote by $B=TU$ the $F$-rational Borel subgroup of $G^*$ stabilizing the complete flag
\[
  X_1\subset\cdots\subset X_{n-1},
\]
where $T$ is the $F$-rational torus stabilizing the lines $Fv_i$ for $1\leq k\leq n-1$. We define a generic character $\mu_c$ of $U$ by
\[
  \mu_c(u)=\psi_F(\langle uv_2,v_1^*\rangle_V+\cdots+\langle uv_{n-1},v_{n-2}^*\rangle_V+\langle ue,v_{n-1}^*\rangle_V).
\]
Let $\scrW=(U,\mu_c)$. Note that in fact $\scrW$ does not depend on the choice of the additive character $\psi_F$, but only depends on the constant $c$ we have picked up.\\

In Case $U_0$, $G^*=G(V^+)$ is an even unitary groups. Recall that we have fixed a trace zero element $\delta\in E^\times$. Let $V^+$ be the $2n$-dimensional Hermitian space over $E$ such that $G^*=U(V^+)$. Since $G^*$ quasi-split, the Witt index of $V^+$ is $n$. We choose a basis $\{v_i,v_i^*~|~i=1,\cdots,n\}$ of $V^+$ such that
\[
  \langle v_i,v_j\rangle_V=\langle v_i^*,v_j^*\rangle_V=0\quad\textit{and}\quad \langle v_i,v_j^*\rangle_V=\delta_{i,j}
\]
for $1\leq i,j\leq n$. We set
\[
  X_k=Ev_1+\cdots+Ev_k\quad\textit{and}\quad X_k^*=Ev_1^*+\cdots+Ev_k^*
\]
for $1\leq i,j\leq n$. We denote by $B=TU$ the $F$-rational Borel subgroup of $G^*$ stabilizing the complete flag
\[
  X_1\subset\cdots\subset X_{n},
\]
where $T$ is the $F$-rational torus stabilizing the lines $Ev_i$ for $1\leq k\leq n$. We define a generic character $\mu$ of $U$ by
\[
  \mu(u)=\psi_F\Big(\frac{1}{2}\Tr_{E/F}\big(\delta\cdot(\langle uv_2,v_1^*\rangle_V+\cdots+\langle uv_{n},v_{n-1}^*\rangle_V+\langle uv_{n}^*,v_{n}^*\rangle_V)\big)\Big).
\]
Let $\scrW=(U,\mu)$.\\

In Case $U_1$, there is an unique Whittaker datum $\scrW$ of $G^*$.\\

Finally we fix a Whittaker datum $\scrW'$ of $H$ as follows.\\

In Case $O$, $W$ is the $2r$-dimensional symplectic space. We choose a basis $\{w_i,w_i^*~|~i=1,\cdots,r\}$ of $W$ such that
\[
  \langle w_i,w_j\rangle_W=\langle w_i^*,w_j^*\rangle_W=0\quad\textit{and}\quad \langle w_i,w_j^*\rangle_W=\delta_{i,j}
\]
for $1\leq i,j\leq r$. We set
\[
  Y_k=Fw_1+\cdots+Fw_k\quad\textit{and}\quad Y_k^*=Fw_1^*+\cdots+Fw_k^*
\]
for $1\leq i,j\leq r$. We denote by $B'=T'U'$ the $F$-rational Borel subgroup of $H$ stabilizing the complete flag
\[
  Y_1\subset\cdots\subset Y_{r},
\]
where $T'$ is the $F$-rational torus stabilizing the lines $Fw_i$ for $1\leq k\leq r$. We define a generic character $\mu'_c$ of $U'$ by
\[
  \mu'_c(u)=\psi_F\Big(c\cdot\big(\langle uw_2,w_1^*\rangle_V+\cdots+\langle uw_{r},w_{r-1}^*\rangle_V+\langle uw_{r}^*,w_{r}^*\rangle_V\big)\Big),
\]
where the constant $c$ is the one appearing in the isometry $V^+\simeq V_{(d,c)}+\calH^{n-1}$ we have fixed. Let $\scrW'=(U',\mu'_c)$. In this case we also define another generic character $\mu'_1$ of $U'$ by
\[
  \mu'_1(u)=\psi_F\Big(\langle uw_2,w_1^*\rangle_V+\cdots+\langle uw_{r},w_{r-1}^*\rangle_V+\langle uw_{r}^*,w_{r}^*\rangle_V\Big),
\]
and let $\scrW'_1=(U',\mu'_1)$.\\

In Case $U_0$, $W$ is an $(2r+1)$-dimensional skew-Hermitian space. Hence there is an unique Whittaker datum $\scrW'$ of $H$. \\

In Case $U_1$, $W$ is an $2r$-dimensional skew-Hermitian space. We choose a basis $\{w_i,w_i^*~|~i=1,\cdots,r\}$ of $W$ such that
\[
  \langle w_i,w_j\rangle_W=\langle w_i^*,w_j^*\rangle_W=0\quad\textit{and}\quad \langle w_i,w_j^*\rangle_W=\delta_{i,j}
\]
for $1\leq i,j\leq r$. We set
\[
  Y_k=Fw_1+\cdots+Fw_k\quad\textit{and}\quad Y_k^*=Fw_1^*+\cdots+Fw_k^*
\]
for $1\leq i,j\leq r$. We denote by $B'=T'U'$ the $F$-rational Borel subgroup of $H$ stabilizing the complete flag
\[
  Y_1\subset\cdots\subset Y_{r},
\]
where $T'$ is the $F$-rational torus stabilizing the lines $Fw_i$ for $1\leq k\leq r$. We define a generic character $\mu'$ of $U'$ by
\[
  \mu'(u)=\psi_F\Big(\frac{1}{2}\Tr_{E/F}\big(\langle uw_2,w_1^*\rangle_V+\cdots+\langle uw_{r},w_{r-1}^*\rangle_V+\langle uw_{r}^*,w_{r}^*\rangle_V\big)\Big).
\]
Let $\scrW'=(U',\mu')$.

\subsection{Theta packets}\label{Def.Theta-Pack}
Now we recall the definition of the main object we want to study in this paper, the so-called ``$\theta$-packet'', which is defined in \cite{CZ2020AMFPIF}. From now on, we let $F$ be a local field of characteristic $0$. We fix Whittaker data $\scrW$ and $\scrW'$ of the group $G^*$ and $H$, depending on the additive character $\psi_F$, as explicated in Section \ref{Whittaker.Data}. Assume now $\dim V<r$. Let $\psi$ be a local $A$-parameter for $G$, and 
\[
  \theta(\psi)=\psi\chi_W^{-1}\chi_V+\chi_V\boxtimes S_{2r-2n+1}
\]
be a local $A$-parameter for $H$ (sometimes we shall also write it as $\theta^r(\psi)$ to emphasize its dependence on the integer $r$). There is an obvious map 
\[
  \calS_\psi\lra\calS_{\theta(\psi)},
\]
sending an element $a_i\in\calS_\psi$ corresponding to an irreducible constituent $\psi_i$ of $\psi$, to the element $a'_i\in\calS_{\theta(\psi)}$ corresponding to the irreducible constituent $\psi_i\chi_W^{-1}\chi_V$ of $\theta(\psi)$. We write the local $A$-packet $\Pi_{\theta(\psi)}^A(H)$ as a representation of $\overline{\calS_{\theta(\psi)}}\times H$
\[
  \Pi_{\theta(\psi)}^A(H)=\bigoplus_\sigma \calJ^A_{\scrW'}(\sigma)\boxtimes\sigma,
\]
as $\sigma$ runs over all irreducible unitary representations of $H$ in $\Pi_{\theta(\psi)}^A(H)$. Then the $\theta$-packet of $G$ associated to the $A$-parameter $\psi$ is defined as
\[
  \Pi_\psi^\theta(G)=\bigoplus_\sigma \left(\calJ^A_{\scrW'}(\sigma)\Big|_{\calS_\psi}\right)\boxtimes\theta(\sigma),
\]
where $\theta(\sigma)$ is the theta lift of $\sigma$ to the group $G$ with respect to $(\psi_F,\chi_V,\chi_W)$, and we regard $\calS_\psi$ as a subgroup of $\calS_{\theta(\psi)}$ via the obvious map between them. Sometimes when we want to emphasize the possible dependence of $\Pi_\psi^\theta(G)$ on the choice of $H=H(W_{(r)})$, we shall also use the notation $\Pi_\psi^{\theta,r}(G)$. The $\theta$-packet $\Pi_\psi^\theta(G)$ can be also regarded as a (multi) set of irreducible unitary representations of $G$, together with a map 
\[
  \calJ_{\psi_F}:\Pi_\psi^\theta(G)\lra\wh{\calS_\psi}
\]
by sending $\theta(\sigma)$ to $\calJ^A_{\scrW'}(\sigma)\Big|_{\calS_\psi}$.
\begin{remark}
For those $\psi$ that do not have bounded image on the Weil group, but come from a localization of some global elliptic $A$-parameter for an even orthogonal or unitary group, we can define the $\theta$-packet $\Pi_\psi^\theta(G)$ in the same manner.
\end{remark}

In this paper, we mainly consider the case that $F$ is non-Archimedean. Our main theorem in this paper is
\begin{theorem}\label{Main.Theorem.Packets}
Let $F$ be a non-Archimedean local field. For all local $A$-parameter $\psi$ of $G=G(V)$ (with bounded image on the Weil group), we have:
\begin{enumerate}
  \item the definition of the packet $\Pi_\psi^\theta(G)$ is indeed independent of the choice of $H=H(W_{(r)})$ with $r>\dim V$;
  \item if $V=V^+$, i.e. $G=G^*$ is quasi-split, then
  \[
    \Pi_\psi^\theta(G)=\Pi_\psi^A(G) 
  \]
  as representations of $\calS_\psi\times G^*$.
\end{enumerate}
\end{theorem}
\begin{remark}
When $F$ is Archimedean, we also expect the same results hold. In the case that $G$ is a real unitary group and $\psi$ is Adams-Johnson, one can refer to \cite{MR3947270} Th\'eor\`eme 1.1.
\end{remark}

Along the way of proving this theorem, we also deduce some by-products. We shall summarize them in the last part of this paper.

\section{Compatibility with parabolic inductions}\label{Compatible.Theta.N.Ind}
In this section, we imitate \cite{MR3573972} Section 8 to construct an equivariant map, and use this equivariant map to deduce some results in the context of the theta correspondence which will be used later.
\subsection{A mixed model}
We shall use a mixed model to do some computations. The same model is also used in \cite{MR3573972} and \cite{MR3788848}, so readers may also consult these two papers for details. For the convenience of readers, we briefly recall it. Suppose that we have
\[
  V=X+V_0+X^*
\]
for some $k$-dimensional totally isotropic subspace $X$ and $X^*$ of $V$. Let $n_0=n-k$, and $r_0=r-k$. Then there is a maximal parabolic subgroup $P=M_PU_P$ of $G$ stabilizing $X$, where $M_P$ is the Levi component of $P$ stabilizing $X^*$ and $U_P$ is the unipotent radical of $P$. We have 
\begin{equation*}
\begin{aligned}
M_P&=\{m_P(a)\cdot h_0~|~a\in GL(X),h_0\in G(V_0)\},\\
U_P&=\{u_P(b)\cdot u_P(c)~|~b\in\Hom(V_0,X),c\in\Herm(X^*,X)\},
\end{aligned}
\end{equation*}

where
\begin{align*}
m_{P}(a)&=\left(\begin{array}{ccc}
                     a & ~ & ~ \\
                     ~ & 1_{V_{0}} & ~ \\ 
                     ~ & ~ & \left(a^{*}\right)^{-1}
                \end{array}\right),\\
u_{P}(b)&=\left(\begin{array}{ccc}
                     1_{X} & b & -\frac{1}{2} b b^{*}\\ 
                     ~ & 1_{V_{0}} & -b^{*} \\ 
                     ~ & ~ & 1_{X^{*}}
                \end{array}\right),\\
u_{P}(c)&=\left(\begin{array}{ccc}
                     {1_{X}} & {} & {c} \\ 
                     {} & {1_{V_{0}}} & {} \\ 
                     {} & {} & {1_{X^{*}}}
                \end{array}\right),
\end{align*}
and 
\begin{equation*}
\Herm(X^*,X)=\{c\in\Hom(X^*,X)~|~c^*=-c\}.
\end{equation*}
Here, the elements $a^*\in GL(X^*)$, $b^*\in\Hom(X^*,V_0)$, and $c^*\in\Hom(X^*,X)$ are the adjoints of $a$, $b$, and $c$ respectively, under the $c$-Hermitian form on $V$. Put
\begin{equation*}
\rho_P=\frac{\dim V_0+k-\varepsilon_0}{2}\quad w_P=\left(\begin{array}{ccc}
                                    {} & {} & -I_X\\
                                    {} & 1_{V_0} & {}\\
                                    -I_X^{-1} & {} & {}

                                 \end{array}\right),
\end{equation*}
where we pick up $I_X\in\Isom(X^*,X)$ in an obvious way after choosing a basis of $X$.\\

Similarly, suppose that we have
\[
  W=Y+W_0+Y^*
\]
for some $k$-dimensional totally isotropic subspace $Y$ and $Y^*$ of $W$. Then there is a maximal parabolic subgroup $Q=M_QU_Q$ of $G$ stabilizing $Y$, where $M_Q$ is the Levi component of $Q$ stabilizing $Y^*$ and $U_Q$ is the unipotent radical of $Q$. For $a\in GL(Y)$, $b\in\Hom(W_0,Y)$ and $c\in\Herm(Y^*,Y)$, we define elements $m_Q(a)\in M_Q$ and $u_Q(b),u_Q(c)\in U_Q$ as above. Put
\begin{equation*}
\rho_Q=\frac{\dim W_0+k+\varepsilon_0}{2},\quad w_Q=\left(\begin{array}{ccc}
                                    {} & {} & -I_Y\\
                                    {} & 1_{W_0} & {}\\
                                    I_Y^{-1} & {} & {}
                                    \end{array}\right),
\end{equation*}
where we pick up $I_Y\in\Isom(Y^*,Y)$ in an obvious way after choosing a basis of $Y$.\\

We write:
\begin{itemize}
\item $\omega$ for the Weil representation $\omega_{\psi_F,V,W}$ of $G\times H$ on a space $\scrS$;
\item $\omega_0$ for the Weil representation $\omega_{\psi_F,V,W_0}$ of $G_0\times H$ on a space $\scrS_0$, where $G_0=G(V_0)$;
\item $\omega_{00}$ for the Weil representation $\omega_{\psi_F,V_0,W_0}$ of $G_0\times H_0$ on a space $\scrS_{00}$, where $H_0=H(W_0)$.
\end{itemize}
We take a mixed model 
\begin{equation*}
\scrS=\scrS(W\otimes X^*)\otimes\scrS_0
\end{equation*}
of $\omega$, where we regard $\scrS$ as a space of functions on $W\otimes X^*$ with values in $\scrS_0$. Similarly, we take a mixed model 
\begin{equation*}
\scrS_0=\scrS(Y^*\otimes V_0)\otimes\scrS_{00}
\end{equation*}
of $\omega_0$, where we regard $\scrS_0$ as a space of functions on $Y^*\otimes V_0$ with values in $\scrS_{00}$. Also, we write:
\begin{itemize}
\item $\rho_0$ for the Heisenberg representation of $\scrH(W\otimes V_0)$ on $\scrS_0$ with central character $\psi_F$;
\item $\rho_{00}$ for the Heisenberg representation of $\scrH(W_0\otimes V_0)$ on $\scrS_{00}$ with central character $\psi_F$.
\end{itemize}
We can derive the following formulas for the Weil representations $\omega$ and $\omega_0$. For $\varphi\in\scrS$ and $x\in W\otimes X^*$, we have
\begin{align*}
(\omega(h)\varphi)(x)&=\omega_0(h)\varphi(h^{-1}x), & h&\in H,\\
(\omega(g_0)\varphi)(x)&=\omega_0(g_0)\varphi(x), & g_0&\in G_0,\\
(\omega(m_P(a))\varphi)(x)&=\chi_W(\det a)|\det a|^{\dim W/2}\varphi(a^*x), & a&\in GL(X),\\
(\omega(u_P(b))\varphi)(x)&=\rho_0((b^*x,0))\varphi(x), & b&\in\Hom(V_0,X),\\
(\omega(u_P(c))\varphi)(x)&=\psi_F(\frac{1}{2}\langle cx,x\rangle)\varphi(x), & c&\in\Herm(X^*,X),\\
(\omega(w_P)\varphi)(x)&=\gamma_W^{-k}\int_{W\otimes X}\varphi(I_X^{-1}y)\psi(-\langle y,x\rangle)dy,
\end{align*}
where $\gamma_W$ is a certain constant. Also, for $\varphi_0\in\scrS_0$ and $x\in Y^*\otimes V_0$, we have
\begin{align*}
(\omega_0(g_0)\varphi_0)(x)&=\omega_{00}(g_0)\varphi_0(g_0^{-1}x), & g_0&\in G_0,\\
(\omega_0(h_0)\varphi_0)(x)&=\omega_{00}(h_0)\varphi_0(x), & h_0&\in H_0,\\
(\omega_0(m_Q(a))\varphi_0)(x)&=\chi_V(\det a)|\det a|^{\dim V_0/2}\varphi_0(a^*x), & a&\in GL(Y),\\
(\omega_0(u_Q(b))\varphi_0)(x)&=\rho_{00}((b^*x,0))\varphi_0(x), & b&\in\Hom(W_0,Y),\\
(\omega_0(u_Q(c))\varphi_0)(x)&=\psi_F(\frac{1}{2}\langle cx,x\rangle)\varphi_0(x), & c&\in\Herm(Y^*,Y),\\
(\omega_0(w_Q)\varphi_0)(x)&=\gamma_V^{-k}\int_{Y\otimes V_0}\varphi_0(I_Y^{-1}y)\psi(-\langle y,x\rangle)dy,
\end{align*}
where $\gamma_V$ is a certain constant. Moreover, we have
\begin{align*}
(\rho_0\left((y+y',0)\right)\varphi_0)(x)&=\psi(\langle x,y\rangle+\frac{1}{2}\langle y',y\rangle)\varphi_0(x+y'), & y&\in Y\otimes V_0,\\
& & y'&\in Y^*\otimes V_0,\\
(\rho_0\left((y_0,0)\right)\varphi_0)(x)&=\rho_{00}\left((y_0,0)\right)\varphi_0(x), & y_0&\in W_0\otimes V_0.
\end{align*}

\subsection{An equivariant map}
In this subsection we construct the explicit equivariant map. A non-vanishing result of this map will be important to us. The construction is roughly the same as \cite{MR3573972} Section 8, except at one place we use the ``small'' theta lift whereas in Gan-Ichino's paper they use the ``big'' theta lift.\\

First of all we need to fix Haar measures on various groups. For this part, we simply follow \cite{MR3788848} Section 6.3 in Case $O$, and \cite{MR3573972} Section 7.2 in Case $U$. We shall identify $GL(X)$ with $GL_k(F)$ using a basis $\{v_i\}_i$ for $X$, and similarly identify $GL(Y)$ with $GL_k(F)$ using a basis $\{w_i\}_i$ for $Y$. We write $\{v_i^*\}_i$ and $\{w_i^*\}_i$ for the dual basis for $X^*$ and $Y^*$ respectively. Then we can define an isomorphism $i:GL(X)\ra GL(Y)$ via these identifications. Put
\begin{align*}
e&=w_1\otimes v_1^*+\cdots+w_k\otimes v_k^*\in Y\otimes X^*,\\
e^*&=w_1^*\otimes v_1+\cdots+w_k^*\otimes v_k\in Y^*\otimes X.
\end{align*}
Then $i(a)^ce=a^*e$ and $\left(i(a)^c\right)^*e^*=ae^*$ for $a\in GL(X)$.\\

For $\varphi\in\scrS=\scrS(W\otimes X^*)\otimes\scrS_0$, we define functions $\gothf(\varphi)$, $\hat{\gothf}(\varphi)$ on $G\times H$ with values in $\scrS_0$ by
\begin{align*}
\gothf(\varphi)(g,h)&=(\omega(g,h)\varphi)\left(\begin{matrix}
                                                e\\
                                                0\\
                                                0
                                                \end{matrix}\right),\\
\hat{\gothf}(\varphi)(g,h)&=\int_{Y\otimes X^*}(\omega(g,h)\varphi)\left(\begin{matrix}x\\0\\0\end{matrix}\right)\psi(\langle x,e^*\rangle)dx
\end{align*}
for $g\in G$ and $h\in H$. Here we write an element in $W\otimes X^*$ as a block matrix relative to the decomposition $W\otimes X^*=Y\otimes X^*\oplus W_0\otimes X^*\oplus Y^*\otimes X^*$. We also define functions $f(\varphi)$, $\hat{f}(\varphi)$ on $G\times H$ with values in $\scrS_{00}$ by
\begin{align*}
f(\varphi)(g,h)&=ev(\gothf(\varphi)(g,h)),\\
\hat{f}(\varphi)(g,h)&=ev(\hat{\gothf}(\varphi)(g,h)).
\end{align*}
Here $ev:\scrS_0=\scrS(Y^*\otimes V_0)\otimes\scrS_{00}\ra\scrS_{00}$ is the evaluation at $0\in Y^*\otimes V_0$. If $f=f(\varphi)$ or $\hat{f}(\varphi)$, then
\begin{align*}
f(ug,u'h)&=f(g,h), & u&\in U_P,\\
& & u'&\in U_Q,\\
f(g_0g,h_0h)&=\omega_{00}(g_0,h_0)f(g,h), & g_0&\in G_0,\\
& & h_0&\in H_0,\\
f(m_P(a)g,m_Q(i(a)^c)h)&=(\chi_V^c\chi_W)(\det a)|\det a|^{\rho_P+\rho_Q}f(g,h), & a&\in GL(X).
\end{align*}

Let $\tau$ be an irreducible unitary representation of $GL_k(E)$ on a space $\scrV_\tau$. We may regard $\tau$ as a representation of $GL(X)$ or $GL(Y)$ via the above identifications. Let $\pi_0$ and $\sigma_0$ be irreducible unitary representations of $G_0$ and $H_0$ on spaces $\scrV_{\pi_0}$ and $\scrV_{\sigma_0}$ respectively. Fix non-zero invariant non-degenerate bilinear forms $\langle\cdot,\cdot\rangle$ on $\scrV_\tau\times\scrV_{\tau^\vee}$, $\scrV_{\pi_0}\times\scrV_{{\pi_0}^\vee}$, and $\scrV_{\sigma_0}\times\scrV_{{\sigma_0}^\vee}$. Let
\begin{equation*}
\langle\cdot,\cdot\rangle:(\scrV_\tau\otimes\scrV_{{\sigma_0}^\vee})\times\scrV_{\tau^\vee}\lra\scrV_{{\sigma_0}^\vee}
\end{equation*}
be the induced map. Assume that 
\begin{equation*}
\sigma_0=\theta_{\psi_F,V_0,W_0}(\pi_0).
\end{equation*}
We fix a non-zero $G_0\times H_0$-equivariant map
\begin{equation*}
\calT_{00}:\omega_{00}\otimes\sigma_0^\vee\lra\pi_0.
\end{equation*}
For $\varphi\in\scrS$, $\varPhi_s\in\Ind_Q^{H}(\tau_s^c\chi_V^c\boxtimes\sigma_0^\vee)$, $g\in G$, $\check{v}\in\scrV_{\tau^\vee}$, and $\check{v}_0\in\scrV_{\pi_0^\vee}$, put
\begin{align*}
\langle\calT_s(\varphi&\otimes\varPhi_s)(g),\check{v}\otimes\check{v}_0\rangle\\
&=L\left(s-s_0,\tau\right)^{-1}\times\int_{U_QH_0\backslash H}\langle\calT_{00}(\hat{f}(\varphi)(g,h)\otimes\langle\varPhi_s(h),\check{v}\rangle),\check{v}_0\rangle dh,
\end{align*}
where we set
\[ 
  s_0=r-n,
\]
and $L(s,\tau)$ is the standard $L$-factor of $\tau$. Similar to \cite{MR3573972}, one can show that
\begin{proposition}\label{Non.Vanish.Equi.Map}
\begin{enumerate}
\item The integral is absolutely convergent when $\Re(s)\gg 0$ and admits a holomorphic continuation to $\CC$. Hence we obtain a $G\times H$-equivariant map
\begin{equation*}
\calT_s:\omega\otimes\Ind_Q^{H}\left(\tau_s^c\chi_V^c\boxtimes\sigma_0^\vee\right)\lra\Ind_P^{G}\left(\tau_s\chi_W\boxtimes\pi_0\right).
\end{equation*}
\item When $\Re(s)\ll 0$, we have
\begin{align*}
\langle\calT_s(\varphi&\otimes\varPhi_s)(h),\check{v}\otimes\check{v}_0\rangle\\
&=L\left(s-s_0,\tau\right)^{-1}\cdot\gamma\left(s-s_0,\tau,\psi_F\right)^{-1}\\
&\times\int_{U_QH_0\backslash H}\langle\calT_{00}({f}(\varphi)(g,h)\otimes\langle\varPhi_s(h),\check{v}\rangle),\check{v}_0\rangle dh.
\end{align*}
\item Let $N_\tau >0$ be a positive real number such that $L(s, \tau)$ has no zeros or poles outside the stripe
\begin{equation*}
-N_\tau<\Re(s)<N_\tau.
\end{equation*}
Assume that $r>N_\tau+n$. Let $\varPhi\in\Ind_Q^{H}\left(\tau_s^c\chi_V^c\boxtimes\sigma_0^\vee\right)$. If $\varPhi\neq0$, then there exists $\varphi\in\scrS$ such that
\begin{equation*}
\calT_0(\varphi\otimes\varPhi)\neq0.
\end{equation*}
\end{enumerate}
\end{proposition}
\begin{proof}
The proofs of the first two statements are totally the same as \cite{MR3573972} Lemma 8.1 and Lemma 8.2. As for the proof of the last statement, it is also similar to Lemma 8.3 in Gan-Ichino's paper; the only difference is that, we use the condition on $r$ to avoid the zeros or poles of the $L$-function $L\left(s,\tau\right)$, whereas they use the tempered condition on the representations.
\end{proof}
\begin{remark}\label{Arthur-type.L-function.Stripe}
If we assume that $\tau$ is an irreducble unitary representation of $GL_k(E)$ of Arthur type, say, corresponding to a local $A$-parameter 
\[
    \psi_\tau=\sum_i\rho_i\boxtimes S_{a_i}\boxtimes S_{b_i},
\]
where $\rho_i$ is an irreducible representation of the Weil group $W_E$ (with bounded image), $S_{a_i}$ is the $a_i$-dimensional irreducible representation of the Weil-Deligne $SL_2$, and $S_{b_i}$ is the $b_i$-dimensional irreducible representation of the Arthur $SL_2$. Then we have
\begin{align*}
    L(s,\tau)&=L(s,\psi_\tau)\\
    &=\prod_i\prod_{j=1}^{b_i}L\left(s+\frac{a_i-1}{2}+\frac{b_i+1}{2}-j,\rho_i\right).
\end{align*}
It is easy to see that in this case, we can take a positive number $N_\tau$ satisfying our requirements, such that
\[
    N_\tau<k.
\]
Moreover, from the above expression one can also see that the $L$-function $L(s,\tau)$ has a pole at some real number if and only if there exists some $i$, such that $\rho_i=\mathbbm{1}$.
\end{remark}

Similarly, we can construct a backward $G\times H$ equivariant map
\[
  \calT'_s:\omega\otimes\Ind_P^{G}\left(\tau_s^c\chi_W^c\boxtimes\pi_0^\vee\right)\lra\Ind_Q^{H}\left(\tau_s\chi_V\boxtimes\sigma_0\right)
\]
and show some similar statements as the previous proposition. In summary, we have:
\begin{corollary}\label{thetaNind}
Let $N_\tau >0$ be a positive real number such that $L(s, \tau)$ has no zeros or poles outside the stripe
\begin{equation*}
-N_\tau<\Re(s)<N_\tau.
\end{equation*}
Assume that $r>N_\tau+n$. Then the theta lift defines a bijection
\begin{equation*}JH\left(\Ind_P^{G}\left(\tau\chi_W\boxtimes\pi_0\right)\right)\longleftrightarrow JH\left(\Ind_Q^{H}\left(\tau\chi_V\boxtimes\sigma_0\right)\right).
\end{equation*}
Here, we use $JH$ to denote ``The multi-set of irreducible constituents''.
\end{corollary}
\begin{proof}
Let
\[
    \calT_0:\omega\otimes\Ind_Q^{H}\left(\tau^c\chi_V^c\boxtimes\sigma_0^\vee\right)\lra\Ind_P^{G}\left(\tau\chi_W\boxtimes\pi_0\right)
\]
be the $G\times H$-equivariant map we constructed in Proposition \ref{Non.Vanish.Equi.Map}, and 
\[
    \calT_0^*:\omega\otimes\Ind_P^{G}\left(\tau^c\chi_W^c\boxtimes\pi_0^\vee\right)\lra\Ind_Q^{H}\left(\tau\chi_V\boxtimes\sigma_0\right)
\]
be the adjoint map associated to $\calT_0$, defined by
\[
    \left\langle\calT_0^*\left(\varphi\otimes\varPhi^\vee\right),\varPsi^\vee\right\rangle=\left\langle\calT_0\left(\varphi\otimes\varPsi^\vee\right),\varPhi^\vee\right\rangle,
\]
where $\varphi\in\scrS$, 
\[
    \varPhi^\vee\in\Ind_P^{G}\left(\tau^c\chi_W^c\boxtimes\pi_0^\vee\right)\simeq\left(\Ind_P^{G}\left(\tau\chi_W\boxtimes\pi_0\right)\right)^\vee,
\]
and 
\[
    \varPsi^\vee\in\Ind_Q^{H}\left(\tau^c\chi_V^c\boxtimes\sigma_0^\vee\right)\simeq\left(\Ind_Q^{H}\left(\tau\chi_V\boxtimes\sigma_0\right)\right)^\vee.
\]
By Proposition \ref{Non.Vanish.Equi.Map}, the map $\calT_0^*$ is surjective. Hence for any irreducible constituent $\sigma$ of $\Ind_Q^{H}\left(\tau\chi_V\boxtimes\sigma_0\right)$, we deduce that its theta lift $\pi$ to the group $G$ is an irreducible constituent of $\Ind_P^{G}\left(\tau\chi_W\boxtimes\pi_0\right)$. Moreover, if we denote by
\[
    m_Q(\sigma)=\dim\Hom_H\left(\sigma,\Ind_Q^{H}\left(\tau\chi_V\boxtimes\sigma_0\right)\right),
\]
and
\[
    m_P(\pi)=\dim\Hom_G\left(\pi,\Ind_P^{G}\left(\tau\chi_W\boxtimes\pi_0\right)\right),
\]
then the surjectivity also implies that
\[
    m_Q(\sigma)\leq m_P(\pi).
\]
Similarly, use the backward $G\times H$-equivariant map
\[
  \calT'_0:\omega\otimes\Ind_P^{G}\left(\tau^c\chi_W^c\boxtimes\pi_0^\vee\right)\lra\Ind_Q^{H}\left(\tau\chi_V\boxtimes\sigma_0\right),
\]
we can prove a reverse inequality. This completes the proof.
\end{proof}

\section{Preparations}
In this section we recall some basic facts we shall need in later proofs.
\subsection{Explicit construction of Arthur packets \`a la M{\oe}glin}
We first recall some results due to M{\oe}glin in \cite{MR2767522}. Readers may also consult the paper \cite{MR3679701}. We emphasize that our proof of Theorem \ref{Main.Theorem.Packets} relies on these results. \\

In this subsection, we temporarily let $G=G(W)$ be either an even orthogonal, or a symplectic, or an unitary group; i.e. $W$ is an orthogonal, or symplectic, or Hermitian space, and $G$ is the isometry group of $W$. Assume that $G$ is quasi-split. We fix a Whittaker datum $\scrW$ of $G$. Let $\psi$ be a local $A$-parameter for $G$. Recall that a local $A$-parameter can be regarded as a formal sum
\begin{equation*}
\psi=\sum_i \rho_i\boxtimes S_{a_i}\boxtimes S_{b_i}
\end{equation*}
satisfies certain properties; where $\rho_i$ is an irreducible representation of the Weil group $W_E$ (with bounded image), $S_{a_i}$ is the $a_i$-dimensional irreducible representation of the Weil-Deligne $SL_2$, and $S_{b_i}$ is the $b_i$-dimensional irreducible representation of the Arthur $SL_2$.
\begin{definition}
\begin{enumerate}
\item $\psi$ is said to be of good parity, if for every $i$, $\rho_i\boxtimes S_{a_i}\boxtimes S_{b_i}$ is (conjugate) self-dual and of the same parity with $\psi$;
\item $\psi$ is said to have discrete diagonal restriction (``DDR'' for short), if the pull-back of $\psi$ along the diagonal map
\begin{align*}
\Delta:W_E\times SL_2(\CC)&\lra W_E\times SL_2(\CC)\times SL_2(\CC),\\
(w,x)&\longmapsto (w,x,x),
\end{align*}
denoted by $\psi_d$, is a discrete $L$-parameter for $G$;
\item $\psi$ is said to be elementary, if it has DDR, and for every $i$, either $a_i$ or $b_i$ is $1$.
\end{enumerate}
\end{definition}

For $\psi\in\Psi(G)$ and $\eta\in\wh{\overline{\calS_\psi}}$, M{\oe}glin constructed a finite-length semi-simple smooth representation $\pi_M(\psi,\eta)$ of $G$. She showed that $\Pi_\psi^A(G)$ consists of $\pi_M(\psi,\eta)$ for all $\eta\in\wh{\overline{\calS_\psi}}$, and by studying their properties she was able to conclude that (cf. \cite{MR3679701} Theorem 8.9):
\begin{theorem}\label{Moeglin.Multi.Free}
$\Pi_\psi^A(G)$ is multiplicity-free. Moreover, there exist a character $\eta_\psi^{M/A}\in\wh{\overline{\calS_\psi}}$, such that for any $\eta\in\wh{\overline{\calS_\psi}}$,
\begin{equation*}
\pi_M(\psi,\eta)=\pi_A(\psi,\eta\cdot\eta_\psi^{M/A}).
\end{equation*}
Here $\pi_A(\psi,\eta)$ is Arthur's parametrization.
\end{theorem}
\begin{remark}
Notice that both M{\oe}glin's and Arthur's parametrizations of local $A$-packets depend on the choice of Whittaker datum!
\end{remark}

There are several steps in M{\oe}glin's explicit construction. The starting point is the elementary $A$-parameters, whose associated $A$-packets are constructed by using some generalized Aubert involutions. Two facts about the local $A$-packets of elementary type are remarkable to us. The first one is about supercuspidal representations. Recall we have the following parametrization of supercuspidal representations of $G$, which is also due to M{\oe}glin (cf. \cite{MR2767522} Th\'eor\`eme 1.5.1, or \cite{MR3713922} Theorem 3.3). 
\begin{theorem}\label{sc.criterion}
Under the local Langlands correspondence for $G$ established by Arthur/ Mok, the irreducible supercuspidal representations of $G$ are parametrized by $\phi\in\Phi_2(G)$ and $\eta\in\widehat{\overline{\calS_\phi}}$, satisfying the following properties
\begin{enumerate}
\item \textbf{(Chain condition)} if $\rho\boxtimes S_a\subset\phi$, then $\rho\boxtimes S_{a-2}\subset\phi$ as long as $a-2>0$;
\item \textbf{(Alternating condition)} if $\rho\boxtimes S_a, \rho\boxtimes S_{a-2} \subset\phi$, then 
\begin{equation*}
\eta(x_{\rho,a})\cdot \eta(x_{\rho,a-2})=-1,
\end{equation*}
where $x_{\rho,a}$ is the element in $\calS_\phi$ which corresponds to $\rho\boxtimes S_a$;
\item \textbf{(Initial condition)} if $\rho\boxtimes S_2\subset\phi$, then $\eta(x_{\rho,2})=-1$.
\end{enumerate}
\end{theorem}
According to M{\oe}glin's construction, we have (cf. \cite{MR3679701} Definition 6.3)
\begin{lemma}
Assume that $\psi\in\Psi(G)$ is elementary, and $\eta\in\wh{\overline{\calS_\psi}}$. If $(\psi_d,\eta)$ is a pair parametrizing a supercuspidal representation (i.e. satisfies conditions listed in the previous theorem, here we identify $\calS_\psi$ and $\calS_{\psi_d}$ in the obvious way), then 
\begin{equation*}
\pi_M(\psi,\eta)=\pi_A(\psi_d,\eta)
\end{equation*}
is supercuspidal. 
\end{lemma}
The second remarkable fact is (cf. \cite{MR3679701} Theorem 6.1)
\begin{lemma}\label{surjectivity.ell.type}
Assume that $\psi\in\Psi(G)$ is elementary. Then the map
\begin{equation*}
\calJ_\scrW^A:\Pi_\psi^A(G)\lra\wh{\overline{\calS_\psi}}
\end{equation*}
is a bijection.
\end{lemma}

The next step is to construct $A$-packets in the DDR case based on the elementary case by taking certain socles. Having the DDR case at hand, the $A$-packets in the good parity case can be constructed based on the DDR case, by taking a sequence of partial Jacquet modules. To be more precise, we briefly describe this procedure.\\

Let $\psi$ be a local $A$-parameter for $G$ which is of good parity, we write
\[
    \psi=\sum_i\psi_i,
\]
where the $\psi_i=\rho_i\boxtimes S_{a_i}\boxtimes S_{b_i}$ occuring in this sum are not necessarily distinct one from each other. We denote by $I_\psi$ the index set of this summation. For each $i\in I_\psi$, put  
\[
    \zeta_i=\begin{cases}
                +1\quad&\textit{if }a_i\geq b_i;\\
                -1\quad&\textit{if }a_i< b_i.
            \end{cases}
\]
Let $\rho$ be an irreducible (conjugate) self-dual representation of $W_E$, we set
\[
    I_{\psi,\rho}=\left\{i\in I_\psi~\big|~\rho_i\simeq\rho\right\}.
\]
\begin{definition}
\begin{enumerate}
\item A total order $>_{\psi,\rho}$ on $I_{\psi,\rho}$ is said to be an admissible order on $I_{\psi,\rho}$ if it satisfies the following condition:
    \[
        \quad\bullet~\textit{For $i,j\in I_{\psi,\rho}$, if $a_i+b_i>a_j+b_j$, $|a_i-b_i|>|a_j-b_j|$, and $\zeta_i=\zeta_j$, then $i>_{\psi,\rho}j$.}
    \]
    A partial order $>_\psi$ on $I_\psi$ is said to be an admissible order on $I_{\psi}$ if its restriction to $I_{\psi,\rho}$ is an admissible order on $I_{\psi,\rho}$ for any irreducible (conjugate) self-dual representation $\rho$ of $W_E$.
\item Let $W^{\gg}$ be a space in the Witt tower containing $W$, and $G_{\gg}=G(W^{\gg})$. We say that a local $A$-parameter $\psi_{\gg}$ for $G_{\gg}$ dominates $\psi$ with respect to the admissible order $>_\psi$, if 
    \[
        \psi_{\gg}=\sum_i\rho_i\boxtimes S_{a'_i}\boxtimes S_{b'_i},
    \]
where the summation runs over $i\in I_\psi$, such that for each $i\in I_\psi$, we have 
    \[
        \left(a'_i,b'_i\right)=\begin{cases}
                                    \left(a_i+2t_i,b_i\right)\quad&\textit{if }\zeta_i=+1;\\
                                    \left(a_i,b_i+2t_i\right)\quad&\textit{if }\zeta_i=-1
                               \end{cases}
    \]
for some non-negative integer $t_i$, and $>_\psi$ is also an admissible order for $\psi_{\gg}$.
\end{enumerate}
\end{definition}
\begin{remark}
\begin{enumerate}
\item For any $\psi\in\Psi(G)$ of good parity, there exists at least one admissible order $>_\psi$ on $I_\psi$. Moreover, if $\psi$ has DDR, then there is an admissible order $>_\psi$ on $I_\psi$, such that for any $i,j\in I_{\psi,\rho}$, we have $i>_\psi j$ if and only if $a_i+b_i>a_j+b_j$. Admissible orders satisfying this condition are called ``natural order''.
\item For any $\psi\in\Psi(G)$ of good parity, let $>_\psi$ be an admissible order on $I_\psi$. Then there exists a group $G_{\gg}$, together with a local $A$-parameter $\psi_{\gg}$, such that 
    \begin{itemize}
        \item $\psi_{\gg}$ has DDR;
        \item $\psi_{\gg}$ dominates $\psi$ with respect to the admissible order $>_\psi$;
        \item $>_\psi$ is a natural order for $\psi_{\gg}$.
    \end{itemize}
\end{enumerate}
\end{remark}
Now, given $\psi\in\Psi(G)$ of good parity, let $>_\psi$ be an admissible order on $I_\psi$, and $\psi_{\gg}$ a local $A$-parameter has DDR for some group $G_{\gg}$, such that $\psi_{\gg}$ dominates $\psi$ with respect to the admissible order $>_\psi$. We have the following deep theorem due to M{\oe}glin (cf. \cite{MR2760663} Proposition 2.8.1, or \cite{MR3679701} Proposition 8.5):
\begin{theorem}\label{Construction.bp.From.DDR}
\begin{enumerate}
\item For $\pi_{\gg}\in\Pi_{\psi_{\gg}}^A(G_{\gg})$, let
    \[
        \pi=\circ_{i\in I_\psi}Jac_{X_i^{\gg}}\pi_{\gg},
    \]
where the composition is taken in the decreasing order with respect to $>_\psi$, and the symbol $Jac_{X_i^{\gg}}$ means applying partial Jacquet module $Jac_{\rho_i|\cdot|^x}$ (cf. \cite{MR3679701} page 897) consecutively for $x$ ranging over the generalized segment
    \[
        X_i^{\gg}=\left(\begin{array}{ccc}
                            \frac{a'_i-b'_i}{2} & \cdots & \frac{a_i-b_i}{2}+\zeta_i \\
                            \vdots & {} & \vdots \\
                            \left(\frac{a'_i+b'_i}{2}-1\right)\zeta_i & \cdots & \frac{a_i+b_i}{2}\zeta_i
                        \end{array}\right)
    \]
from top to bottom and from left to right, i.e.
    \[
        Jac_{X_i^{\gg}}=\left(Jac_{\rho_i|\cdot|^{\frac{a_i+b_i}{2}\zeta_i}}\circ\cdots\circ Jac_{\rho_i|\cdot|^{\frac{a_i-b_i}{2}+\zeta_i}}\right)\circ\cdots\circ\left(Jac_{\rho_i|\cdot|^{\frac{a'_i+b'_i}{2}\zeta_i-\zeta_i}}\circ\cdots\circ Jac_{\rho_i|\cdot|^{\frac{a'_i-b'_i}{2}}}\right).
    \] 
Then $\pi$ is either zero or irreducible.
\item As a set, the local $A$-packet of $G$ associated to the $A$-parameter $\psi$ is
    \[
        \Pi_\psi^A(G)=\left\{\pi=\circ_{i\in I_\psi}Jac_{X_i^{\gg}}\pi_{\gg}~\big|~\pi_{\gg}\in\Pi_{\psi_{\gg}}^A(G_{\gg})\right\}\Big\backslash\left\{0\right\} .
    \]
    The definition of $\Pi_\psi^A(G)$ is indeed independent of the choice of $G_{\gg}$ and $\psi_{\gg}$, as well as the admissible order $>_\psi$.
\end{enumerate}
\end{theorem}
Finally the $A$-packet associated to a general $\psi$ can be constructed from the good parity case by using parabolic inductions. For this part, readers may also consult \cite{MR3135650} Proposition 2.4.3, or \cite{MR3338302} Proposition 3.4.4. 

\subsection{Globalizations}
Since in our later proofs we heavily use global methods, in this subsection, we collect some results on globalizing local data.
\begin{lemma}\label{Golobalization-Origin}
Let $(\dot{F},\dot{E})$ be a pair of number fields, and $u$ is a place of $\dot{F}$, such that $(\dot{F}_u,\dot{E}_u)\simeq (F,E)$. We require that $\dot{F}$ has enough real places in the case $\dot E=\dot F$. Let $S$ be a finite set of non-Archimedean places of $\dot{F}$ not containing $u$. Let $m$ be a positive integer. Fix $\kappa=\pm 1$. For each place $v\in S\cup\{u\}$, let $\phi_v$ be a discrete $m$-dimensional (conjugate) self-dual representation of $L_{\dot E_v}$ with parity $\kappa$. Then there exists an irreducible cuspidal automorphic representation $\dot\phi$ of $GL_m(\AAA_{\dot E})$, such that the following conditions hold:
\begin{enumerate}
\item $\dot\phi$ is (conjugate) self-dual with parity $\kappa$;
\item for each place $v\in S\cup\{u\}$, $\dot\phi_v=\phi_v$.
\end{enumerate}
\end{lemma}
\begin{proof}
In Case $O$, this is follows from \cite{MR3135650} Lemma 6.2.2. In Case $U$, this follows from \cite{MR3004076} Theorem 5.13. (In Case $U$, readers may also consult \cite{MR3573972} Section 6.4; if $\dot F$ has enough real places different from $u$, then this also follows from \cite{MR3338302} Lemma 7.2.3.)
\end{proof}
\begin{remark}\label{Globalization-Origin'}
Indeed, from the proof of this lemma, we can also allow $S$ to contain one Archimedean place.
\end{remark}
As an application of this lemma, we deduce:
\begin{corollary}\label{globalize.quasi-split}
Assume that $\psi$ is a local $A$-parameter of good parity for $G$. Then, for any irreducible unitary representation $\pi$ in $\Pi_{\psi}^A(G)$, there exist a tuple of data $(\dot{F},\dot{E},\dot{G},\dot{\pi},\dot{\psi},u)$, where 
\begin{itemize}
\item $\dot{F}$ is a number field, and $\dot{E}$ is either $\dot F$ itself or a quadratic extension of $\dot F$, according to the case at hand;
\item $\dot{G}$ a quasi-split even orthogonal or symplectic or unitary group over $\dot{F}$, according to the group $G$; in the case that $\dot G$ is an unitary group, $\dot E$ is the splitting field of $\dot G$;
\item $\dot{\pi}$ is an automorphic representation occuring in the automorphic discrete spectrum of $\dot{G}$, and $\dot{\psi}$ is the elliptic $A$-parameter associated to $\dot{\pi}$;
\item $u$ is a finite place of $\dot{F}$.
\end{itemize}
such that the following condition holds
\begin{equation*}
(\dot{F}_u,\dot{E}_u,\dot{G}_u,\dot{\pi}_u,\dot{\psi}_u)\simeq(F,E,G,\pi,\psi).
\end{equation*}
\end{corollary}
\begin{proof}
First we choose a pair of number fields $(\dot{F},\dot{E})$, together with two places $u$ and $w$ of $\dot F$, satisfy the following properties:
\begin{enumerate}
\item $(\dot{F}_{u},\dot{E}_{u})\simeq(F,E)$;
\item if $G$ is an even orthogonal or symplectic group, then $\dot F$ has enough real places, and $\dot F_w$ is a finite extension of $\QQ_2$ with a sufficiently big residue field (this condition guarantees we will have enough irreducible orthogonal representations of $W_{\dot F_w}$; see \cite{CZ2020AMFPIF} Appendix B);
\item[$(2')$] if $G$ is an unitary group, then $\dot F_w$ is a finite extension of $\QQ_p$ for some $p\neq2$ with a sufficiently big residue field, and $\dot E_w$ is a ramified quadratic field extension of $\dot F_w$ (this condition guarantees we will have enough irreducible conjugate self-dual representations of $W_{\dot E_w}$, with any given parity; see \cite{CZ2020AMFPIF} Appendix C).
\end{enumerate}
Let $S=\{w\}$. We write
\[
    \psi=\sum_i\phi_i\boxtimes S_{d_i}
\]
for some (not necessarily distinct) $n_i$-dimensional irreducible (conjugate) self-dual representations of $L_E\times SL_2$. For each $i$, we pick up an irreducible (conjugate) self-dual representations $\rho_{w,i}$ of $W_{\dot E_w}$, with the same dimension and parity of $\phi_i$. Apply the Lemma \ref{Golobalization-Origin}, we can globalize $\phi_i$ to an irreducible (conjugate) self-dual cuspidal representation $\dot\phi_i$ of $GL_{n_i}(\AAA_{\dot E})$ with the same parity as $\phi_i$, such that 
\begin{enumerate}
\item $(\dot\phi_i)_{u}=\phi_{i}$;
\item $(\dot\phi_i)_{w}=\rho_{w,i}\boxtimes S_1$.
\end{enumerate}
We require that when $i\neq i'$, $\rho_{w,i}\not\simeq\rho_{w,i'}$. Let
\[
    \dot\psi=\begin{cases}
                  \sum_i\left(\dot\phi_i\cdot\dot\omega\right)\boxtimes S_{d_i}&\textit{if $G$ is symplectic};\\
                  \sum_i\dot\phi_i\boxtimes S_{d_i}\quad&\textit{if $G$ is even orthogonal or unitary},
             \end{cases}
\]
where 
\[
    \dot\omega=\prod_i\dot\omega_i^{d_i},
\]
and $\dot\omega_i$ is the central character of $\dot\phi_i$. If $G$ is a symplectic group, then $\dot\psi$ is already an elliptic $A$-parameter for some symplectic group $\dot G$ over $\dot F$. If $G$ is an even orthogonal or unitary group, since $\dot\psi$ is (conjugate) self-dual with the same parity as $\psi$, at each place $v$ of $\dot F$ we can pick up a $c$-Hermitian space $V_v$, such that $G(V_v)$ is quasi-split, and $\dot\psi_v$ is a local $A$-parameter for the group $G(V_v)$. It follows from the local-global principle for orthogonal or unitary groups that the collection $\{V_v\}_v$ indeed form a $c$-Hermitian space $\dot V$ over $\dot E$. In this case we put $\dot G=G(\dot V)$. Then $\dot G$ is quasi-split, and $\dot\psi$ is an elliptic $A$-parameter for $\dot G$. Notice that 
\[
    \dot\psi_w=\sum_i\rho_{w,i}\boxtimes S_1\boxtimes S_{d_i}
\]
is elementary, and the localization map
\[
    \iota_w:\calS_{\dot\psi}\lra\calS_{\dot\psi_w}
\]
is an isomorphism.\\

Next we define an irreducible representation $\dot\pi$ as follows:
\begin{itemize}
\item at a place $v\notin \{u,w\}$, if $\dot G_v$ and $\dot\psi_v$ are both unramified, then $\dot\pi_v$ is the unramified representation of $\dot G_v$ with $L$-parameter $\phi_{\dot\psi_v}$; otherwise, let $\dot\pi_v$ be an arbitrarily given representation of $\dot G_v$ lying in the $A$-packet $\Pi_{\dot\psi_v}^A(\dot G_v)$;
\item at the place $u$, $\dot\pi_u=\pi$;
\item at the place $w$, $\dot\pi_w=\pi_A\left(\dot\psi_{w},\eta_{w}\right)$, where $\eta_{w}$ is the character of $\calS_{\dot{\psi}_{w}}$, determined by the formula
    \[
        \prod_v\eta_v=\epsilon_{\dot\psi},
    \]
    where $\eta_v=\calJ_{\scrW}^A(\dot\pi_v)$, and $\epsilon_{\dot\psi}$ is the canonical sign character associated to $\dot\psi$. It follows from Lemma \ref{surjectivity.ell.type} that $\dot\pi_w\neq0$.
\end{itemize}
Then, according to the Arthur's multiplicity formula for $\dot G$, $\dot\pi$ is an irreducible subrepresentation of $L^2_{\dot\psi}(\dot G)$. One can easily check that the tuple of data $(\dot{F},\dot{E},\dot{G},\dot{\pi},\dot{\psi},u)$ satisfies all our requirements.
\end{proof}

\section{Comparison of packets from different level as sets}
From now we let $F$ be a non-Archimedean local field of characteristic $0$. In this section, we shall prove that as a set, the definition of the $\theta$-packets (see Section \ref{Def.Theta-Pack}) is independent of the choice of $H=H(W_{(r)})$. The word ``level'' in the title of this section refers to the integer $r$, i.e. the ``level'' in the Witt tower.
\subsection{First properties}
Before we start to prove, we develop some first properties of $\Pi_\psi^\theta(G)$.
\begin{lemma}
$\Pi_\psi^\theta(G)$ is multiplicity-free. Hence we can regard $\Pi_\psi^\theta(G)$ as a subset of $\Irr_{unit}(G)$.
\end{lemma}
\begin{proof}
This is simply follows from the Howe duality and that $\Pi_{\theta(\psi)}^{A}(H)$ is multiplicity-free.
\end{proof}
Consider the theta lift between $(G,H)$ for all pure inner form $G$ of $G^*$ simultaneously, we get
\begin{proposition}\label{Count.Size.Theta.Packets}
Suppose $\psi$ is a local $A$-parameter of good parity for $G^*$. Then the theta lift provide us a bijection of sets
\[
    \theta:\Pi_{\theta(\psi)}^A(H)\lra\bigsqcup\Pi_\psi^\theta(G),
\]
where the disjoint union on the RHS of the arrow runs over all pure inner forms $G$ of $G^*$.
\end{proposition}
\begin{proof}
First we construct this map, i.e. we need to show that, for any irreducible representation $\sigma\in\Pi_{\theta(\psi)}^A(H)$, there is a pure inner form $G$ of $G^*$, such that the theta lift of $\sigma$ to $G$ is non-zero. Applying Lemma \ref{Golobalization-Origin}, we globalize the tuple of local data $(F,E,V^+,W,\psi,\psi_F,\chi_V,\chi_W)$ to a tuple global data
\[
    (\dot F,\dot E,\dot V^+,\dot W,\dot \psi,\psi_{\dot F},\chi_{\dot V},\chi_{\dot W}),
\]
where
\begin{itemize}
\item $\dot{F}$ is a number field, and $\dot{E}$ is either $\dot F$ itself or a quadratic extension of $\dot F$, according the cases;
\item $\dot{V}^+$ is a $c$-Hermitian space over $\dot{E}$ so that $\dot G^*=G(\dot V^+)$ is quasi-split;
\item $\dot{W}$ is the split $c$-skew-Hermitian space over $\dot{E}$, in particular $\dot H=H(\dot W)$ is also quasi-split;
\item $\dot{\psi}$ is an elliptic $A$-parameter of $\dot G^*$;
\item $(\psi_{\dot F},\chi_{\dot V},\chi_{\dot W})$ is a tuple of auxiliary data allow us to define the theta lift between $(\dot G^*,\dot H)$;
\end{itemize}
together with two places $u$ and $w$ of $\dot F$, such that
\begin{enumerate}
\item $(\dot F_u,\dot E_u,\dot V^+_u,\dot W_u,\dot \psi_u,\psi_{\dot F_u},\chi_{\dot V_u},\chi_{\dot W_u})\simeq(F,E,V^+,W,\psi,\psi_F,\chi_V,\chi_W)$;
\item at the place $w$, $\dot\psi_w$ is elementary, and the localization map
    \[
        \iota_w:\calS_{\dot\psi}\lra\calS_{\dot\psi_w}
    \]
    is an isomorphism.
\end{enumerate}
For any irreducible unitary representation $\sigma\in\Pi_{\theta(\psi)}^A(H)$, similar to the proof of Corollary \ref{globalize.quasi-split}, by using the Arthur's multiplicity formula for $\dot H$, we may globalize $\sigma$ to an irreducible subrepresentation $\dot\sigma$ of $L^2_{\theta(\dot\psi)}(\dot H)$, where
\[
    \theta(\dot\psi)=\dot\psi\chi_{\dot W}^{-1}\chi_{\dot V}+\chi_{\dot V}\boxtimes S_{2r-2n+1}
\]
is an elliptic $A$-parameter for the group $\dot H$. It then follows from J-S. Li's work on low rank representations (Theorem \ref{J-S.Li.Low.rk.Global}) that $\dot\sigma$ is of rank $\dim V$, and so is $\sigma$. Indeed, J-S. Li's results also assert that there exists a pure inner form $\dot G$ of $\dot G^*$, together with an automorphic representation $\dot\pi$ of $\dot G$, such that $\dot\sigma$ is the theta lift of $\dot\pi$. We set
\[
    \sigma\longmapsto\dot\pi_u.
\]
This gives us the desired map. Notice that $\dot\pi_u$ is nothing but the theta lift of $\sigma$ to $\dot G_u$. Hence by the conservation relation \cite{MR3369906}, this map is well-defined and independent of the globalization. The injectivity of this map then follows from the Howe duality principle, and the surjectivity simply follows from the definition of the $\theta$-packets.
\end{proof}

Let $\widetilde{V}=V\oplus\calH^k$, where $\calH$ is the $c$-Hermitian hyperbolic plane. We can decompose $\widetilde{V}$ as
\begin{equation*}
\widetilde{V}=X+V+X^*,
\end{equation*}
where $X$ and $X^*$ are $k$-dimensional totally isotropic subspaces of $\widetilde{V}$ such that $X\oplus X^*\simeq\calH^k$ and orthogonal to $V$. Let $P$ be the maximal parabolic subgroup of $\wt{G}=G(\widetilde{V})$ stabilizing $X$ and $M$ be its Levi component stabilizing $X^*$, so that
\begin{equation*}
M\simeq GL(X)\times G.
\end{equation*}
For an irreducible unitary representation $\tau$ of $GL_k(E)$ of Arthur type corresponding to an $A$-parameter $\psi_\tau$, we consider the induced representation
\begin{equation*}
\Ind_P^{\wt G}\Big(\tau\chi_W\boxtimes\Pi_\psi^\theta(G)\Big).
\end{equation*}
\begin{lemma}\label{indNmultifree}
Assume that $r>\dim V+k$. Then 
\begin{equation*}
\Ind_P^{\wt G}\Big(\tau\chi_W\boxtimes\Pi_\psi^\theta(G)\Big)
\end{equation*}
is semi-simple and multiplicity-free as a representation of $\wt G$.
\end{lemma}
\begin{proof}
The semi-simplicity simply follows from the unitaricity. We now prove the multiplicity-freeness. Let $N_\tau >0$ be a positive number such that $L(s, \tau)$ has no zeros or poles outside the stripe
\begin{equation*}
-N_\tau<\Re(s)<N_\tau.
\end{equation*}
Since $\tau$ is of Arthur type, we may take $N_\tau<k$. Hence all the requirements in Corollary \ref{thetaNind} are satisfied. Let $\widetilde{W}=W\oplus\calH'^k$, where $\calH'$ is the $c$-skew-Hermitian hyperbolic plane. We can decompose $\widetilde{W}$ as
\begin{equation*}
\widetilde{W}=Y+W+Y^*,
\end{equation*}
where $Y$ and $Y^*$ are $k$-dimensional totally isotropic subspaces of $\widetilde{W}$ such that $Y\oplus Y^*\simeq\calH'^k$ and orthogonal to $W$. Let $Q$ be the maximal parabolic subgroup of $\wt H=H(\wt W)$ stabilizing $Y$ and $N$ be its Levi component stabilizing $Y^*$, so that
\begin{equation*}
N\simeq GL(Y)\times H.
\end{equation*}
By Corollary \ref{thetaNind}, for each $\pi\in\Pi_\psi^\theta(G)$, the theta lift between $\wt G\times\wt H$ defines a bijection
\begin{equation*}
JH\left(\Ind_P^{\wt G}\left(\tau\chi_W\boxtimes\pi\right)\right)\longleftrightarrow JH\left(\Ind_Q^{\wt H}\left(\tau\chi_V\boxtimes\sigma\right)\right),
\end{equation*}
where $\sigma$ is the (small) theta lift of $\pi$ to $H$. By the construction, $\sigma\in\Pi_{\theta(\psi)}^A(H)$. Add up all $\pi\in\Pi_\psi^\theta(G)$ together, we obtain an injection 
\[
    JH\left(\Ind_P^{\wt G}\left(\tau\chi_W\boxtimes\Pi_\psi^\theta(G)\right)\right)\longrightarrow JH\left(\Ind_Q^{\wt H}\left(\tau\chi_V\boxtimes\Pi_{\theta(\psi)}^A(H)\right)\right).
\]
Hence in order to show the multiplicity-freeness of the LHS, it is sufficient to show the RHS is multiplicity-free. Let 
\begin{equation*}
\wt\psi=\psi_\tau\chi_W+\psi+\left(\psi_\tau\chi_W\right)^\vee,
\end{equation*}
and
\begin{equation*}
\theta(\wt\psi)=\psi_\tau\chi_V+\theta(\psi)+(\psi_\tau\chi_V)^\vee.
\end{equation*}
Then according to \cite{MR3135650} Proposition 2.4.3 (or \cite{MR3338302} Proposition 3.4.4), we have 
\begin{equation*}
\Ind_Q^{\wt H}\Big(\tau\chi_V\boxtimes\Pi_{\theta(\psi)}^A(H)\Big)=\Pi_{\theta(\wt\psi)}^A(\wt H)
\end{equation*}
as representations of $\wt H$. By Theorem \ref{Moeglin.Multi.Free}, as a representation of $\wt H$, $\Pi_{\theta(\wt\psi)}^A(\wt H)$ is multiplicity-free. Therefore $\Ind_P^{\wt G}\Big(\tau\chi_W\boxtimes\Pi_\psi^\theta(G)\Big)$ is also multiplicity-free.
\end{proof}
Under the hypothesis of Lemma \ref{indNmultifree}, we deduce
\begin{corollary}\label{Ind.Relation}
Suppose $\psi$ is a local $A$-parameter of good parity for $G^*$, and all requirements in Lemma \ref{indNmultifree} are satisfied. We retain the notations in the proof of Lemma \ref{indNmultifree}. Then 
\[
    \Pi_{\wt\psi}^{\theta,r+k}(\wt G)=\Ind_P^{\wt G}\Big(\tau\chi_W\boxtimes\Pi_\psi^{\theta,r}(G)\Big)
\]
as representations of $\wt G$. In particular, Proposition \ref{Count.Size.Theta.Packets} also holds for general $A$-parameters.
\end{corollary}
\begin{proof}
This is a simple combination of the Corollary \ref{thetaNind}, Proposition \ref{Count.Size.Theta.Packets} and the proof of the Lemma \ref{indNmultifree}.
\end{proof}

Now we specialize to the case that $G=G^*$ is quasi-split.
\begin{proposition}\label{Containment.A.in.Theta}
Suppose $\psi$ is a local $A$-parameter of good parity for $G^*$. Then as subsets of $\Irr_{unit}(G^*)$, we have
\begin{equation*}
\Pi_\psi^A(G^*)\subset\Pi_\psi^\theta(G^*).
\end{equation*}
\end{proposition}
\begin{proof}
First we apply Lemma \ref{globalize.quasi-split} to the irreducible unitary representation $\pi\in\Pi_\psi^A(G^*)$. We obtain a tuple of global data $(\dot F,\dot E,\dot G,\dot\pi,\dot\psi,u)$. We also globalize the tuple of local auxiliary data $(\psi_F,\chi_V,\chi_W)$ in the definition of the local theta lift to a global tuple $(\psi_{\dot F},\chi_{\dot V},\chi_{\dot W})$. Let $\dot W$ be the split $c$-skew-Hermitian space over $\dot E$ with the same dimension as $W$. Put $\dot H=H(\dot W)$. Consider the abstract theta lift $\dot\sigma=\theta^{abs}(\dot\pi)$ of $\dot\pi$ to $\dot H$. Since $m_{disc}(\dot\pi)\geq 1$, we deduce from J-S. Li's inequality that $m_{disc}(\dot\sigma)\geq 1$. Hence, Arthur/ Mok has attached an elliptic $A$-parameter $\dot\psi_H$ to $\dot\sigma$.\\

We claim that $\dot\psi_H$ is of the form $\dot\psi_H=\theta(\dot\psi)$, where
\[
    \theta(\dot\psi)=\dot\psi\chi_{\dot W}^{-1}\chi_{\dot V}+\chi_{\dot V}\boxtimes S_{2r-2n+1}.
\]
In fact, for almost all place $v$ of $\dot F$, the $L$-parameter of $\dot\pi_v$ is $\phi_{\dot\psi_v}$; hence by the local theta correspondence for unramified representations, for almost all place $v$ of $\dot F$, the $L$-parameter of $\dot\sigma_v$ is
\begin{equation}
    \phi_{\dot\psi_v}\chi_{\dot W,v}^{-1}\chi_{\dot V,v}+\left(\bigoplus_{j=n-r}^{r-n}|\cdot|^j\right)\chi_{\dot V,v}.
\end{equation}
It then follows that $\dot\psi_H=\theta(\dot\psi)$. Therefore by the Arthur's multiplicity formula for $\dot H$, the localization of $\dot\sigma$ at the place $u$ lies in the local $A$-packet $\Pi_{\dot\psi_{H,u}}^A(\dot H_u)$, i.e. we have 
\[
    \sigma\in\Pi_{\theta(\psi)}^A(H),
\]
which is equivalent to say that $\pi\in\Pi_\psi^{\theta}(G^*)$.
\end{proof}
\begin{remark}
If we assume the same multiplicity-freeness result Theorem \ref{Moeglin.Multi.Free} hold for the Archimedean places, then we can show that this lemma is also true for Archimedean places.
\end{remark}

\subsection{A special class of parameters}\label{sc.A-parameter}
In this subsection we deal with a special class of local $A$-parameters. Let $\psi\in\Psi(G)$ be a local $A$-parameter. Suppose that the following hypothesis is satisfied:
\begin{hypo}\label{Hypothesis.Special.Class.A-parameter}
There is a tuple of data $(\dot{F},\dot{E},\dot{V}^+,\dot{\psi},u,w_1,w_2)$, where:
\begin{itemize}
\item $\dot{F}$ is a number field, and $\dot{E}$ is either $\dot F$ itself or a quadratic extension of $\dot F$, according to the cases at hand;
\item $\dot{V}^+$ is a $c$-Hermitian space over $\dot{E}$ so that $\dot G^*=G(\dot V^+)$ is quasi-split;
\item $\dot{\psi}$ is an elliptic $A$-parameter of $\dot G^*$;
\item $u,w_1,w_2$ are finite places of $\dot{F}$;
\end{itemize}
such that the following conditions hold:
\begin{enumerate}
\item $(\dot{F}_u,\dot{E}_u,\dot{V}^+_u,\dot{\psi}_u)\simeq(F,E,V^{+},\psi)$;
\item $\dot{\psi}_{w_1}$ is elementary, and there exists $\eta_{w_1}\in\widehat{\overline{\calS_{\dot{\psi}_{w_1}}}}$, such that $\pi_A(\dot{\psi}_{w_1},\eta_{w_1})$ is supercuspidal;
\item $\dot{\psi}_{w_2}$ is elementary, and the localization map
\begin{equation*}
\iota_{w_2}:\calS_{\dot{\psi}}\lra\calS_{\dot{\psi}_{w_2}}
\end{equation*}
is an isomorphism.
\end{enumerate}
\end{hypo}
Let $r$ be any integer greater than $\dim V$. Let $W=W_{(r)}$, and
\begin{equation*}
\theta(\psi)=\psi\chi_W^{-1}\chi_V+\chi_V\boxtimes S_{2r-2n+1}
\end{equation*}
be a local $A$-parameter for $H=H(W_{(r)})$.
\begin{proposition}\label{Antipasto}
Suppose that the local $A$-parameter $\psi$ satisfies the Hypothesis \ref{Hypothesis.Special.Class.A-parameter}. Let $\pi$ be an irreducible unitary representation of $G$, such that its small theta lift $\sigma=\theta(\pi)$ to $H$ lies in the local $A$-packet $\Pi_{\theta(\psi)}^{A}(H)$. Then there is a pair of data $(\dot G, \dot\pi)$, where
\begin{itemize}
\item $\dot G$ is a pure inner form of $\dot G^*$ over $\dot F$;
\item $\dot\pi$ is an automorphic representation occuring in the automorphic discrete spectrum of $\dot G$ with the elliptic $A$-parameter $\dot\psi$;
\end{itemize}
such that $\dot\pi_v=\pi$. Moreover, if $G=G^*$ is quasi-split, we can take $\dot G=\dot G^*$. Hence in this case, we have
\begin{equation*}
\pi\in\Pi_\psi^{A}(G^*).
\end{equation*}
\end{proposition}
\begin{proof}
We globalize the tuple of local auxiliary data $(\psi_F,\chi_V,\chi_W)$ in the definition of the local theta lift to a global tuple $(\psi_{\dot F},\chi_{\dot V},\chi_{\dot W})$. Let $\dot W=\dot W_{(r)}$ be the split $c$-skew-Hermitian space over $\dot E$ with the same dimension as $W=W_{(r)}$. Put $\dot H=H(\dot W)$. Let $S$ be a finite set of places of $\dot{F}$, including $u$, $w_1$, $w_2$, and all Archimedean places, such that for all $v\notin S$, the dual-pair $\dot G\times \dot H$, the auxiliary data $(\psi_{\dot F_v},\chi_{\dot V_v},\chi_{\dot W_v})$, and the local $A$-parameter $\dot\psi_v$ are all unramified. We construct an automorphic representation $\dot{\sigma}$ occuring in the automorphic discrete spectrum of $\dot H$ with elliptic $A$-parameter 
\begin{equation*}
\theta(\dot\psi)=\dot\psi\chi_{\dot W}^{-1}\chi_{\dot V}+\chi_{\dot V}\boxtimes S_{2r-2n+1}
\end{equation*}
as follows:
\begin{itemize}
\item at a place $v\notin S$, $\dot\sigma_v$ is the unramified representation of $\dot H_v$ with $L$-parameter $\phi_{\theta(\dot\psi_v)}$; then, by the theta lift for unramified representations, it is clear that $\dot\sigma_v=\theta(\pi_v)$, where $\pi_v$ is the unramified representation of $\dot G^*_v$ with $L$-parameter $\phi_{\dot\psi_v}$;
\item at a place $v\in S\backslash\{u,w_1,w_2\}$, let $\pi_v$ be an arbitrarily given representation of $\dot G^*_v$ lying in the $A$-packet $\Pi_{\dot\psi_v}^A(\dot G^*_v)$, and $\dot\sigma_v=\theta(\pi_v)$ is the theta lift of $\pi_v$ to the group $\dot H_v$; similar to the proof of Proposition \ref{Containment.A.in.Theta}, one can show that $\dot\sigma_v\in\Pi_{\theta(\dot\psi_v)}^A(\dot H_v)$ by using some global arguments; 
\item at the place $u$, $\dot\sigma_u=\sigma$, which lies in $\Pi_{\theta(\dot\psi_u)}^A(\dot H_u)$ by our assumptions;
\item at the place $w_1$, $\dot\sigma_{w_1}=\theta\left(\pi_A(\dot\psi_{w_1},\eta_{w_1})\right)$, which lies in $\Pi_{\theta(\dot\psi_{w_1})}^A(\dot H_{w_1})$ by Proposition \ref{Containment.A.in.Theta};
\item at the place $w_2$, $\dot\sigma_{w_2}=\pi_A\left(\theta(\dot\psi_{w_2}),\eta_{w_2}\right)$, where $\eta_{w_2}$ is the character of $\calS_{\dot{\psi}_{w_2}}$, determined by the formula
    \[
        \prod_v\eta_v=\epsilon_{\theta(\dot\psi)},
    \]
    where $\eta_v=\calJ_{\scrW'}^A(\dot\sigma_v)$, and $\epsilon_{\theta(\dot\psi)}$ is the canonical sign character associated to $\theta(\dot\psi)$. It follows from Lemma \ref{surjectivity.ell.type} that $\dot\sigma_{w_2}\neq0$.
\end{itemize}
Then, according to the Arthur's multiplicity formula for $\dot H$, $\dot\sigma$ is an irreducible subrepresentation of $L^2_{\theta(\dot\psi)}(\dot H)$. By J-S. Li's work on low rank representations (Theorem \ref{J-S.Li.Low.rk.Global}), there is an unique pure inner form $\dot G$ of $\dot G^*$ and an automorphic representation $\dot\pi$ of $\dot G$, such that
\begin{equation*}
\dot\sigma=\theta^{abs}(\dot\pi).
\end{equation*}
Also, by our construction, $\dot\pi_{w_1}=\pi_A(\dot\psi_{w_1},\eta_{w_1})$ is supercuspidal. This forces any automorphic realization of $\dot\pi$ to be cuspidal. Hence $\dot\pi$ lies in the automorphic discrete spectrum of $\dot G$, with elliptic $A$-parameter $\dot\psi$.\\
~\\
When $G=G^*$ is quasi-split, by the uniqueness of $\dot G$ (see Theorem \ref{J-S.Li.Low.rk.Global}) and the local-global principle for even orthogonal or unitary groups, it is easy to see that $\dot G=\dot G^*$. Hence by Arthur's multiplicity formula for $\dot G^*$, the localization of $\dot\pi$ at the place $u$ will lie in the local $A$-packet $\Pi_{\dot\psi_u}^{A}(\dot G^*_u)$, i.e.
\begin{equation*}
\pi\in\Pi_\psi^{A}(G^*).
\end{equation*}
This completes the proof.
\end{proof}
As a corollary of this proposition, we deduce
\begin{corollary}\label{Comparison.as.Sets.Special.case}
Suppose that the $A$-parameter $\psi$ satisfies Hypothesis \ref{Hypothesis.Special.Class.A-parameter}. Then
\begin{enumerate}
    \item as a set, the definition of the packet $\Pi_\psi^\theta(G)$ is indeed independent of the choice of $H=H(W_{(r)})$;
    \item if $G=G^*$ is quasi-split, then
    \[
        \Pi_\psi^\theta(G^*)=\Pi_\psi^A(G^*) 
    \]
    as sets.
\end{enumerate}
\end{corollary}
\begin{proof}
Firstly we globalize the tuple of local auxiliary data $(\psi_F,\chi_V,\chi_W)$ in the definition of the local theta lift to a global tuple $(\psi_{\dot F},\chi_{\dot V},\chi_{\dot W})$. For any positive integer $r'>\dim V$, let $\dot W=\dot W_{(r')}$ be the split $c$-skew-Hermitian space over $\dot E$ with the same dimension as $W=W_{(r')}$. Put $\dot H^{r'}=H\left(\dot W_{(r')}\right)$.\\

For an irreducible unitary representation $\pi\in\Pi_\psi^{\theta,r}(G)$, we have proved that we can globalize it to a cuspidal representation $\dot\pi$ of $\dot G$, with elliptic $A$-parameter $\dot\psi$. Similar to the proof of Proposition \ref{Containment.A.in.Theta}, one can easily see that the abstract theta lift $\dot\sigma^{r'}=\theta^{abs}(\dot\pi)$ of $\dot\pi$ to $\dot H^{r'}$ occurs in the automorphic discrete spectrum of $\dot H^{r'}$, with elliptic $A$-parameter
\[
    \theta^{r'}(\dot\psi)=\dot\psi\chi_{\dot W}^{-1}\chi_{\dot V}+\chi_{\dot V}\boxtimes S_{2r'-2n+1}.
\]
Consider the localizations of $\dot\pi$ and $\dot\sigma^{r'}$ at the place $u$, it follows that $\pi\in\Pi_\psi^{\theta,r'}(G)$, i.e. $\Pi_\psi^{\theta,r}(G)\subset\Pi_\psi^{\theta,r'}(G)$. Symmetrically, we also have the reverse containment. Hence
\[
    \Pi_\psi^{\theta,r}(G)=\Pi_\psi^{\theta,r'}(G).
\]
The second statement is a tautology of Proposition \ref{Antipasto}.
\end{proof}

\subsection{Sharp construction}
In this subsection we describe a key construction for our later proof. This construction allows us to ``embed'' any $A$-parameter of good parity as a ``sub $A$-parameter'' which satisfies the Hypothesis \ref{Hypothesis.Special.Class.A-parameter}.
\begin{lemma}\label{sharp}
For any local $A$-parameter $\psi\in\Psi(G^*)$ of good parity, there exists a tuple of data $(\dot{F},\dot{E},\dot{V}^{\sharp},\dot{\psi}^{\sharp},u,w_1,w_2)$, where:
\begin{itemize}
\item $\dot{F}$ is a number field, and $\dot{E}$ is either $\dot F$ itself or a quadratic extension of $\dot F$, according to the cases at hand;
\item $\dot{V}^{\sharp}$ is a $c$-Hermitian space over $\dot{E}$ so that $\dot G^{\sharp}=G(\dot V^{\sharp})$ is quasi-split;
\item $\dot{\psi}^{\sharp}$ is an elliptic $A$-parameter of $\dot G^{\sharp}$;
\item $u,w_1,w_2$ are finite places of $\dot{F}$;
\end{itemize}
such that the following conditions hold:
\begin{enumerate}
\item $(\dot{F}_u,\dot E_u)\simeq (F,E)$;
\item $\dim\dot{V}^{\sharp}$ is bounded by some constant which only depends on $\dim V$ but not on $\psi$;
\item $\dot{V}^{\sharp}_u=V^+\oplus\calH^k$ for some integer $k$, where $\calH$ is the $c$-Hermitian hyperplane;
\item $\dot{\psi}^{\sharp}_u=\psi_\tau+\psi+\left(\psi_\tau^c\right)^\vee$, where $\psi_\tau$ is a sum of tempered irreducible (conjugate) self-dual representations of the Weil group $W_E$ (regarded as representations of $L_E\times SL_2$ which is trivial on Weil-Deligne and Arthur $SL_2$) with the same parity as $\psi$;
\item $\dot{\psi}^{\sharp}_{w_1}$ is elementary, and there exists $\eta_{w_1}\in\widehat{\overline{\calS_{\dot{\psi}^{\sharp}_{w_1}}}}$, such that $\pi_A(\dot{\psi}^{\sharp}_{w_1},\eta_{w_1})$ is supercuspidal;
\item $\dot{\psi}^{\sharp}_{w_2}$ is elementary, and the localization map
\begin{equation*}
\iota_{w_2}:\calS_{\dot{\psi}^{\sharp}}\lra\calS_{\dot{\psi}^{\sharp}_{w_2}}
\end{equation*}
is an isomorphism.
\end{enumerate}
In short, the tuple of data $(\dot{F},\dot{E},\dot{V}^{\sharp},\dot{\psi}^{\sharp},u,w_1,w_2)$ satisfies Hypothesis \ref{Hypothesis.Special.Class.A-parameter}, and with $\dot{\psi}^{\sharp}_u$ related with $\psi$ as in the condition $4$.
\end{lemma}
\begin{proof}
First we choose a pair of number fields $(\dot{F},\dot{E})$, together with three places $u$, $w_1$, and $w_2$ of $\dot F$, satisfying the following properties:
\begin{enumerate}
\item $(\dot{F}_{u},\dot{E}_{u})\simeq(F,E)$, and $(\dot{F}_{w_1},\dot{E}_{w_1})\simeq(\dot{F}_{w_2},\dot{E}_{w_2})$;
\item if we are in Case $O$, then $\dot F$ has enough real places, and $\dot F_{w_i}$ is a finite extension of $\QQ_2$ with a sufficiently big residue field (this condition guarantees we will have enough irreducible orthogonal representations of $W_{\dot F_{w_i}}$; see \cite{CZ2020AMFPIF} Appendix B);
\item[$(2')$] if we are in Case $U$, then $\dot F_{w_i}$ is a finite extension of $\QQ_p$ for some $p\neq2$ with a sufficiently big residue field, and $\dot E_{w_i}$ is a ramified quadratic field extension of $\dot F_{w_i}$ (this condition guarantees we will have enough irreducible conjugate self-dual representations of $W_{\dot E_{w_i}}$, with any given parity; see \cite{CZ2020AMFPIF} Appendix C).
\end{enumerate}
We write the local $A$-parameter $\psi$ as a sum
\begin{equation*}
\psi=\sum_i\psi_i,
\end{equation*}
where each $\psi_i=\rho_i\boxtimes S_{a_i}\boxtimes S_{b_i}$ is a (conjugate) self-dual irreducible representation of $L_E\times SL_2$, with the same parity as $\psi$. For every $i$ and every positive integer $j$ such that $b_i-2j>0$, let $n_{i,j}=\dim\rho_i\cdot a_i\cdot (b_i-2j)$. We define an $n_{i,j}$-dimensional discrete (conjugate) self-dual representation $\phi_{i,j}$ of $L_E$ with the same parity as $\psi_i$ as follows:
\begin{itemize}
\item[-]Suppose we are in the Case $O$, we take:
    \begin{itemize}
        \item[$\bullet$] if $n_{i,j}$ is even, then we arbitrarily pick up an $n_{i,j}$-dimensional irreducible orthogonal representation of the Weil group of $E$, say $\rho_{i,j}$, and let $\phi_{i,j}=\rho_{i,j}\boxtimes S_1$ be a discrete orthogonal representation of $L_E$;
        \item[$\bullet$] if $n_{i,j}$ is odd, then we arbitrarily pick up an $(n_{i,j}-1)$-dimensional irreducible orthogonal representation of the Weil group of $E$, say $\rho_{i,j}$, and a quadratic character of the Weil group of $E$, say $\chi_{i,j}$, and let $\phi_{i,j}=\rho_{i,j}\boxtimes S_1+\chi_{i,j}\boxtimes S_1$ be a discrete orthogonal representation of $L_E$.
    \end{itemize}
\item[-]Suppose we are in the Case $U$, we take:
    \begin{itemize}
        \item[$\bullet$] for $1\leq\xi\leq n_{i,j}$, we arbitrarily pick up conjugate self-dual characters $\chi_{i,j}^{\xi}$ with the same parity as $\psi$, and we require that they are distinct one from each other; let 
            \[
                \phi_{i,j}=\sum_\xi\chi_{i,j}^\xi\boxtimes S_1
            \]
        be a discrete conjugate self-dual representation of $L_E$ with the same parity as $\psi$.
    \end{itemize}
\end{itemize}
Put $\psi_{i,j}=\phi_{i,j}\boxtimes S_1$ a representation of $L_E\times SL_2$. Let $\chi$ be a (conjugate) self-dual representation of the Weil group of $E$, with the same parity as $\psi$, and
\begin{equation*}
\psi^\sharp=\chi+\sum_i\left(\bigg(\sum_j\psi_{i,j}\bigg)+\psi_i+\bigg(\sum_j\left(\psi_{i,j}^c\right)^\vee\bigg)\right)+\left(\chi^c\right)^\vee,
\end{equation*}
which is also of good parity. Now we apply Lemma \ref{Golobalization-Origin} to globalize each $\psi_i$, $\psi_{i,j}$, $\left(\psi_{i,j}^c\right)^\vee$, $\chi$ and $\chi^{\vee}$.
\begin{itemize}
\item For each $\psi_i=\rho_i\boxtimes S_{a_i}\boxtimes S_{b_i}$, we globalize $\rho_i\boxtimes S_{a_i}$ to a (conjugate) self-dual cuspidal representation $\dot{\phi}_i$ of $GL_{n_i}(\dot{E})$, where $n_i=\dim \rho_i\cdot a_i$, which is of the same parity as $\rho_i\boxtimes S_{a_i}$, such that its localizations at places $w_1$ and $w_2$ are isomorphic, and supercuspidal. For $w\in\{w_1,w_2\}$, we use $\phi_{w,i}=\rho_{w,i}\boxtimes S_1$ to denote the $L$-parameter of this supercuspidal representation. We also require that when $i\neq i'$, $\rho_{w,i}\not\simeq\rho_{w,i'}$. Let $\dot{\psi}_i=\dot{\phi_i}\boxtimes S_{b_i}$.
\item For each $\psi_{i,j}=\phi_{i,j}\boxtimes S_1$, we globalize $\phi_{i,j}$ to a (conjugate) self-dual cuspidal representation $\dot{\phi}_{i,j}$ of $GL_{n_{i,j}}(\dot{E})$, which is of the same parity as $\phi_{i,j}$, such that the localizations at places $w_1$ and $w_2$ are isomorphic, and correspond to the $L$-parameter $\rho_{w,i}\boxtimes S_{b_i-2j}$. Let $\dot\psi_{i,j}=\dot{\phi}_{i,j}\boxtimes S_1$.
\item For each $\left(\psi_{i,j}^c\right)^\vee=\left(\phi_{i,j}^c\right)^\vee\boxtimes S_1$, we globalize $\left(\phi_{i,j}^c\right)^\vee$ to a (conjugate) self-dual cuspidal representation $\dot\phi_{i,j}^\dagger$ of $GL_{n_{i,j}}(\dot{E})$, which is of the same parity as $\left(\phi_{i,j}^c\right)^\vee$, such that the localizations at places $w_1$ and $w_2$ are isomorphic, and supercuspidal. We use $\phi_{w,i,j}^{\dagger}=\rho_{w,i,j}^{\dagger}\boxtimes S_1$ to denote the $L$-parameter of this supercuspidal representation. We also require that all $\rho_{w,i}$ and $\rho_{w,i,j}^\dagger$ are distinct one from each other. Let $\dot\psi_{i,j}^\dagger=\dot{\phi}_{i,j}^\dagger\boxtimes S_1$. Notice that $\dot\psi_{i,j}^\dagger\not\simeq\left(\dot\psi_{i,j}^c\right)^\vee$!
\item For $\chi$, we globalize it to a (conjugate) self-dual character $\dot{\chi}$ of $GL_1({\dot{E}})$, which is of the same parity as $\chi$, such that the localizations at places $w_1$ and $w_2$ are isomorphic. We use $\chi_w$ to denote the localization at $w_1$ (or $w_2$). We also require that $\chi_w$ is distinct from all $\rho_{w,i}$ and $\rho_{w,i,j}^\dagger$.
\item For $\left(\chi^c\right)^{\vee}$, we globalize it to a (conjugate) self-dual character $\dot{\chi}^\dagger$ of $GL_1({\dot{E}})$, which is of the same parity as $\left(\chi^c\right)^{\vee}$, such that the localizations at places $w_1$ and $w_2$ are isomorphic. We use $\chi_w^\dagger$ to denote the localization at $w_1$ (or $w_2$). We also require that $\chi_w^{\dagger}$ is distinct from $\chi_w$ and all $\rho_{w,i}$ and all $\rho_{w,i,j}^\dagger$.
\end{itemize}
Let
\begin{equation*}
\dot{\psi}^\sharp=\dot\chi+\sum_i\left(\bigg(\sum_j\dot\psi_{i,j}\bigg)+\dot\psi_i+\bigg(\sum_j\dot\psi_{i,j}^\dagger\bigg)\right)+\dot\chi^\dagger.
\end{equation*}
Since $\dot\psi^\sharp$ is (conjugate) self-dual with the same parity as $\psi$, similar to the proof of Corollary \ref{globalize.quasi-split}, one can show that there exists a $c$-Hermitian space $\dot V^\sharp$, such that $\dot G^\sharp=G(\dot V^\sharp)$ is quasi-split, and $\dot\psi^\sharp$ is an elliptic $A$-parameter for $\dot G^\sharp$.\\

Finally we check that the tuple of data $(\dot{F},\dot{E},\dot{V}^{\sharp},\dot{\psi}^{\sharp},u,w_1,w_2)$ we have constructed satisfies all our requirements. Indeed, except for the condition $5$, all other requirements follow from the construction directly. As for the condition $5$, notice that
\begin{align*}
\dot\psi_{w_1}^{\sharp}&=\sum_i\left(\rho_{w,i}\boxtimes S_1\boxtimes S_{b_i}+\sum_j\rho_{w,i}\boxtimes S_{b_i-2j}\boxtimes S_1\right)\\
&+\sum_{i,j}\rho_{w,i,j}^\dagger\boxtimes S_1\boxtimes S_1+\chi_w\boxtimes S_1\boxtimes S_1+\chi_w^{\dagger}\boxtimes S_1\boxtimes S_1
\end{align*}
is elementary and $\dot\psi_{w_1,d}^{\sharp}$ satisfies the chain condition. Hence by Theorem \ref{sc.criterion}, we may define a character $\eta'_{w_1}\in\widehat{\overline{\calS_{\dot{\psi}^{\sharp}_{w_1}}}}$ satisfying the alternating condition and initial condition, so that $\pi_M(\dot\psi_{w_1}^{\sharp},\eta'_{w_1})$ is supercuspidal ($\dot\chi$ and $\dot\chi^{\dagger}$ here guarantee that we can pick $\eta'_{w_1}$ in $\widehat{\overline{\calS_{\dot{\psi}^{\sharp}_{w_1}}}}$, rather than just $\widehat{{\calS_{\dot{\psi}^{\sharp}_{w_1}}}}$). Let 
\begin{equation*}
\eta_{w_1}=\eta'_{w_1}\cdot\left(\eta_{\dot\psi_{w_1}^{\sharp}}^{M/A}\right)^{-1}.
\end{equation*}
Then $\pi_A(\dot\psi_{w_1}^{\sharp},\eta_{w_1})=\pi_M(\dot\psi_{w_1}^{\sharp},\eta'_{w_1})$ is supercuspidal, as we required.
\end{proof}
Now, for $G=G(V^\epsilon)$, where $\epsilon\in\{\pm1\}$, we let $V^{\sharp,\epsilon}=V^\epsilon\oplus\calH^k$, and $G^\sharp=G(V^{\sharp,\epsilon})$. Let $P^\sharp$ be a maximal parabolic subgroup of $G^\sharp$ with Levi component
\[
    M^\sharp\simeq GL_k(E)\times G,
\]
and $\tau$ be the irreducible unitary representation of $GL_k(E)$ with $A$-parameter $\psi_\tau$.
\begin{corollary}\label{Ind.to.Sharp.Level}
Assume that $r>\dim V+k$. Then we have
\[
    \Pi_{\psi^\sharp}^{\theta}\left(G^\sharp\right)=\Ind_{P^\sharp}^{G^\sharp}\Big(\tau\boxtimes\Pi_\psi^{\theta,r}(G)\Big)
\]
as representations of $G^\sharp$, where the LHS is independent of the choice of $r$. In particular, if $G=G^*$ is quasi-split, then as sets, we have
\[
    \Pi_\psi^{\theta,r}(G^*)=\Pi_\psi^{A}(G^*).
\]
\end{corollary}
\begin{proof}
The first statement simply follows from Corollary \ref{Ind.Relation}, Corollary \ref{Comparison.as.Sets.Special.case}, and our ``sharp construction'' in this subsection. For the second statement, recall that Lemma \ref{Containment.A.in.Theta} already asserts that $\Pi_\psi^{A}(G^*)\subset\Pi_\psi^{\theta,r}(G^*)$, it remains to show reverse containment. But this is easy. By Corollary \ref{Comparison.as.Sets.Special.case}, when $G=G^*$ is quasi-split, we have 
\[
    \Pi_{\psi^\sharp}^{\theta}\left(G^{\sharp}\right)=\Pi_{\psi^\sharp}^{A}\left(G^{\sharp}\right)
\]
as representations of $G^{\sharp}$. On the other hand, we also have
\[
    \Pi_{\psi^\sharp}^{A}\left(G^\sharp\right)=\Ind_{P^\sharp}^{G^\sharp}\Big(\tau\boxtimes\Pi_\psi^{A}(G)\Big).
\]
Compare this equality with the equality in the first statement, we get
\[
    \Ind_{P^\sharp}^{G^\sharp}\Big(\tau\boxtimes\Pi_\psi^{\theta,r}(G)\Big)=\Ind_{P^\sharp}^{G^\sharp}\Big(\tau\boxtimes\Pi_\psi^{A}(G)\Big).
\]
Hence we have no choice but
\[
    \Pi_\psi^{\theta,r}(G^*)=\Pi_\psi^{A}(G^*).
\]
\end{proof}

\subsection{Descent along the Witt tower}
In the previous subsection, given a local $A$-parameter $\psi$ of good parity for the group $G$, we have constructed another local $A$-parameter $\psi^\sharp$ for some larger group $G^\sharp$. By using this construction, we have proved that some parabolic induction of the $\theta$-packet $\Pi_\psi^\theta(G)$ is indeed independent of the choice of $r$ (as a set). In this subsection, we shall prove that the $\theta$-packet $\Pi_\psi^\theta(G)$ itself is independent of the choice of $r$ (as a set). To achieve this, we use some techniques of the Jacquet modules. The method we are using here is similar to that in \cite{MR2906916} Section 5.2.\\

We retain the notations and assumptions from the last subsection. So $\psi$ is a local $A$-parameter of good parity for $G^*$. Let
\[
    \calU=\left\{\pi\in\Irr_{unit}\left(G\right)~\big|~\Hom_{M^\sharp}\left(\tau\boxtimes\pi,s.s.Jac_{P^\sharp}\Pi_{\psi^\sharp}^{\theta}\left(G^\sharp\right)\right)\neq0\right\},
\]
where $s.s.Jac_{P^\sharp}$ means taking the semi-simplification of the Jacquet-module along the parabolic $P^\sharp$. Obviously $\calU$ is a finite subset of $\Irr_{unit}\left(G\right)$. Moreover, it follows from Corollary \ref{Ind.to.Sharp.Level} that for $r$ sufficiently large, we have
\[
    \Pi_\psi^{\theta,r}(G)\subset\calU.
\]
\begin{lemma}
Fix a positive integer $r_0>\dim V$. Then, for $r$ sufficiently large, we have
\[
    \Pi_\psi^{\theta,r}(G)\subset\Pi_\psi^{\theta,r_0}(G).
\]
\end{lemma}
\begin{proof}
To distinguish notations, for an irreducible representation $\pi$ of $G$, we shall use $\sigma^r=\theta^r(\pi)$ to denote the theta lift of $\pi$ to the group $H^r=H(W_{(r)})$. Since $\calU$ is a finite subset of $\Irr_{unit}\left(G\right)$, according to \cite{MR3502978} Proposition 3.2, there is a positive integer $N_u$, such that for all $r>r_0+N_u$, and all $\pi\in\calU$, $\sigma^r$ is a subrepresentation of 
\[
    \Ind_{Q^{r,r_0}}^{H^r}\left(\chi_V|\cdot|^{n-r}\boxtimes\cdots\boxtimes\chi_V|\cdot|^{n-r_0-1}\boxtimes\sigma^{r_0}\right),
\]
where $Q^{r,r_0}$ is a parabolic subgroup with Levi component $GL_1(E)\times\cdots\times GL_1(E)\times H^{r_0}$.\\

Next we choose a tuple of data
\[
    \left(>_\psi,\psi_{\gg},V^{\gg},r_1\right),
\]
where
\begin{itemize}
\item $>_\psi$ is an admissible order on $I_\psi$;
\item $\psi_{\gg}$ is a local $A$-parameter has DDR for the group $G_{\gg}=G(V^{\gg})$, with $V^{\gg}$ a space in the Witt tower containing $V$; for any $c$-skew-Hermitian space $W$, we shall also let $W^{\gg}$ be the space in the Witt tower containing $W$, such that
    \[
        \dim W^{\gg}-\dim V^{\gg}=\dim W-\dim V;
    \]
\item $r_1>\max\left\{r_0,\dim V^\gg\right\}$ is a positive integer;
\end{itemize}
such that the following conditions holds:
\begin{enumerate}
\item for any $r\geq r_0$, $>_\psi$ can be uniquely extended to an admissible order $>_{\psi,r}$ on the index set of $\theta^r(\psi)$ 
\[
    I_{\theta^{r}(\psi)}=I_\psi\sqcup\left\{i^r\right\},
\]
such that $i^r$ is the unique maximal element under the partial order $>_{\psi,r}$; here $i^r$ is the element in the index set $I_{\theta^{r}(\psi)}$ corresponding to the irreducible constituent $\chi_V\boxtimes S_{2r-2n+1}$;
\item for any $r>r_1$, the local $A$-parameter
\[
    \theta^{r}(\psi_{\gg})=\psi_{\gg}\chi_W^{-1}\chi_V+\chi_V\boxtimes S_{2r-2n+1}
\]
for the group $H_{\gg}^r=H\left(\left(W_{(r)}\right)^\gg\right)$ has DDR, and dominates $\theta^r(\psi)$ with respect to the admissible order $>_{\psi,r}$; in particular, $\theta^{r}(\psi_{\gg})$ will also dominate $\theta^{r_0}(\psi)$ with respect to the admissible order $>_{\psi,r_0}$.\\
\end{enumerate}

Now assume that $r>\max\left\{\dim V+k,r_0+N_u,r_1\right\}$. For any irreducible unitary representation $\pi\in\Pi_\psi^{\theta,r}(G)$, it is sufficient to prove that its theta lift $\sigma^{r_0}$ to the group $H^{r_0}$ lies in the local $A$-packet $\Pi_{\theta^{r_0}(\psi)}^A\left(H^{r_0}\right)$. We consider the theta lift $\sigma^r$ of $\pi$ to the group $H^r$. By the definition of the $\theta$-packet, $\sigma^r$ lies in the local $A$-packet $\Pi_{\theta^{r}(\psi)}^A\left(H^{r}\right)$. According to Theorem \ref{Construction.bp.From.DDR}, there exists some $\sigma_{\gg}\in\Pi_{\theta^{r}(\psi_{\gg})}^A\left(H^r_{\gg}\right)$, such that
\[
    \sigma^r=\circ_{i\in I_\psi}Jac_{X_i^\gg}\sigma_\gg,
\]
where we identify $I_\psi$ with a subset of $I_{\theta^r(\psi)}$ in the obvious way, and each $X_i^\gg$ is some generalized segment. Since $\theta^{r}(\psi_{\gg})$ also dominants $\theta^{r_0}(\psi)$ with respect to the admissible order $>_{\psi,r_0}$, again by Theorem \ref{Construction.bp.From.DDR}, the representation
\[
    Jac_{X_{i^{r_0}}^\gg}\sigma^r=Jac_{X_{i^{r_0}}^\gg}\circ\left(\circ_{i\in I_\psi}Jac_{X_i^\gg}\sigma_\gg\right)
\]
is either zero or irreducible and lies in the $A$-packet $\Pi_{\theta^{r_0}(\psi)}^A\left(H^{r_0}\right)$, where the generalized segment $X_{i^{r_0}}^\gg$ is 
\[
    X_{i^{r_0}}^\gg=\Big(\begin{array}{ccc}
                        n-r & \cdots & n-r_0-1
                    \end{array}\Big).
\]
On the other hand, since $\sigma^r$ is a subrepresentation of 
\[
    \Ind_{Q^{r,r_0}}^{H^r}\left(\chi_V|\cdot|^{n-r}\boxtimes\cdots\boxtimes\chi_V|\cdot|^{n-r_0-1}\boxtimes\sigma^{r_0}\right),
\]
it follows that
\[
    Jac_{X_{i^{r_0}}^\gg}\sigma^r\simeq\sigma^{r_0}
\]
is non-zero and hence lies in the local $A$-packet $\Pi_{\theta^{r_0}(\psi)}^A\left(H^{r_0}\right)$. This completes the proof.
\end{proof}
\begin{corollary}\label{Main.Theorem.Set.Level}
Let $\psi$ be a local $A$-parameter for $G^*$. Then as a set, the $\theta$-packet $\Pi_\psi^\theta\left(G\right)$ is indeed independent of the choice of $H=H\left(W_{(r)}\right)$. Moreover, if $G=G^*$ is quasi-split, then as sets, we have
\[
    \Pi_\psi^{\theta}(G^*)=\Pi_\psi^{A}(G^*).
\]
\end{corollary}
\begin{proof}
We first assume that $\psi$ is of good parity. In this case, it suffices to prove that as sets, we have
\[
    \Pi_\psi^{\theta,r}(G)=\Pi_\psi^{\theta,r_0}(G)
\]
in the previous lemma. We consider all pure inner forms of $G^*$ simultaneously: it then follows from the previous lemma that
\[
    \Big|\bigsqcup\Pi_\psi^{\theta,r}\left(G\right)\Big|\leq\Big|\bigsqcup\Pi_\psi^{\theta,r_0}\left(G\right)\Big|,
\]
where the disjoint unions on both sides run over all pure inner forms $G$ of $G^*$. On the other hand, we deduce from Proposition \ref{Count.Size.Theta.Packets} that
\[
    \Big|\bigsqcup\Pi_\psi^{\theta,r}\left(G\right)\Big|=\Big|\Pi_{\theta^r(\psi)}^A(H^r)\Big|\quad\textit{and}\quad\Big|\bigsqcup\Pi_\psi^{\theta,r_0}\left(G\right)\Big|=\Big|\Pi_{\theta^{r_0}(\psi)}^A(H^{r_0})\Big|.
\]
Also, as explicated in the proof of the previous lemma, Theorem \ref{Construction.bp.From.DDR} asserts that $\Pi_{\theta^{r_0}(\psi)}^A(H^{r_0})$ can be obtained from $\Pi_{\theta^r(\psi)}^A(H^r)$ by taking some partial Jacquet modules. To be more precise, we have 
\[
    \Pi_{\theta^{r_0}(\psi)}^A(H^{r_0})=\left\{\sigma_0=Jac_{X_{i^{r_0}}^\gg}\sigma~\big|~\sigma\in\Pi_{\theta^r(\psi)}^A(H^r)\right\}\Big\backslash\left\{0\right\} ,
\]
where the generalized segment $X_{i^{r_0}}^\gg$ is the same that in the previous lemma. It follows that
\[
    \Big|\Pi_{\theta^r(\psi)}^A(H^r)\Big|\geq\Big|\Pi_{\theta^{r_0}(\psi)}^A(H^{r_0})\Big|.
\]
Hence there is no other choice that we must have
\[
    \Pi_\psi^{\theta,r}(G)=\Pi_\psi^{\theta,r_0}(G)
\]
as subsets of $\Irr_{unit}\left(G\right)$. This completes the proof in the good parity case.\\

The general case then easily follows from the good parity case and Corollary \ref{Ind.Relation}.
\end{proof}

\section{Independency on the auxiliary data}\label{Chap.Indepedency.Aux.Data}
We have already proved that, as subsets of $\Irr_{unit}\left(G\right)$, the definition of the $\theta$-packets is independent of the choice of $H=H\left(W_{(r)}\right)$. But this is not the only choice we have made: recall that in the definition of the theta lift between $(G,H)$, we also need to choose a tuple of auxiliary data $(\psi_F,\chi_V,\chi_W)$. In fact as subsets of $\Irr_{unit}\left(G\right)$, the definition of the $\theta$-packets is also independent of the choice of these data, though maybe this fact is not so apparent. In this section we shall investigate this independency.

\subsection{Similitude group action vs. Adjoint group action}
In this subsection we consider two actions on the group $H$, one is by the similitude group, and another one is by the adjoint group.\\

Recall that $W$ is a $c$-skew-Hermitian space over $E$, and $H=H(W)$ is the isometry group associated to $W$. Let $H^\sim=H^\sim(W)$ be the group of elements $h$ in $GL(W)$ such that
\[
	\langle hv,hw\rangle_W=\lambda_h\cdot\langle v,w\rangle_W\quad\textit{for }v,w\in W,
\]
where $\lambda_h\in E^\times$ is some constant, called the scale of $h$. We shall call $H^\sim$ the similitude group associated to $W$. Let $Z_H$ be the center of $H$, we also have another group
\[
	H_{ad,/F}=H_{/F}\big/Z_{H,/F},
\]
which we shall call it the adjoint group of $H$. Here we use the subscript ``$~_{/F}$'' to emphasize that the groups are regarded as algebraic groups over $F$. Denote by $H_{ad}$ the $F$-points of $H_{ad,/F}$. There is a commutative diagram of algebraic groups
\[
	\begin{CD}
	1 @>>> Z_{H,/F} @>>> H_{/F} @>>> H_{ad,/F} @>>> 1\\
	@. @VVV @VVV @| @.\\
	1 @>>> Res_{E/F}\GG_m @>>> H^\sim_{/F} @>>> H_{ad,/F} @>>> 1
	\end{CD}\quad.
\]
We derive from this diagram that
\[
	H_{ad}\simeq H^\sim\big/E^\times
\]
as abstract groups. Recall that $H^\sim$ acts on the group $H$ by conjugation. The adjoint group $H_{ad}$ also has an action on the group $H$, which can be described as follows. Let $\overline{h}\in H_{ad}$, and $\wt h\in H_{/F}(\overline{F})$ be a lift of $\overline{h}$, where $\overline{F}$ is the algebraic closure of $F$. Then for $x\in H$, the action of $\overline{h}$ on $x$ is 
\[
	\overline{h}.x=\wt h\cdot x\cdot \wt h^{-1}.
\]
These actions induce the actions of $H^\sim$ and $H_{ad}$ on representations/ functions/ distributions of $H$.
\begin{lemma}\label{Similitude.vs.Adjoint}
The conjugation action of $H^\sim$ on the group $H$ factor through $H_{ad}$, i.e. for any $h\in H^\sim$, and $x\in H$, we have
\[
	h.x=\overline{h}.x,
\]
where $\overline{h}$ is the image of $h$ in $H_{ad}$.
\iffalse\[
	\begindc{\commdiag}[500]
	\obj(0,1){$H^\sim$}
	\obj(1,1){$H_{ad}$}
	%\obj(0,0){$C$}
	\obj(1,0){$Aut(H)$}
	\mor{$H^\sim$}{$H_{ad}$}{$~$}
	%\mor{$A$}{$C$}{$b$}
	\mor{$H_{ad}$}{$Aut(H)$}{$~$}
	\mor{$H^\sim$}{$Aut(H)$}{$~$}
	\enddc
\]\fi
\end{lemma}
\begin{proof}
This is trivial.
\end{proof}

Notice that in our cases, the derived group of $H$ is simply-connected. Hence the stable conjugacy in $H$ is just the same as the $H_{/F}(\overline{F})$-conjugacy. We have
\begin{lemma}\label{Adjoint.Preserve.A-packets}
The action of $H_{ad}$ on irreducible representations of $H$ preserves the local $A$-packets of $H$.
\end{lemma}
\begin{proof}
Let $\psi_H$ be a local $A$-parameter for $H$. Recall that Arthur/ Mok have attached a stable distribution $S\Theta_{\psi_H}$ to $\psi_H$, which is a linear combination of characters of irreducible unitary representations in the $A$-packet $\Pi_{\psi_H}^A(H)$
\[
	S\Theta_{\psi_H}=\sum_{\sigma\in\Pi_{\psi_H}^A(H)}\calJ^A_{\scrW'}(\sigma)(s_{\psi_H})\cdot\Theta_\sigma,
\]
where $s_{\psi_H}$ is certain element in the component group, and $\Theta_\sigma$ is the character of $\sigma$. Since $S\Theta_{\psi_H}$ is stable, the action of $H_{ad}$ preserves this distribution. Let $\overline{h}\in H_{ad}$, we obtain
\[
	\sum_{\sigma\in\Pi_{\psi_H}^A(H)}\calJ^A_{\scrW'}(\sigma)(s_{\psi_H})\cdot\Theta_\sigma=\sum_{\sigma\in\Pi_{\psi_H}^A(H)}\calJ^A_{\scrW'}(\sigma)(s_{\psi_H})\cdot\Theta_{^{\overline{h}}\sigma}.
\]
By the linear independency of the characters, we can conclude that the lemma holds.
\end{proof}

\subsection{Even orthogonal case: the scaling property}
Now we prove the independence of $\theta$-packets (as sets) on the auxiliary data $(\psi_F,\chi_V,\chi_W)$ in Case $O$. In this case, the pair of characters $(\chi_V,\chi_W)$ is fixed, so we only need to consider the changes of the additive character $\psi_F$.\\

Let $\psi$ be a local $A$-parameter for $G$. In this subsection, to emphasize the possible dependence of the $\theta$-packets on the choice of $\psi_F$, we shall write $\Pi_\psi^\theta(G)$ as $\Pi_\psi^{\theta,\psi_F}(G)$. Let $a\in F^\times$, and $\psi_{F,a}=\psi_F(a\cdot~)$ be another additive character of $F$. Recall that we have the well-known scaling property of the theta lift (cf. \cite{kudla1996notes} II Corollary 6.2 and IV Proposition 1.9)
\[
	\theta_{\psi_{F,a}}(\pi)=~^{\delta_a}\theta_{\psi_F}(\pi)
\]
for any irreducible smooth representation $\pi$ of $G$, where $\theta_{\psi_F}(\pi)$ means the theta lift of $\pi$ to the group $H$ with respect to the additive character $\psi_F$, and $\delta_a$ is an element in the similitude $H^\sim$ with scale $a$. We rewrite the $\theta$-packet $\Pi_\psi^{\theta,\psi_F}(G)$ as
\[
	\Pi_\psi^{\theta,\psi_F}(G)=\left\{\pi\in\Irr_{unit}(G)~\big|~\theta_{\psi_F}(\pi)\in\Pi_{\theta(\psi)}^A(H)\right\}.
\]
Let $\overline{\delta}_a$ be the image of $\delta_a$ in $H_{ad}$. Then, for any $\pi\in\Pi_\psi^{\theta,\psi_F}(G)$, according to Lemma \ref{Similitude.vs.Adjoint} and Lemma \ref{Adjoint.Preserve.A-packets}, we have
\[
	\theta_{\psi_{F,a}}(\pi)=~^{\overline{\delta}_a}\theta_{\psi_F}(\pi)
\]
also lies in the $A$-packet $\Pi_{\theta(\psi)}^A(H)$. It follows that 
\[
	\Pi_\psi^{\theta,\psi_F}(G)\subset\Pi_\psi^{\theta,\psi_{F,a}}(G).
\]
Similarly, we also have the reverse containment. This implies that as a set of irreducible unitary representations, $\Pi_\psi^{\theta,\psi_F}(G)$ is in fact not dependent on the choice of $\psi_F$.

\subsection{Unitary case: inputs from Archimedean places}
Now we prove the independence of $\theta$-packets (as sets) on the auxiliary data $(\psi_F,\chi_V,\chi_W)$ in Case $U$. In this case, we have the flexibility of choosing the pair of characters $(\chi_V,\chi_W)$. Hence the scaling property is not sufficient for us to prove the independence. We shall use another approach.\\

The idea is the same as that of Section \ref{sc.A-parameter}, i.e. trying to use the global method. Let $\psi$ be a local $A$-parameter of good parity for the group $G$. Applying Lemma \ref{Golobalization-Origin}, one can easily construct a tuple of data $(\dot F,\dot E,\dot V,\dot\psi,u,w)$, where:
\begin{itemize}
	\item $\dot F$ is a number field, and $\dot E$ is a quadratic field extension of $\dot F$;
	\item $\dot{V}^+$ is a $c$-Hermitian space over $\dot{E}$ so that $\dot G^*=G(\dot V^+)$ is quasi-split;
	\item $\dot{\psi}$ is an elliptic $A$-parameter of $\dot G$;
	\item $u,w$ are finite places of $\dot{F}$;
\end{itemize}
such that the following conditions hold:
\begin{enumerate}
	\item $(\dot{F}_u,\dot{E}_u,\dot{V}^+_u,\dot{\psi}_u)\simeq(F,E,V^+,\psi)$;
	\item $\dot F$ has at least one real place, and $\dot E$ is not split at this place;
	\item $\dot{\psi}_{w}$ is elementary, and the localization map
		\begin{equation*}
			\iota_{w}:\calS_{\dot{\psi}}\lra\calS_{\dot{\psi}_{w}}
		\end{equation*}
	is an isomorphism.
\end{enumerate}
Indeed, as stated in Remark \ref{Globalization-Origin'}, we can further impose some requirements on one more Archimedean place. Let $w_\infty$ be a real place of $\dot F$, such that $\dot E$ is not split at $w_\infty$. We require that
\begin{enumerate}
	\item[$(4)$]$\dot\psi_{w_\infty}$ is of good parity, and $\left(\dot\psi_{w_\infty}\right)_d$ is a square-integrable $L$-parameter for $\dot G_{w_\infty}$; here $\left(\dot\psi_{w_\infty}\right)_d$ is a $\dim V$-dimensional representation of $L_\CC$ defined by
		\[
			\left(\dot\psi_{w_\infty}\right)_d(z)=\dot\psi_{w_\infty}\left(z,\left(\begin{array}{cc}
																						\left(z/\bar{z}\right)^{1/2} & \\
																						 & \left(z/\bar{z}\right)^{-1/2}
																				    \end{array}\right)\right)
		\]
\end{enumerate}
We shall need the following remarkable fact:
\begin{lemma}
Let $V_{w_\infty}$ be an anisotropic Hermitian space over $\CC$ of the same dimension as $V$, and let $G_{w_\infty}=G(V_{w_\infty})$ be the compact unitary group associated to it. Let $W_{w_\infty}$ be the split skew Hermitian space over $\CC$ of the same dimension as $W$ (as explicated at the begining of Section \ref{Statements.Main.Results}), and let $H_{w_\infty}=H(W_{w_\infty})$. Then we have
\[
	\Pi_{\dot\psi_{w_\infty}}^\theta\left(G_{w_\infty}\right)=\Pi_{\left(\dot\psi_{w_\infty}\right)_d}^L\left(G_{w_\infty}\right).
\]
In particular, $\Pi_{\dot\psi_{w_\infty}}^\theta\left(G_{w_\infty}\right)$ is non-empty.
\end{lemma}
\begin{proof}
By \cite{MR3947270} Th\'eor\`eme 1.1, if we replace the terminology ``$A$-packets'' by ``$AJ$-packets'' (i.e. Adams-Johnson packets), then Conjecture \ref{Intro.Adams.Conjecture} (B) holds for real unitary groups. On the other hand, by \cite{cossutta2009theta} Corollary 3.8, we have 
\[
	\Pi_{\theta(\dot\psi_{w_\infty})}^{AJ}\left(H_{w_\infty}\right)=\Pi_{\theta(\dot\psi_{w_\infty})}^{A}\left(H_{w_\infty}\right).
\]
It follows that
\[
	\Pi_{\dot\psi_{w_\infty}}^\theta\left(G_{w_\infty}\right)=\Pi_{\dot\psi_{w_\infty}}^{AJ}\left(G_{w_\infty}\right)=\Pi_{\left(\dot\psi_{w_\infty}\right)_d}^L\left(G_{w_\infty}\right).
\]
\end{proof}

Let $\pi\in\Pi_\psi^\theta(G)$. We now try to globalize $\pi$ to a discrete automorphic representation of some unitary group over $\dot F$, with $A$-parameter $\dot\psi$. Firstly we globalize the tuple of local auxiliary data $(\psi_F,\chi_V,\chi_W)$ in the definition of the local theta lift to a global tuple $(\psi_{\dot F},\chi_{\dot V},\chi_{\dot W})$. Let $\dot W=\dot W_{(r)}$ be the split skew-Hermitian space over $\dot E$ with the same dimension as $W=W_{(r)}$. Put $\dot H=H(\dot W)$. Similar to the proof of Proposition \ref{Antipasto}, applying the Arthur's multiplicity formula for $\dot H$, we may construct a discrete automorphic representation $\dot{\sigma}$ of $\dot H$ with elliptic $A$-parameter 
\begin{equation*}
\theta(\dot\psi)=\dot\psi\chi_{\dot W}^{-1}\chi_{\dot V}+\chi_{\dot V}\boxtimes S_{2r-2n+1},
\end{equation*}
such that:
\begin{enumerate}
\item at the place $u$, $\dot\sigma_u=\theta(\pi)$;
\item at the place $w_\infty$, $\dot\sigma_{w_\infty}=\theta\left(\pi_{w_\infty}\right)$ for the unique $\pi_{w_\infty}\in\Pi_{\left(\dot\psi_{w_\infty}\right)_d}^L\left(G_{w_\infty}\right)$.
\end{enumerate}
By J-S. Li's work on low rank representations (Theorem \ref{J-S.Li.Low.rk.Global}), there is an pure inner form $\dot G$ of $\dot G^*$ and an automorphic representation $\dot\pi$ of $\dot G$, such that
\begin{equation*}
\dot\sigma=\theta^{abs}(\dot\pi).
\end{equation*}
Also, by our construction, $\dot G_{w_\infty}=G_{w_\infty}$ is compact. This forces any automorphic realization of $\dot\pi$ to be cuspidal. Hence $\dot\pi$ lies in the automorphic discrete spectrum of $\dot G$.\\

So now, we have successfully globalize the local triple $(G,\psi,\pi)$ to a global triple $(\dot G,\dot\psi,\dot\pi)$. The same argument as Corollary \ref{Comparison.as.Sets.Special.case} implies that for any tuple of auxiliary data $(\psi'_F,\chi'_V,\chi'_W)$, $\pi$ will also lie in the $\theta$-packet defined with respect to $(\psi'_F,\chi'_V,\chi'_W)$. It follows that the definition of the $\theta$-packets is independent of the choice of $(\psi_F,\chi_V,\chi_W)$.
\begin{remark}
The argument in this subsection can be used to prove Corollary \ref{Main.Theorem.Set.Level} in Case $U$ as well. Indeed, this would simplify our proof in Case $U$.
\end{remark}

\section{Local intertwining relation}
To investigate the ``labeling'' of irreducible unitary representations inside a local $A$-packet for quasi-split classical groups, i.e. the characters of the component group attached to each element inside the packet, we need the so-called ``local intertwining relation''. However, the original version of local intertwining relation formulated by Arthur is not so convenient for our applications, we shall formulate an alternative version of it, in the spirit of the one formulated in \cite{MR3573972} \cite{MR3708200}. To distinguish these two versions, we shall write Arthur's version ``LIR-A'' for abbreviation, and write Gan-Ichino's version ``LIR-B'' for abbreviation.\\

In this section we retain the notations of Section \ref{Compatible.Theta.N.Ind}.

\subsection{Local intertwining operators}\label{Def.NLIO}
We first briefly recall the definition of the (normalized) intertwining operators, for both quasi-split and non quasi-split groups.\\

Recall that we have
\[
  V=X+V_0+X^*
\]
for some $k$-dimensional totally isotropic subspace $X$ and $X^*$ of $V$. Let $n_0=n-k$, and $r_0=r-k$. We have a maximal parabolic subgroup $P=M_PU_P$ of $G$ stabilizing $X$, where $M_P$ is the Levi component of $P$ stabilizing $X^*$ and $U_P$ is the unipotent radical of $P$. Let $\tau$ be an irreducible unitary representation of $GL(X)$ on a space $\scrV_\tau$ with central character $\omega_\tau$. Suppose that $\tau$ is of Arthur type, and corresponds to an irreducible $A$-parameter $\psi_\tau$. Let $\psi_0$ be a local $A$-parameter for $G_0$, and $\pi_0$ an irreducible unitary representation in the packet $\Pi_{\psi_0}^\theta(G_0)$. We consider the induced representation
\begin{equation*}
\Ind_P^{G}\left(\tau_s\boxtimes\pi_0\right)
\end{equation*}
of $G$. Let $A_P$ be the split component of the center of $M_P$ and $W(M_P)=N_{G}(A_P)/M_P$ be the relative Weyl group for $M_P$. Noting that $W(M_P)\simeq\ZZ/2\ZZ$, we denote by $w$ the non-trivial element in $W(M_P)$. For any representative $\wt{w}\in G$ of $w$, we define an unnormalized intertwining operator
\begin{equation*}
\calM(\wt{w},\tau_s\boxtimes \pi_0):\Ind_P^{G}(\tau_s\boxtimes \pi_0)\lra\Ind_P^{G}\left(w(\tau_s\boxtimes \pi_0)\right)
\end{equation*}
by (the meromorphic continuation of) the integral
\begin{equation*}
\calM(\wt{w},\tau_s\boxtimes \pi_0)\varPhi_s(g)=\int_{U_P}\varPhi_s(\wt{w}^{-1}ug)du,
\end{equation*}
where the Haar measures on various groups are fixed as in Section \ref{Compatible.Theta.N.Ind}.\\

Having fixed an additive character $\psi_F$, as explicated in Section \ref{Whittaker.Data}, we can define a Whittaker datum $\scrW_{\psi_F}=\scrW$ of $G^*$. To normailze this operator with respect to the Whittaker datum $\scrW_{\psi_F}$, we need to choose the following data appropriately:
\begin{itemize}
\item a representative $\wt{w}$;
\item a normalizing factor $r(w,\tau_s\boxtimes\pi_0)$;
\item an intertwining isomorphism $\calA_w$.
\end{itemize}
For the representative, we take $\wt{w}\in G$ defined by
\begin{equation*}
\wt{w}=w_P\cdot m_P\left((-1)^{n'}\cdot\kappa_V\cdot J\right)\cdot(-1_{V_0})^{k},
\end{equation*}
where $w_P$ is as in Section \ref{Compatible.Theta.N.Ind}, $n'=[\dim V/2]$, and
\begin{equation*}
\kappa_V=\begin{cases}
c\quad&\textit{Case $O$};\\
-\delta\quad&\textit{Case $U_0$};\\
-1\quad&\textit{Case $U_1$},
\end{cases}
\end{equation*}
where the constant $c$ in Case $O$ and the constant $\delta$ in Case $U_0$ arise in the choice of Whittaker datum $\scrW_{\psi_F}$; and 
\begin{equation*}
J=\left(\begin{array}{cccc}
                                    {} & {} & {} & (-1)^{k-1}\\
                                    {} & {} & \iddots & {}\\
                                    {} & -1 & {} & {}\\
                                    1 & {} & {} & {}

                                 \end{array}\right)\in GL_k(E).
\end{equation*}
Here, we have identified $GL(X)$ with $GL_k(E)$ by choosing a basis of $X$. In \cite{MR3573972} Section 7.3, it was shown that the representative defined above coincides with the representative defined in \cite{MR3338302} when we are in Case $U$ and $\epsilon(V)=1$.\\

Next we define the normalizing factor $r(w,\tau_s\boxtimes\sigma_0)$. Put
\begin{equation*}
\lambda(w,\psi_F)=\begin{cases}
	\lambda(E'/F,\psi_F)^k \quad&\textit{Case $O$, here $E'$ is the splitting field of $V^+$};\\
    \lambda(E/F,\psi_F)^{(k-1)k/2} \quad&\textit{Case $U_0$};\\
    \lambda(E/F,\psi_F)^{(k+1)k/2} \quad&\textit{Case $U_1$}.
\end{cases}
\end{equation*}
where $\lambda(E'/F,\psi_F)$ or $\lambda(E/F,\psi_F)$ is the Langlands $\lambda$-factor. Let $\phi_\tau$ and $\phi_0$ be the $L$-parameters associated to the $A$-parameter $\psi_\tau$ and $\psi_0$ respectively. We set
\begin{equation*}
r(w,\tau_s\boxtimes\pi_0)=\lambda(w,\psi_F)\cdot\gamma\left(s,\phi_\tau\otimes\phi_0^\vee,\psi_E\right)^{-1}\cdot\gamma\left(2s,R\circ\phi_\tau,\psi_F\right)^{-1},
\end{equation*}
and the normalized intertwining operator
\begin{equation*}
\calR(w,\tau_s\boxtimes\pi_0)\coloneqq|\kappa_V|^{k\rho_P}\cdot r(w,\tau_s\boxtimes\pi_0)^{-1}\cdot\calM(\wt w,\tau_s\boxtimes\pi_0);
\end{equation*}
where 
\[
	R=\begin{cases}
		  \bigwedge^2 \quad&\textit{Case $O$};\\
		  As^- \quad&\textit{Case $U_0$};\\
		  As^+ \quad&\textit{Case $U_1$}.
	  \end{cases}
\]
\begin{lemma}\label{holomorphy.NIO}
The normalized intertwining operator $\calR(w,\tau_s\boxtimes\pi_0)$ is holomorphic at $s=0$.
\end{lemma}
At least when $G=G^*$ is quasi-split, this should follow from \cite{MR3135650} Proposition 2.3.1 in Case $O$, or \cite{MR3338302} Proposition 3.3.1 in Case $U$. We shall prove this lemma in the next section for $G$ is non quasi-split.\\

Finally we define the intertwining isomorphism. Assume that $w(\tau\boxtimes\pi_0)\simeq\tau\boxtimes\pi_0$, which is equivalent to $\left(\tau^c\right)^\vee\simeq\tau$. We may take the unique isomorphism
\begin{equation*}
\calA_w:\scrV_\tau\otimes\scrV_{\pi_0}\lra\scrV_\tau\otimes\scrV_{\pi_0}
\end{equation*}
such that:
\begin{itemize}
\item $\calA_w\circ (w(\tau\boxtimes\pi_0))(m)=(\tau\boxtimes\pi_0)(m)\circ\calA_w$ for all $m\in M_P$;
\item $\calA_w=\calA_w'\otimes1_{\scrV_{\pi_0}}$ with an isomorphism
    \begin{equation*}
        \calA_w':\scrV_\tau\lra\scrV_\tau
    \end{equation*}
    as described in \cite{MR3135650} Section 2.2, or \cite{MR3338302} Section 3.2.
\end{itemize}

Note that $\calA_w^2=1_{\scrV_\tau\otimes\scrV_{\pi_0}}$. We define a self-intertwining operator
\begin{equation*}
R(w,\tau\boxtimes\pi_0):\Ind_P^{G}(\tau\boxtimes \pi_0)\lra\Ind_P^{G}(\tau\boxtimes \pi_0)
\end{equation*}
by
\begin{equation*}
R(w,\tau\boxtimes\sigma_0)\varPhi(g)=\calA_w(\calR(w,\tau\boxtimes\sigma_0)\varPhi(g)).
\end{equation*}
We shall also use the notation $R(w,\tau\boxtimes\pi_0,\psi_F)$ if we want to emphasize the dependence of $R(w,\tau\boxtimes\pi_0)$ on the additive character $\psi_F$.\\

Similarly, we can define the intertwining operator for $H$, with respect to a Whittaker datum $\scrW_{\psi_F}$ of $H$, where
\[
	\scrW_{\psi_F}=\begin{cases}
			  \scrW'_1\quad&\textit{Case $O$};\\
			  \scrW'\quad&\textit{Case $U$}.
		  \end{cases}
\]
We put 
\begin{equation*}
\wt{w}=w_Q\cdot m_Q\left((-1)^{r'}\cdot\kappa_W\cdot J\right)\cdot(-1_{W_0})^{k},
\end{equation*}
where $w_Q$ is as in Section \ref{Compatible.Theta.N.Ind}, $r'=[\dim W/2]$, and
\[
	\kappa_W=\begin{cases}
				1\quad&\textit{Case $O$};\\
				-\delta\quad&\textit{Case $U_0$};\\
				1\quad&\textit{Case $U_1$},
			 \end{cases}
\]
Let $\sigma_0$ be an irreducible unitary representation of $H_0$ lies in some local $A$-packet $\Pi_{\psi'_0}^A(H_0)$. We denote the $L$-parameters associated to the $A$-parameter $\psi'$ by $\phi'_0$. We set
\begin{equation*}
r(w,\tau_s\boxtimes\sigma_0)=\lambda(w,\psi_F)\cdot\gamma\left(s,\phi_\tau\otimes\left(\phi'_0\right)^\vee,\psi_E\right)^{-1}\cdot\gamma(2s,R\circ\phi_\tau,\psi_F)^{-1},
\end{equation*}
where
\[
	\lambda(w,\psi_F)=\begin{cases}
	1 \quad&\textit{Case $O$};\\
    \lambda(E/F,\psi_F)^{(k+1)k/2} \quad&\textit{Case $U_0$};\\
    \lambda(E/F,\psi_F)^{(k-1)k/2} \quad&\textit{Case $U_1$},
					  \end{cases}
\]
and
\[
	R=\begin{cases}
		  \bigwedge^2 \quad&\textit{Case $O$};\\
		  As^+ \quad&\textit{Case $U_0$};\\
		  As^- \quad&\textit{Case $U_1$}.
	  \end{cases}
\]
Put 
\begin{equation*}
\calR(w,\tau_s\boxtimes\sigma_0)\coloneqq|\kappa_W|^{k\rho_Q}\cdot r(w,\tau_s\boxtimes\sigma_0)^{-1}\cdot\calM(\wt w,\tau_s\boxtimes\sigma_0),
\end{equation*}
it follows from Arthur and Mok's work that this normalized intertwining operator is holomorphic at $s=0$. Assume that $w(\tau\boxtimes\sigma_0)\simeq\tau\boxtimes\sigma_0$, we take an isomorphism $\calA_w$ similarly, and define the self-intertwining operator $R(w,\tau\boxtimes\sigma_0)$ by 
\[
	R(w,\tau\boxtimes\sigma_0)\varPhi(h)=\calA_w\left(\calR(w,\tau_s\boxtimes\sigma_0)\varPhi(h)\right)
\]
for $\varPhi\in\Ind_Q^{H}\left(\tau\boxtimes\sigma_0\right)$, and $h\in H$.

\subsection{Local intertwining relation: an alternative version}
Now we can state the desired version of the local intertwining relation, which gives us a chance to interpret the ``labeling'' of a local $A$-packet as some representation-theoretical quantities. For the original version of the local intertwining relation, one can refer to \cite{MR3135650} Proposition 2.4.3 (Case $O$), \cite{MR3338302} Proposition 3.4.4, and also \cite{kaletha2014endoscopic} Chapter 2 (Case $U$).
\begin{theorem}[LIR-B for quasi-split groups]\label{LIR-B.Qusi.Split}
Suppose that $G=G^*$ is quasi-split. Let $\pi$ be an irreducible constituent of $\Ind_P^{G}(\tau\boxtimes\pi_0)$. Then:
\begin{enumerate}
\item $\pi$ is in the local $A$-packet $\Pi_\psi^{A}(G)$, where
	\[
		\psi=\psi_\tau+\psi_0+\left(\psi_\tau^c\right)^\vee;
	\]
\item if we regard $\calS_{\psi_0}$ as a subgroup of $\calS_\psi$ via the natural embedding, then
\begin{equation*}
\calJ^A_{\scrW_{\psi_F}}(\pi)\Big|_{\calS_{\psi_0}}=\calJ^A_{\scrW_{\psi_F}}(\pi_0);
\end{equation*}
\item if we further assume that $\psi_\tau$ is (conjugate) self-dual with the same parity as $\psi$, then the restriction of the normalized intertwining operator $R(w,\tau\boxtimes\pi_0,\psi_F)$ to $\pi$ is the scalar multiplication by
\[
	R(w,\tau\boxtimes\pi_0)\Big|_\pi=\calJ^A_{\scrW_{\psi_F}}(\pi)(a_\tau),
\]
where $a_\tau$ is the element in $\calS_\psi$ corresponding to $\psi_\tau$.
\end{enumerate}
Similar results also hold for the group $H$.
\end{theorem}
\begin{proof}
The first and second claims directly follow from \cite{MR3135650} Proposition 2.4.3 in Case $O$, and \cite{MR3338302} Proposition 3.4.4 in Case $U$. We show the last claim. Indeed, this proof is almost the same as that of \cite{MR3801418} Theorem 2.4. In that paper, Atobe proved this LIR-B in the tempered case. We shall use some notations occuring in \cite{MR3135650} diagram (2.4.3) or in \cite{MR3338302} diagram (3.4.2) without explanations.\\

Since $\psi_\tau$ and $\psi_0$ are (conjugate) self-dual representations (with appropriate parity), there are non-degenerated bilinear forms $B_\tau\left(\cdot,\cdot\right)$ and $B_{V_0}\left(\cdot,\cdot\right)$ on $\CC^k$ and $\CC^{N_0}$ respectively, where $N_0$ is the dimension of the standard representation of $\wh G_0$, such that they are preserved by $\psi_\tau$ and $\psi_0$, in the sense that for all $w\in L_E\times SL_2$, we have
\[
	B_\tau\left(\psi_\tau(w)v,\psi_\tau^c(w)v'\right)=B_\tau\left(v,v'\right)\quad\textit{for all }v,v'\in\CC^k,
\]
and
\[
	B_{V_0}\left(\psi_0(w)v,\psi_0^c(w)v'\right)=B_{V_0}\left(v,v'\right)\quad\textit{for all }v,v'\in\CC^{N_0}.
\]
Let $A_\tau$ and $A_0$ be two matrices represent these two forms respectively. We regard $\widehat{G}$ as the $c$-isometry group with respect to the biliner form represented by the matrix $diag(A_\tau,A_0,-A_\tau)$. Let $\{e'_1,\cdots,e'_k,e_1,\cdots,e_{N_0},e''_1,\cdots,e''_k\}$ be the canonical basis of $\wh V=\CC^{N}$, where $N=N_0+2k$. Then $\widehat{M}_P$ can be realized as the Levi subgroup of $\widehat{G}$ stabilizing two isotropic subspaces
\begin{equation*}
\wh X=Span\{e'_i+e''_i~|~i=1,\cdots,k\}\quad\textit{and}\quad \wh X^*=Span\{e'_i-e''_i~|~i=1,\cdots,k\}.
\end{equation*}
Note that the image of $\psi$ stabilizes these two subspaces, so we can also regard $\widehat{G}_0$ as the $c$-isometry group with respect to the bilinear form represent by the matrix $A_0$ on $\wh V_0=Span\{e_1,\cdots,e_{N_0}\}$. Via these identifications, we have $\widehat{M}_P\simeq GL(\wh X)\times\widehat{G}_0$.\\

Let $\psi_{M_P}=(\psi_\tau,\psi_0)$ be a local $A$-parameter for $M_P$, and $\pi_{M_P}=\tau\boxtimes\pi_0$ be an irreducible unitary representation lies in $\Pi_{\psi_{M_P}}^A(M_P)$. Let $\scrU\in\widehat{G}$ be the element which acts on $\{e'_1,\cdots,e'_k,e_1,\cdots,e_{N_0}\}$ by $1$ and on $\{e''_1,\cdots,e''_k\}$ by $-1$. Then $\scrU\in Norm\left(A_{\widehat{M}_P},S_\psi\right)$. We write $u$ for the image of $\scrU$ in $\gothN_\psi(G,M_P)$. One can easily check the following
\begin{itemize}
\item the image of $u$ in $W_\psi(G,M_P)\subset W(M_P)$ is the unique non-trivial element $w$;
\item the image of $u$ in $\calS_\psi(G,M_P)\subset\calS_\psi$ is $a_\tau$;
\item as endomorphisms of $\Ind_P^G\left(\pi_{M_P}\right)$, $R_P(w_u,\widetilde{\pi}_{M_P}, \psi_{M_P})=R(w,\tau\boxtimes\pi_0)$. 
\end{itemize}
By applying the endoscopic character identity (\cite{MR3135650} Theorem 2.2.1 in Case $O$ or \cite{MR3338302} Theorem 3.2.1 in Case $U$) and the original local intertwining relation (\cite{MR3135650} Theoerem 2.4.1 in Case $O$ or \cite{MR3338302} Theorem 3.4.3 in Case $U$) to $u$, we obtain
\begin{equation}\label{Key.Eqn.LIR-A.To.LIR-B}
\sum_{\pi\in\Pi_\psi^A(G)}\calJ^A_{\scrW_{\psi_F}}(\pi)(a_\tau) \cdot \Theta_\pi(f)=\sum_{\pi_0\in\Pi_{\psi_0}^A(G_0)}\langle\widetilde{u},\widetilde{\tau\boxtimes\pi_0}\rangle\Tr\left(R(w,\tau\boxtimes\pi_0)\Ind_P^G(\tau\boxtimes\pi_0,f)\right)
\end{equation}
for any $f\in\calH(G)$. The constant $\langle\widetilde{u},\widetilde{\tau\boxtimes\pi_0}\rangle$ can be computed as follows (see \cite{MR3338302} bottom of page 62 in Case $U$): the restriction map
\[
	Norm\left(A_{\widehat{M}_P},S_\psi\right)\lra \widehat{G}_0,\quad\scrU'\longmapsto \scrU'\Big|_{V_0}
\]
induces a section
\begin{equation*}
\goths:\gothN_\psi(G,M_P)\lra\calS_{\psi_0}
\end{equation*}
then we have
\begin{equation*}
\langle\widetilde{u},\widetilde{\tau\boxtimes\pi_0}\rangle=\calJ^A_{\scrW_{\psi_F}}(\pi_0)\left(\goths(u)\right)
\end{equation*}
By the definition of $\scrU$ and $u$ we have $\goths(u)=1$, thus $\langle\widetilde{u},\widetilde{\tau\boxtimes\pi_0}\rangle=1$. Then equation (\ref{Key.Eqn.LIR-A.To.LIR-B}) together M{\oe}glin's multiplicity-freeness result Theorem \ref{Moeglin.Multi.Free} will imply that
\begin{equation*}
\sum_{\pi\in\Pi_\psi^A(G)}\calJ^A_{\scrW_{\psi_F}}(\pi)(a_\tau)\cdot \Theta_\pi(f)=\sum_{\pi_0\in\Pi_{\psi_0}^A(G_0)}\sum_{\pi\subset\Ind_P^G(\tau\boxtimes\pi_0)}\Tr(R(w,\tau\boxtimes\pi_0)\pi(f))
\end{equation*}
for any $f\in\calH(G)$. Therefore by Schur's Lemma and linear independence of characters, we have
\[
	R(w,\tau\boxtimes\pi_0)\Big|_\pi=\calJ^A_{\scrW_{\psi_F}}(\pi)(a_\tau).
\]
\end{proof}

There is also an anolog for $\theta$-packets.
\begin{theorem}[LIR-B for $\theta$-packets]\label{LIR-B.PIF}
Suppose that $G=G(V^\epsilon)$ is an even orthogonal or unitary group as in Section \ref{Compatible.Theta.N.Ind} (which is not necessarily quasi-split). Let $\pi$ be an irreducible constituent of $\Ind_P^{G}(\tau\boxtimes\pi_0)$. Then:
\begin{enumerate}
\item $\pi$ is in the local $\theta$-packet $\Pi_\psi^{\theta}(G)$, where
	\[
		\psi=\psi_\tau+\psi_0+\left(\psi_\tau^c\right)^\vee;
	\]
\item if we regard $\calS_{\psi_0}$ as a subgroup of $\calS_\psi$ via the natural embedding, then
\begin{equation*}
\calJ_{\psi_F}(\pi)\Big|_{\calS_{\psi_0}}=\calJ_{\psi_F}(\pi_0);
\end{equation*}
\item if we further assume that $\psi_\tau$ is (conjugate) self-dual with the same parity as $\psi$, then the restriction of the normalized intertwining operator $R(w,\tau\boxtimes\pi_0,\psi_F)$ to $\pi$ is the scalar multiplication by
\[
	R(w,\tau\boxtimes\pi_0)\Big|_\pi=\epsilon^k\cdot\calJ_{\psi_F}(\pi)(a_\tau),
\]
where $a_\tau$ is the element in $\calS_\psi$ which corresponds to $\psi_\tau$.
\end{enumerate}
\end{theorem}
We shall devote to proving this theorem in the next section. Combining this theorem with Corollary \ref{Main.Theorem.Set.Level} and Theorem \ref{LIR-B.Qusi.Split}, we deduce
\begin{corollary}\label{Main.Theorem.Labeling.As.Cor.LIR-B}
Theorem \ref{Main.Theorem.Packets} holds.
\end{corollary}
\begin{proof}
In Corollary \ref{Main.Theorem.Set.Level} we have proved that as sets, $\theta$-packets are independent of the choice of $H=H\left(W_{(r)}\right)$. It remains to show that the ``labeling'' is also independent of the choice of $H=H\left(W_{(r)}\right)$. Let $G=G(V^\epsilon)$, $\psi$ be a local $A$-parameter for $G$, and $\pi$ be an irreducible unitary representation in the packet $\Pi_\psi^\theta(G)$. We shall prove that $\calJ_{\psi_F}(\pi)$ is also independent of the choice of $H$.\\

Let $\psi_{\tau_i}$ be any irreducible (conjugate) self-dual subrepresentation of $\psi$, and also with the same parity as $\psi$. Then $\psi_{\tau_i}$ correponds to an irreducible unitary representation $\tau_i$ of $GL_d(E)$, for some $d\leq \dim V$. Let $\widetilde{V}=V+\calH^d$, where $\calH$ is the ($c$-Hermitian) hyperbolic plane. We can decompose $\widetilde{V}$ as following
\begin{equation*}
\widetilde{V}=X_{\tau_i}+ V+ X_{\tau_i}^*,
\end{equation*}
where $X_{\tau_i}$ and $X_{\tau_i}^*$ are $d$-dimensional totally isotropic subspaces of $\widetilde{V}$ such that $X_{\tau_i}\oplus X_{\tau_i}^*\simeq\calH^d$ and orthogonal to $V$. Let $\wt P$ be the maximal parabolic subgroup of $\wt G=G(\wt V)$ stabilizing $X_{\tau_i}$ and $\wt L$ be its Levi component stabilizing $X_{\tau_i}^*$, so that
\begin{equation*}
\wt L\simeq GL(X_{\tau_i})\times G.
\end{equation*}
We consider the induced representation $\Ind_{\wt P}^{\wt G}(\tau_i\boxtimes\pi)$. Let $\wt{\pi}$ be any irreducible constituent of $\Ind_{\wt P}^{\wt G}(\tau_i\boxtimes\pi)$. By Theorem \ref{LIR-B.PIF}, we know that it lies in the $\theta$-packet $\Pi_{\wt\psi}^\theta(\wt G)$, with
\[
	\wt\psi=\psi_{\tau_i}+\psi+\left(\psi_{\tau_i}^c\right)^\vee;
\]
also the ``labeling'' of $\wt\pi$ is related to the ``labeling'' of $\pi$ by
\[
	\calJ_{\psi_F}(\wt\pi)\Big|_{\calS_{\psi}}=\calJ_{\psi_F}(\pi),
\]
where we use the natural map $\calS_{\psi}\simeq\calS_{\wt\psi}$ to identify $\calS_{\psi}$ with $\calS_{\wt\psi}$. Let $R(w,\tau_i\boxtimes\pi)$ be the normalized intertwining operator defined in Section \ref{Def.NLIO}. Then Theorem \ref{LIR-B.PIF} also asserts that
\[
	\calJ_{\psi_F}(\wt\pi)(a_{\tau_i})=\epsilon^d\cdot R(w,\tau_i\boxtimes\pi)\big|_{\wt\pi},
\]
where $a_{\tau_i}$ is the element in $\calS_\psi$ corresponding to $\psi_{\tau_i}$. Since $\psi_{\tau_i}$ is arbitrary, and $\epsilon^d\cdot R(w,\tau_i\boxtimes\pi)\big|_{\wt\pi}$ is obviously independent of the choice of $H$, it follows that $\calJ_{\psi_F}(\pi)$ is also independent of the choice of $H$.\\

When $\epsilon=+1$, $G=G^*$ is quasi-split, we have also proved in Corollary \ref{Main.Theorem.Set.Level} that as sets, $\theta$-packets and $A$-packets coincide. In this case by Theorem \ref{LIR-B.Qusi.Split} we also have
\[
	\calJ^A_{\scrW_{\psi_F}}(\wt\pi)\Big|_{\calS_{\psi}}=\calJ^A_{\scrW_{\psi_F}}(\pi),
\]
and
\[
	\calJ^A_{\scrW_{\psi_F}}(\wt\pi)(a_{\tau_i})= R(w,\tau_i\boxtimes\pi)\big|_{\wt\pi}.
\]
These equalities imply that 
\[
	\calJ_{\psi_F}(\pi)=\calJ^A_{\scrW_{\psi_F}}(\pi).
\]
This completes the proof.
\end{proof}

\subsection{Changes of Whittaker data}
As an application of the LIR-B, we prove a formula which concerns the behavior of the ``labeling'' in a local $A$-packet for $H$ with respect to changes of Whittaker data.\\

In this subsection, we shall temporarily use $\psi$ to denote a local $A$-parameter for $H$. Let $\sigma$ be an irreducible unitary representation in $\Pi_\psi^A(H)$. Let $\scrW_{\psi_F}$ and $\scrW_{\psi_{F,c}}$ be the two Whittaker data of $H$, associated to the additive character $\psi_F$ and $\psi_{F,c}$ respectively, where $c\in F^\times$.
\begin{lemma}\label{Change.Whit.Data}
Let $\eta=\calJ^A_{\scrW_{\psi_F}}(\sigma)$ and $\eta'=\calJ^A_{\scrW_{\psi_{F,c}}}(\sigma)$. Then we have
\begin{equation*}
\eta'=\eta\cdot\eta_{\psi,c},
\end{equation*}
where $\eta_{\psi,c}$ is the character of $\calS_\psi$ defined by
\begin{equation*}
\eta_{\psi,c}(a_i)=\det(\psi_i)(c),
\end{equation*}
for the element $a_i$ in $\calS_\psi$ which corresponds to an irreducible constituent $\psi_i$ of $\psi$.
\end{lemma}
\begin{proof}
Similar to the proof of Corollary \ref{Main.Theorem.Labeling.As.Cor.LIR-B}, given $G,\psi,\pi$ and $\psi_{\tau_i}$, we define $\wt V,\wt G,\wt P,\wt L$ and $\wt\pi$. Then we have
\[
	\eta(a_{\tau_i})=R(w,\tau_i\boxtimes\pi,\psi_F)\big|_{\wt\pi},
\]
and also
\[
	\eta'(a_{\tau_i})=R(w,\tau_i\boxtimes\pi,\psi_{F,c})\big|_{\wt\pi}.
\]
From the definition of (normalized) local intertwining operators, one can easily show that
\[
	R(w,\tau_i\boxtimes\pi,\psi_{F,c})=R(w,\tau_i\boxtimes\pi,\psi_F)\cdot\omega_{\tau_i}(c).
\]
Hence we have
\[
	\eta'(a_{\tau_i})=\eta(a_{\tau_i})\cdot\det(\psi_{\tau_i})(c).
\]
Since $\psi_{\tau_i}$ is arbitrary, we conclude that
\[
	\eta'=\eta\cdot\eta_{\psi,c}.
\]
This completes the proof.
\end{proof}
From the proof of this lemma one can see that certainly an analog of this lemma will also hold for the group $G$. We omit the details here.

\section{Completion of the proof}
In this section we prove Lemma \ref{holomorphy.NIO} and Theorem \ref{LIR-B.PIF}. These results will complete our proof of Theorem \ref{Main.Theorem.Packets}.
\subsection{A diagram}
We retain the notations in the last section. Having fixed irreducible unitary representations $\tau$, $\pi_0$, and $\sigma_0$, we shall write
\begin{align*}
\calR({\wt w}_P,s)&=\calR({\wt w}_P,\tau_s\chi_W\boxtimes\pi_0),\\
\calR({\wt w}_Q,s)&=\calR\left({\wt w}_Q,\tau^c_s\chi^c_V\boxtimes\sigma_0^\vee\right)
\end{align*}
for the normalized intertwining operators. Recall that in Section \ref{Compatible.Theta.N.Ind}, we have constructed a $G\times H$-equivariant map
\begin{equation*}
\calT_s:\omega\otimes\Ind_Q^{H}\left(\tau_s^c\chi_V^c\boxtimes\sigma_0^\vee\right)\lra\Ind_P^{G}\left(\tau_s\chi_W\boxtimes\pi_0\right),
\end{equation*}
where $\sigma_0=\theta_{\psi_F,V_0,W_0}(\pi_0)$. By the Howe duality, the diagram
\begin{equation}\label{Diagram.Commutative}
\begin{CD}
\omega\otimes\Ind_Q^{H}\left(\tau^c_s\chi^c_V\boxtimes\sigma_0^\vee\right) @>{\calT_s}>> \Ind_P^{G}\left(\tau_s\chi_W\boxtimes\pi_0\right)\\
@V{1\otimes\calR({\wt w}_Q,s)}VV @VV{\calR({\wt w}_P,s)}V\\
\omega\otimes\Ind_Q^{H}\left(w_P\left(\tau^c_s\chi^c_V\boxtimes\sigma_0^\vee\right)\right) @>{\calT_{-s}}>> \Ind_P^{G}\left(w_Q\left(\tau_s\chi_W\boxtimes\pi_0\right)\right)
\end{CD}
\end{equation}
commutes up to a scalar. This scalar can be computed explicitly as follows. Recall that $\tau\in\Pi_{\psi_\tau}^{A}\left(GL_k(E)\right)$, $\pi_0\in\Pi_{\psi_0}^{\theta}(G_0)$, and $\sigma_0\in\Pi_{\theta(\psi_0)}^{A}(H)$ respectively, where
\[
	\theta(\psi_0)=\psi_0\chi_W^{-1}\chi_V+\chi_V\boxtimes S_{2r-2n+1}.
\]
Let $\phi_\tau$, $\phi_0$, and $\phi'_0$ be the $L$-parameter associated to $\psi_\tau$, $\psi_0$, and $\theta(\psi_0)$ respectively. Then
\begin{proposition}\label{CompaLIO}
For $\varphi\in\scrS$ and $\varPhi_s\in\Ind_Q^{H}(\tau_s^c\chi^c_V\boxtimes\sigma_0^\vee)$, we have
\begin{equation*}
\calR(\wt{w}_P,\tau_s\chi_W\boxtimes\pi_0)\calT_s(\varphi\otimes\varPhi_s)=\alpha\cdot\beta(s)\cdot\calT_{-s}\left(\varphi\otimes\calR\left(\wt{w}_Q,\tau^c_s\chi^c_V\boxtimes\sigma_0^\vee\right)\varPhi_s\right),
\end{equation*}
where
\[
\alpha=\begin{cases}
	   	    \gamma_V^k\cdot\chi_V\left(-1\right)^k\cdot\omega_\tau\left((-1)^{r-n+1}\cdot c^{-1}\right)\cdot\lambda\left(\wt w_P,\psi_F\right)^{-1}\quad&\textit{Case $O$};\\
	   	    ~\\
	   	    \left[\gamma_W^{-1}\cdot\gamma_V\cdot\chi_W\left((-1)^{n'-1}\cdot\kappa_V^{-1}\right)\cdot\chi_V\left((-1)^{r'-1}\cdot\kappa_W^{-1}\right)\cdot(\chi_W^{-\dim V}\chi_V^{\dim W})(\delta)\right]^k\\
			\quad\times\omega_\tau\left((-1)^{n'+r'-1}\cdot\kappa_W^c\kappa_V^{-1}\right)\cdot\lambda(\wt w_Q,\psi_F)\cdot\lambda(\wt w_P,\psi_F)^{-1}\quad &\textit{Case $U$},
	   \end{cases}
\]
and
\begin{align*}
\beta(s)=&L\left(s-s_0,\phi_\tau\right)^{-1}\cdot L\left(-s-s_0,\left(\phi_\tau^c\right)^\vee\right)\\
&\times\gamma\left(-s-s_0,\left(\phi_\tau^c\right)^\vee,\psi_E\right)\cdot|\kappa_W\kappa_V^{-1}|^{ks}\\
&\times\gamma\left(s,\phi_\tau^c\chi_V^c\otimes\phi'_0,\psi_E\right)^{-1}\cdot\gamma\left(s,\phi_\tau\chi_W\otimes\phi_0^\vee,\psi_E\right).
\end{align*}
\end{proposition}
\begin{proof}
Similar to \cite{MR3573972} Proposition 8.4 and Corollary 8.5.
\end{proof}
\begin{lemma}
The function $\beta(s)$ is holomorphic at $s=0$. 
\end{lemma}
\begin{proof}
Indeed, by the assumption that $r>\dim V$, we have $s_0>k$. As $\tau$ is of Arthur type, it follows from Remark \ref{Arthur-type.L-function.Stripe} that $L\left(s-s_0,\phi_\tau\right)^{-1}$, $L\left(-s-s_0,\left(\phi_\tau^c\right)^\vee\right)$, and $\gamma\left(-s-s_0,\left(\phi_\tau^c\right)^\vee,\psi_E\right)$ should be holomorphic at $s=0$. On the other hand, from the definitions one can easily see that
\[
	\phi'_0=\phi_0\chi_W^{-1}\chi_V+\chi_V\cdot\left(\bigoplus_{i=n-r}^{r-n}|\cdot|^i\right);
\]
hence
\begin{equation}\label{holomorphicity-beta.1}
	\gamma\left(s,\phi_\tau^c\chi_V^c\otimes\phi'_0,\psi_E\right)^{-1}\cdot\gamma\left(s,\phi_\tau\chi_W\otimes\phi_0^\vee,\psi_E\right)=\prod_{j=n-r}^{r-n}\gamma\left(s+j,\phi_\tau,\psi_E\right)^{-1}.
\end{equation}
If we write the $A$-parameter $\psi_\tau$ as 
\[
	\psi_\tau=\sum_i\rho_i\boxtimes S_{a_i}\boxtimes S_{b_i},
\]
and let $\tau_i$ be the irreducible unitary representation of some general linear group corresponding to the $A$-parameter $\rho_i\boxtimes S_{a_i}\boxtimes S_{b_i}$, then the RHS of equality (\ref{holomorphicity-beta.1}) can be written as
\[
	\prod_{j=n-r}^{r-n}\gamma\left(s+j,\phi_\tau,\psi_E\right)^{-1}=\prod_i\prod_{j=n-r}^{r-n}\gamma\left(s+j,\tau_i,\psi_E\right)^{-1}.
\]
Again, as explicated in Remark \ref{Arthur-type.L-function.Stripe}, if $\rho_i\not\simeq\mathbbm{1}$, then
\[
	\prod_{j=n-r}^{r-n}\gamma\left(s+j,\tau_i,\psi_E\right)^{-1}
\]
is holomorphic at $s=0$; otherwise if $\rho_i=\mathbbm{1}$, then $\tau_i\simeq\tau_i^\vee$, and it follows from the functional equation that
\[
	\prod_{j=n-r}^{r-n}\gamma\left(s+j,\tau_i,\psi_E\right)^{-1}=\gamma\left(s-s_0,\tau_i,\psi_E\right)^{-1}\cdot\prod_{j=1}^{r-n}\Big(\gamma\left(s+j,\tau_i,\psi_E\right)\cdot\gamma\left(s+1-j,\tau_i^\vee,\psi_E\right)\Big)^{-1}
\]
is also holomorphic at $s=0$. Thus we can conclude that the function $\beta(s)$ is holomorphic at $s=0$. 
\end{proof}
We deduce from diagram (\ref{Diagram.Commutative}) that:
\begin{corollary}
The normalized intertwining operator $\calR({\wt w}_P,\tau_s\chi_W\boxtimes\pi_0)$ is holomorphic at $s=0$.
\end{corollary}
\begin{proof}
Since $\tau$ is of Arthur type, when $r>\dim V$, the requirements in Proposition \ref{Non.Vanish.Equi.Map} and Corollary \ref{thetaNind} are automatically satisfied. It follows that the equivariant map $\calT_s$ is surjective at $s=0$. This fact allows us to ``approximate'' any holomorphic section of $\Ind_P^{G}\left(\tau_s\chi_W\boxtimes\pi_0\right)$ by images of $\calT_s$ at $s=0$.\\

Let $\varPhi_s$ be a holomorphic section of $\Ind_P^{G}\left(\tau_s\chi_W\boxtimes\pi_0\right)$. Since $\calT_s$ is surjective at $s=0$, we may take $\varphi_0\in\scrS$ and $\varPsi^{(0)}\in\Ind_Q^H\left(\tau^c\chi^c_V\boxtimes\sigma_0^\vee\right)$, such that
\[
	\calT_0\left(\varphi_0\otimes\varPsi^{(0)}\right)=\varPhi_0.
\]
We extend $\varPsi^{(0)}$ to a holomorphic section $\varPsi_s^{(0)}$ of $\Ind_Q^H\left(\tau^c_s\chi^c_V\boxtimes\sigma_0^\vee\right)$. Then we have
\[
	\varPhi_s=\calT_s\left(\varphi_0\otimes\varPsi_s^{(0)}\right)+s\cdot\varPhi_s^{(1)},
\]
where
\[
	\varPhi_s^{(1)}=s^{-1}\cdot\bigg(\varPhi_s-\calT_s\left(\varphi_0\otimes\varPsi_s^{(0)}\right)\bigg)
\]
is again a holomorphic section of $\Ind_P^{G}\left(\tau_s\chi_W\boxtimes\pi_0\right)$. Repeat this procedure, for any positive integer $k$, we obtain an expansion
\[
	\varPhi_s=\sum_{0\leq i<k}s^i\cdot\calT_s\left(\varphi_i\otimes\varPsi_s^{(i)}\right)+s^k\cdot\varPhi_s^{(k)}
\]
for some $\varphi_i\in\scrS$, $\varPsi_s^{(i)}$ holomorphic sections of $\Ind_Q^H\left(\tau^c_s\chi^c_V\boxtimes\sigma_0^\vee\right)$, and another holomorphic section $\varPhi_s^{(k)}$ of $\Ind_P^{G}\left(\tau_s\chi_W\boxtimes\pi_0\right)$. Hence
\begin{align*}
	\calR\left({\wt w}_P,\tau_s\chi_W\boxtimes\pi_0\right)\varPhi_s&=\calR\left({\wt w}_P,\tau_s\chi_W\boxtimes\pi_0\right)\left(\sum_{0\leq i<k}s^i\cdot\calT_s\left(\varphi_i\otimes\varPsi_s^{(i)}\right)\right)+s^k\cdot\calR\left({\wt w}_P,\tau_s\chi_W\boxtimes\pi_0\right)\varPhi_s^{(k)}\\
	&=\alpha\cdot\beta(s)\cdot\left(\sum_{0\leq i<k}s^i\cdot\calT_{-s}\left(\varphi_i\otimes\calR\left(\wt{w}_Q,\tau^c_s\chi^c_V\boxtimes\sigma_0^\vee\right)\varPsi_s^{(i)}\right)\right)\\
	&\qquad\qquad+s^k\cdot\calR\left({\wt w}_P,\tau_s\chi_W\boxtimes\pi_0\right)\varPhi_s^{(k)},
\end{align*}
where the second equality follows from the diagram (\ref{Diagram.Commutative}). Since $\beta(s)$, $\calT_{-s}$, and $\calR\left(\wt{w}_Q,\tau^c_s\chi^c_V\boxtimes\sigma_0^\vee\right)$ are all holomorphic at $s=0$, we know that the first term in the last equality is holomorphic at $s=0$. On the other hand, since we already know that $\calR\left({\wt w}_P,\tau_s\chi_W\boxtimes\pi_0\right)$ is meromorphic, we may take the positive integer $k$ to be sufficiently large, such that
\[
	s^k\cdot\calR\left({\wt w}_P,\tau_s\chi_W\boxtimes\pi_0\right)
\]
is holomorphic at $s=0$; then the second term in the last equality is also holomorphic at $s=0$. It follows that $\calR\left({\wt w}_P,\tau_s\chi_W\boxtimes\pi_0\right)\varPhi_s$ is holomorphic at $s=0$.
\end{proof}
This corollary implies that Lemma \ref{holomorphy.NIO} holds.

\subsection{Contragredient and Arthur packets}
To compute the ``labeling'' of $\theta$-packets using the diagram (\ref{Diagram.Commutative}), we also need to know the behavior of $A$-parameters and characters of component groups under taking contragredient. In this subsection, we prove such a formula for the group $H$.
\begin{proposition}\label{Contragredient}
Let $\psi_H$ be a local $A$-parameter for $H$, and $\sigma\in\Pi_{\psi_H}^A(H)$ an irreducible unitary representation. Then
\begin{enumerate}
	\item $\sigma^\vee$ lie in the $A$-packet $\Pi_{\psi_H^\vee}^A(H)$;
	\item let $\eta_{\sigma}$ be the character of $\calS_{\psi_H}$ associated to $\sigma$, and $\eta_{\sigma^\vee}$ be the character of $\calS_{\psi_H^\vee}$ associated to $\sigma^\vee$, both with respective to the Whittaker datum $\scrW_{\psi_F}$ of $H$ associated to the additive character $\psi_F$, we have
		\[
			\eta_{\sigma^\vee}=\eta_\sigma\cdot\nu,
		\]
	where we use the obvious isomorphism between $\calS_{\psi_H}$ and $\calS_{\psi_H^\vee}$ to identify them, and the character $\nu$ of $\calS_{\psi_H}$ is defined by
		\[
			\nu(a_i)=\det(\psi_{H,i})(-1)
		\]
	for the element $a_i$ in $\calS_{\psi_H}$ which corresponds to an irreducible constituent $\psi_{H,i}$ of $\psi_H$.
\end{enumerate}
\end{proposition}
\begin{proof}
Since the local $A$-packets for general $A$-parameters can be constructed using the parabolic induction from the good parity case, without loss of generality, we may assume that $\psi_H$ is of good parity.\\

Indeed, if $\sigma$ lies in the $L$-packet $\Pi_{\phi_{\psi_H}}^L(H)$ inside $\Pi_{\psi_H}^A(H)$, then the desired conclusions were already proved by Kaletha in \cite{MR3194648}. Hence in particular, if $\psi_H=\phi_H$ is a square-integrable $L$-parameter (regarded as an $A$-parameter trivial on Arthur $SL_2$) for $H$, then this proposition holds for $\psi_H$. We shall prove the good parity case based on this.\\

We first assume that $\psi_H$ is an elementary $A$-parameter for $H$, and is trivial on Weil-Deligne $SL_2$. Then we have
\[
	\psi_H=\wh\phi_H
\]
for some square-integrable $L$-parameter $\phi_H$ for $H$, where we use $\wh\phi_H$ to denote the Aubert involution of $\phi_H$. Since the Aubert involution commutes with taking contragredient (see \cite{MR1285969} Th\'eor\`eme 1.7), it follows from the compatibilities of $A$-packets and the Aubert involution that the proposition also holds for these $\psi_H$.\\

Next we appeal to the global method to prove this proposition for any $\psi_H$ of good parity. Let $\sigma$ be an irreducible unitary representation in $\Pi_{\psi_H}^A(H)$. Similar to the proof of Corollary \ref{globalize.quasi-split}, we may construct a tuple of data $(\dot{F},\dot{E},\dot{H},\dot{\psi}_H,u_1,u_2,w)$, where 
\begin{itemize}
\item $\dot{F}$ is a number field, and $\dot{E}$ is either $\dot F$ itself or a quadratic extension of $\dot F$, according the cases;
\item $\dot{H}$ a symplectic or quasi-split  unitary group over $\dot{F}$, according to the group $H$; in the case that $\dot H$ is an unitary group, $\dot E$ is the splitting field of $\dot H$;
\item $\dot{\psi}_H$ is an elliptic $A$-parameter for $\dot{H}$;
\item $u_1$, $u_2$, and $w$ are finite places of $\dot{F}$.
\end{itemize}
such that the following conditions hold:
\begin{enumerate}
\item $(\dot{F}_{u_1},\dot{E}_{u_1},\dot{H}_{u_1},\dot{\psi}_{H,u_1})\simeq(\dot{F}_{u_2},\dot{E}_{u_2},\dot{H}_{u_2},\dot{\psi}_{H,u_2})\simeq(F,E,H,\psi_H)$;
\item if we are in the Case $U$, then $\dot E_w/\dot F_w$ is a quadratic field extension;
\item the localization maps $\calS_{\dot{\psi}_H}\lra\calS_{\dot{\psi}_{H,u_1}}$ and $\calS_{\dot{\psi}_H}\lra\calS_{\dot{\psi}_{H,u_2}}$ agree, and they are surjections;
\item $\dot\psi_{H,w}$ is elementary, and is trivial on the Weil-Deligne $SL_2$; moreover, the localization map $\calS_{\dot{\psi}}\lra\calS_{\dot{\psi}_{w}}$ is an isomorphism.
\end{enumerate}
Let $S$ be a finite set of places of $\dot{F}$, include $u_1$, $u_2$, $w$, and all Archimedean places, such that for all $v\notin S$, the group $\dot H_v$, and the local $A$-parameter $\dot\psi_{H,v}$ are both unramified. We construct an automorphic representation $\dot{\sigma}$ which occurs in the automorphic discrete spectrum of $\dot H$ with elliptic $A$-parameter $\dot\psi_H$ as follows:
\begin{itemize}
\item at a place $v\notin S$, $\dot\sigma_v$ is the unramified representation of $\dot H_v$ with $L$-parameter $\phi_{\dot\psi_{H,v}}$;
\item at a place $v\in S\backslash\{u_1,u_2,w\}$, let $\dot\sigma_v$ be an arbitrarily given representation of $\dot H_v$ lies in the $A$-packet $\Pi_{\dot\psi_{H,v}}^A(\dot H_v)$;
\item at the places $u_1$ and $u_2$, $\dot\sigma_{u_1}=\dot\sigma_{u_2}=\sigma$;
\item at the place $w$, $\dot\sigma_{w}$ lies in the $A$-packet $\Pi_{\dot\psi_{H,w}}^A(H)$, corresponds to the character $\eta_{\sigma_w}$ of $\calS_{\dot{\psi}_{H,w}}$, determined by the formula
    \[
        \prod_v\eta_v=\epsilon_{\dot\psi_H},
    \]
    where $\eta_v=\calJ_{\scrW_{\psi_F}}^A(\dot\sigma_v)$, and $\epsilon_{\dot\psi_H}$ is the canonical sign character associated to $\dot\psi_H$. It follows from Lemma \ref{surjectivity.ell.type} that $\dot\sigma_w\neq0$.
\end{itemize}
Then, according to the Arthur's multiplicity formula for $\dot H$, $\dot\sigma$ is an irreducible subrepresentation of $L^2_{\dot\psi_H}(\dot H)$. Consider the contragredient $\dot\sigma^\vee$ of $\dot\sigma$. It is not hard to see that $\dot\sigma^\vee$ also occurs in the automorphic discrete spectrum of $\dot H$, with elliptic $A$-parameter $\dot\psi_H^\vee$. Indeed, any realization $\calV$ of $\dot\sigma$ in $\calA^2(\dot H)$ gives a realization 
\[
	\overline{\calV}=\left\{\overline{f}~\big|~f\in\calV\right\}
\]
of $\dot\sigma^\vee$ in $\calA^2(\dot H)$, where $\overline{f}$ means the complex conjugate of the function $f$. By the Arthur's multiplicity formula for $\dot H$, localizing at the place $u_1$, we obtain
\[
	\sigma^\vee\in\Pi_{\psi_H^\vee}^A(H).
\]
Let $\epsilon_{\dot\psi_H^\vee}$ be the canonical sign character associated to ${\dot\psi_H^\vee}$. If we identify $\calS_{\dot\psi_H}$ and $\calS_{\dot\psi_H^\vee}$ in the obvious way, then
\[
\epsilon_{\dot\psi_H}=\epsilon_{\dot\psi_H^\vee}.
\]
Indeed, if $H$ is a symplectic group, this is obvious; if $H$ is an unitary group, this follows from the fact that the epsilon factor is invariant under the Galois conjugation. Then comparing Arthur's multiplicity formula for $\dot\sigma$ and $\dot\sigma^\vee$, we get
\begin{equation}\label{Contragredient.Auxiliary-1}
	\prod_{v\in S\backslash\{u_1,u_2,w\}}\eta_{\sigma_v^\vee}=\prod_{v\in S\backslash\{u_1,u_2,w\}}\eta_{\sigma_v}\cdot\nu_v
\end{equation}
as characters of $\calS_{\dot\psi_H}$. Here we use the fact that the proposition holds for $\dot\psi_{H,w}$ and places outside $S$.\\

Now we construct another automorphic representation $\dot{\sigma}'$ occuring in the automorphic discrete spectrum of $\dot H$ with elliptic $A$-parameter $\dot\psi_H$ as following:
\begin{itemize}
\item at a place $v\notin\{u_2,w\}$, $\sigma'_v=\sigma_v$;
\item at the place $u_2$, let $\dot\sigma'_{u_2}$ be an irreducible unitary representation in the $L$-packet $\Pi_{\phi_{\psi_H}}^L(H)$;
\item at the place $w$, $\dot\sigma'_{w}$ lies in the $A$-packet $\Pi_{\dot\psi_{H,w}}^A(H)$, corresponds to the character $\eta_{\sigma'_w}$ of $\calS_{\dot{\psi}_{H,w}}$, which is determined by the formula
    \[
        \prod_v\eta'_v=\epsilon_{\dot\psi_H},
    \]
    where $\eta'_v=\calJ_{\scrW_{\psi_F}}^A(\dot\sigma'_v)$. It follows from Lemma \ref{surjectivity.ell.type} that $\dot\sigma'_w\neq0$.
\end{itemize}
Then, according to the Arthur's multiplicity formula for $\dot H$, $\dot\sigma'$ is an irreducible subrepresentation of $L^2_{\dot\psi_H}(\dot H)$. Again, comparing the Arthur's multiplicity formula for $\dot\sigma'$ and its contragredient, we get
\begin{equation}\label{Contragredient.Auxiliary-2}
	\prod_{v\in S\backslash\{u_2,w\}}\eta_{\sigma_v^\vee}=\prod_{v\in S\backslash\{u_2,w\}}\eta_{\sigma_v}\cdot\nu_v
\end{equation}
as characters of $\calS_{\dot\psi_H}$. Combining these two equalities (\ref{Contragredient.Auxiliary-1}) and (\ref{Contragredient.Auxiliary-2}), we obtain the desired formula for $\eta_{\sigma^\vee}$. This completes the proof.
\end{proof}

\subsection{Calculation of the labeling}
Finally we are now able to calculate the actions of normalized intertwining operators on induced representations of $G$ explicitly.
\begin{proposition}\label{Proof.LIR-B.4.Theta-Pack}
Let $G=G(V^\epsilon)$ be an even orthogonal or unitary group, and
\[
	\psi=\psi_\tau+\psi_0+\left(\psi_\tau^c\right)^\vee,
\]
with $\psi_\tau$ and $\psi_0$ as in the setting of this section. Let $\pi$ be an irreducible constituent of $\Ind_P^{G}(\tau\chi_W\boxtimes\pi_0)$. Assume that $\psi_\tau$ is (conjugate) orthogonal. Then the restriction of the normalized intertwining operator $R(w,\tau\chi_W\boxtimes\pi_0,\psi_F)$ to $\pi$ is the scalar multiplication by 
\[
	R(w,\tau\chi_W\boxtimes\pi_0)\Big|_\pi=\epsilon^k\cdot\calJ_{\psi_F}(\pi)(a_\tau),
\]
where $a_\tau$ is the element in $\calS_\psi$ corresponding to $\psi_\tau\chi_W$.
\end{proposition}
\begin{proof}
Since $\tau$ is of Arthur type, when $r>\dim V$, the requirements in Proposition \ref{Non.Vanish.Equi.Map} and Corollary \ref{thetaNind} are automatically satisfied. It follows that the equivariant map
\[
	\calT_s:\omega\otimes\Ind_Q^{H}\left(\tau_s^c\chi_V^c\boxtimes\sigma_0^\vee\right)\lra\Ind_P^{G}\left(\tau_s\chi_W\boxtimes\pi_0\right)
\]
is surjective at $s=0$. Moreover, by Lemma \ref{indNmultifree} and Theorem \ref{Moeglin.Multi.Free}, the induced representations $\Ind_P^{G}\left(\tau\chi_W\boxtimes\pi_0\right)$ and $\Ind_Q^{H}\left(\tau^c\chi_V^c\boxtimes\sigma_0^\vee\right)$ are semi-simple and multiplicity-free. Therefore we can restrict the diagram (\ref{Diagram.Commutative}) to
\[
	\begin{CD}
		\omega\otimes\sigma^\vee @>{\calT_0}>> \pi\\
		@V{1\otimes R(w,\tau^c\chi^c_V\boxtimes\sigma_0,\psi_F)}VV @VV{R(w,\tau\chi_W\boxtimes\pi_0,\psi_F)}V\\
		\omega\otimes\sigma^\vee @>{\calT_{0}}>> \pi
	\end{CD},
\]
where $\sigma$ is the theta lift of $\pi$ to $H$. Applying Proposition \ref{CompaLIO} to this sub-diagram, we deduce
\begin{equation}
\begin{aligned}
	R(w,\tau\chi_W\boxtimes\pi_0,\psi_F)\Big|_\pi&=\alpha\cdot\beta(0)\cdot R(w,\tau^c\chi^c_V\boxtimes\sigma_0,\psi_F)\Big|_{\sigma^\vee}\\
	&=\alpha\cdot\beta(0)\cdot\calJ^A_{\scrW_{\psi_F}}(\sigma^\vee)(a'_\tau),
\end{aligned}\label{Compare.NLIO.Final}
\end{equation}
here $a'_\tau$ is the element in $\calS_{\theta(\psi)^\vee}$ corresponds to $\psi_\tau^c\chi_V^c$. Let $\phi_\tau$, $\phi_0$, and $\phi'_0$ be the $L$-parameter associated to $\psi_\tau$, $\psi_0$, and $\theta(\psi_0)$ respectively. Then we have 
\[
	\phi'_0=\phi_0\chi_W^{-1}\chi_V+\chi_V\cdot\left(\bigoplus_{i=n-r}^{r-n}|\cdot|^i\right),
\]
where $n$ is a integer which depends on the group $G$ (see the begining of Section \ref{Statements.Main.Results}). It follows that
\begin{align*}
	\beta(0)&=\prod_{i=n-r+1}^{r-n}\gamma\left(i,\phi_\tau,\psi_E\right)^{-1}\\
	&=\omega_\tau(-1)^{n-r}.
\end{align*}
Here we use the functional equation
\[
\gamma(s,\tau,\psi_E)\cdot\gamma(1-s,\tau^\vee,\psi_E)=\omega_\tau(-1).
\]
Now we calculate case by case.\\

\underline{Case $O$:} In this case, the Whittaker datum $\scrW'$ of $H$ is the Whittaker datum associated to the additive character $\psi_{F,c}$ (recall that we have fixed $d,c\in F^\times$ such that $V^+$ is of type $(d,c)$). Also, we have
\[
	\gamma_V\cdot\lambda\left(E'/F,\psi_F\right)^{-1}=\epsilon\cdot\chi_V(c)
\]
and 
\[
	\gamma_W=1.
\]
Substitute these into the equality (\ref{Compare.NLIO.Final}), we obtain
\begin{align*}
R(w,\tau\chi_W\boxtimes\pi_0,\psi_F)\Big|_\pi&=\epsilon^k\cdot\omega_\tau(-c)\cdot\chi_V(-c)^k\cdot\calJ^A_{\scrW_{\psi_F}}(\sigma^\vee)(a'_\tau) &~\\
&=\epsilon^k\cdot\omega_\tau(c)\cdot\chi_V(c)^k\cdot\calJ^A_{\scrW_{\psi_F}}(\sigma)(a_\tau) &\textit{(by Proposition \ref{Contragredient})}\\
&=\epsilon^k\cdot\calJ^A_{\scrW'}(\sigma)(a_\tau) &\textit{(change Whittaker data)}\\
&=\epsilon^k\cdot\calJ_{\psi_F}(\pi)(a_\tau). &\textit{(by our construction of }\calJ_{\psi_F}\textit{)}
\end{align*}
Hence the proposition holds in this case.\\

\underline{Case $U_0$:} In this case, the Whittaker datum $\scrW'$ of $H$ is just the Whittaker datum associated to the additive character $\psi_{F}$. Also, we have
\[
	\gamma_V=\epsilon
\]
and
\[
	\gamma_W\cdot\lambda\left(E/F,\psi_F\right)^{-1}=\chi_W(\delta)^{-1}.
\]
Substitute these into the equality (\ref{Compare.NLIO.Final}), we obtain
\begin{align*}
R(w,\tau\chi_W\boxtimes\pi_0,\psi_F)\Big|_\pi&=\epsilon^k\cdot\calJ^A_{\scrW'}(\sigma^\vee)(a'_\tau) &~\\
&=\epsilon^k\cdot\calJ^A_{\scrW'}(\sigma)(a_\tau) &\textit{(by Proposition \ref{Contragredient})}\\
&=\epsilon^k\cdot\calJ_{\psi_F}(\pi)(a_\tau). &\textit{(by our construction of }\calJ_{\psi_F}\textit{)}
\end{align*}
Hence the proposition holds in this case.\\

\underline{Case $U_1$:} In this case, the Whittaker datum $\scrW'$ of $H$ is just the Whittaker datum associated to the additive character $\psi_{F}$. Also, we have
\[
	\gamma_V\cdot\lambda\left(E/F,\psi_F\right)^{-1}=\epsilon
\]
and
\[
	\gamma_W=\chi_W(\delta)^{-1}.
\]
Substitute these into the equality (\ref{Compare.NLIO.Final}), we obtain
\begin{align*}
R(w,\tau\chi_W\boxtimes\pi_0,\psi_F)\Big|_\pi&=\epsilon^k\cdot\omega_\tau(-1)\cdot\chi_V(-1)^k\cdot\calJ^A_{\scrW'}(\sigma^\vee)(a'_\tau) &~\\
&=\epsilon^k\cdot\calJ^A_{\scrW'}(\sigma)(a_\tau) &\textit{(by Proposition \ref{Contragredient})}\\
&=\epsilon^k\cdot\calJ_{\psi_F}(\pi)(a_\tau). &\textit{(by our construction of }\calJ_{\psi_F}\textit{)}
\end{align*}
Hence the proposition holds in this case.\\
\end{proof}
\begin{corollary}
Theorem \ref{LIR-B.PIF} holds.
\end{corollary}
\begin{proof}
The first two statements follows from the construction of the $\theta$-packets and LIR-B for $H$ Theorem \ref{LIR-B.Qusi.Split}. The last statement follows from Proposition \ref{Proof.LIR-B.4.Theta-Pack}.
\end{proof}
So now, we have finished proving our main theorems.

\subsection{Summary}
To make things more clear, in this subsection, we shall briefly summarize some expected and known results for $\theta$-packets.\\

First we let $F$ be a local field of characteristic $0$, and $G=G(V)$ be an even orthogonal or unitary group over $F$. Let $\psi$ be a local $A$-parameter for $G$ with bounded image on the Weil group. Then, after choosing a symplectic or quasi-split unitary group $H$ over $F$ with sufficiently big split rank, and also a tuple of auxiliary data $(\psi_F,\chi_V,\chi_W)$ as in Section \ref{Statements.Main.Results}, one can define the $\theta$-packet $\Pi_\psi^\theta(G)$ by using the theta lift between $(G,H)$ with respect to $(\psi_F,\chi_V,\chi_W)$. It is a (multi) set of irreducible unitary representations of $G$, equipped with a map
\[
	\calJ_{\psi_F}:\Pi_\psi^\theta(G)\lra\wh{\calS_\psi}
\]
to the Pontryagin dual of the component group $\calS_\psi$. It can be also regarded as a representation of $\calS_\psi\times G$ as explicated in Section \ref{Def.Theta-Pack}. Let $\Pi_\psi^M(G)$ be the local packet contructed by M{\oe}glin in \cite{MR2767522}. We have:
\begin{theorem}\label{Theta-Pack.Summary}
When $F$ is non-Archimedean, the $\theta$-packet $\Pi_\psi^\theta(G)$ has the following properties:
\begin{enumerate}
\item as a representation of $\calS_\psi\times G$, it is independent of the choice of $H$; moreover, if $G=G^*$ is quasi-split, we have
\[
	\Pi_\psi^\theta(G^*)=\Pi_\psi^A(G^*)
\] 
as representations of $\calS_\psi\times G^*$;
\item as a representation of $G$, it is independent of the choice of auxiliary data $(\psi_F,\chi_V,\chi_W)$; moreover, it is multiplicity-free;
\item it coincides with $\Pi_\psi^M(G)$ (as sets); in particular, it contains the $L$-packet $\Pi_{\phi_\psi}^L(G)$ as a subset;
\item it satisfies LIR-B, as stated in Theorem \ref{LIR-B.PIF}.
\end{enumerate}
\end{theorem}
\begin{proof}
The $1$st statement is provided by Theorem \ref{Main.Theorem.Packets}; the $2$nd statement is provided by Section \ref{Chap.Indepedency.Aux.Data}; the $3$rd statement is provided by \cite{MR2906916}, and the last statement is just Theorem \ref{LIR-B.PIF}.
\end{proof}
We remark that, except for the $3$rd statement, the proofs of all other properties in this theorem do not rely on the construction of $\Pi_\psi^M(G)$ for non quasi-split $G$. We expect the following conjecture holds:
\begin{conjecture}
Similar results as stated in Theorem \ref{Theta-Pack.Summary} also hold when $F$ is real or complex.
\end{conjecture}

From the main body of this paper, we also conclude that:
\begin{theorem}
When $F$ is non-Archimedean, there is a commutative diagram
\begin{equation}\label{Diag-LocalCompare.Theta-Pack}
	\begin{CD}
        \Pi^A_{\theta(\psi)}(H) @>\calJ^A_{\scrW'}>> \widehat{\overline{\calS_{\theta(\psi)}}}\\
        @VV\theta V @VV\ell^*V\\
        \bigsqcup\Pi_\psi^\theta(G) @>\calJ_{\psi_F}>> \widehat{\calS_\psi}
    \end{CD}
\end{equation}
where the disjoint union runs over all pure inner forms of $G^*$, the arrow $\theta$ is a bijection given by the theta lift, and the arrow $\ell^*$ is the pull-back of the natural map
\[
	\ell:\calS_\psi\lra\calS_{\theta(\psi)}.
\] 
\end{theorem}
\begin{proof}
This is the combination of Theorem \ref{Main.Theorem.Packets} and Corollary \ref{Ind.Relation}.
\end{proof}
This can be more or less regarded as a refined version of Conjecture \ref{Intro.Adams.Conjecture} (B).\\

For our global purpose, we also need to treat those $A$-parameter of $G$ with non-bounded image on the Weil group, since the Ramanujan conjecture is not proved yet. Let $\psi$ be a local $A$-parameter of $G$ with non-bounded image on the Weil group, but we assume that $\psi$ is a localization of some global elliptic $A$-parameter for an even orthogonal or unitary group. We write
\begin{equation*}
\psi=(\psi_{\tau_1}|\cdot|^{s_1}+\cdots+\psi_{\tau_r}|\cdot|^{s_r})+\psi_0+\left((\psi_{\tau_1}|\cdot|^{s_1}+\cdots+\psi_{\tau_r}|\cdot|^{s_r})^c\right)^\vee,
\end{equation*}
where 
\begin{itemize}
\item for $i=1,\cdots, r$, $\psi_{\tau_i}$ is an irreducible representation of $L_E\times SL_2$ with bounded image on the Weil group, which corresponds to an irreducible unitary representation $\tau_i$ of $GL_{k_i}(E)$, and $s_i$ is a real number such that
\begin{equation*}
s_1\geq\cdots\geq s_r>0;
\end{equation*}
\item $\psi_0$ is a local $A$-parameter for some smaller group $G_0=G(V_0)$, where $V_0$ is the space in the Witt tower containing $V$ with appropriate dimension (if there is no such $V_0$, it follows from the induction principle of the theta lift \cite{MR818351} and Theorem \ref{Diag-LocalCompare.Theta-Pack} that the $\theta$-packet $\Pi_\psi^\theta(G)$ is empty). 
\end{itemize}
So we have a natural isomorphism $\calS_{\psi_0}\simeq\calS_\psi$. There is a parabolic subgroup of $G$, say $P$, with Levi component $M$, so that
\begin{equation*}
M\simeq GL_{k_1}(E)\times\cdots\times GL_{k_r}(E)\times G_0.
\end{equation*}
For any irreducible unitary representation $\pi_0\in\Pi_{\psi_0}^\theta(G_0)$, we denote by $I(\pi_0)$ the parabolic induction 
\begin{equation*}
\Ind_P^{G}\left(\tau_1|\det|^{s_1}\boxtimes\cdots\boxtimes\tau_r|\det|^{s_r}\boxtimes\pi_0\right).
\end{equation*}
We expect the following conjecture holds:
\begin{conjecture}\label{Unbounded.A-parameter.Conjecture}
The induced representation $I(\pi_0)$ is irreducible for any $\pi_0\in\Pi_{\psi_0}^\theta(G_0)$. Moreover, if 
\[
	I(\pi_0)\simeq I(\pi'_0)
\]
for some $\pi_0,\pi'_0\in\Pi_{\psi_0}^\theta(G_0)$, then we have $\pi_0\simeq \pi'_0$.
\end{conjecture}
About this conjecture, we have:
\begin{proposition}
\begin{enumerate}
\item Conjecture \ref{Unbounded.A-parameter.Conjecture} holds in the following cases:
	\begin{itemize}
		\item if the $A$-parameter $\psi$ is trivial on the Arthur $SL_2$;
		\item if $G=G^*$ is quasi-split, and $F$ is non-Archimedean.
	\end{itemize}
	Similar results also hold for the group $H$.
\item When $F$ is non-Archimedean, in general, for any irreducible unitary representation $\pi\in\Pi_\psi^\theta(G)$, there exists an unique $\pi_0\in\Pi_{\psi_0}^\theta(G_0)$, such that $\pi$ is a sub-quotient of $I(\pi_0)$, and 
\[
	\calJ_{\psi_F}(\pi)=\calJ_{\psi_F}(\pi_0),
\]
where we use the natural isomorphism between $\calS_\psi$ and $\calS_{\psi_0}$ to identify them. Moreover, the map
\begin{align*}
\Pi_\psi^\theta(G)&\lra\Pi_{\psi_0}^\theta(G_0),\\
\pi&\longmapsto\pi_0
\end{align*}
is a bijection.
\end{enumerate}
\end{proposition}
\begin{proof}
The two cases in $1$st statement follows from \cite{CZ2020AMFPIF} Section 6 and \cite{MR2822218} Proposition 5.1 respectively, and the $2$nd statement follows from the combination of the $1$st statement (for $H$), the induction principle of the theta lift \cite{MR818351}, and Theorem \ref{Diag-LocalCompare.Theta-Pack}.
\end{proof}

Finally we turn to the global properties of $\theta$-packets. Now let $F$ be a number field, and fix an additive character $\psi_F$. Let $G$ be an even orthogonal or unitary group over $F$. Given an elliptic $A$-parameter $\psi$ for $G$, we define the global packet $\Pi_\psi^\theta(G)$ associated to $\psi$ as the restricted tensor product of the local $\theta$-packets
\begin{align*}
\Pi_\psi^\theta(G)&=\otimes'_v\Pi_{\psi_v}^\theta(G_v)\\
&=\{\pi=\otimes'_v\pi_v~|~\pi_v\in\Pi_{\psi_v}^\theta(G_v),~\pi_v\textit{ unramified with the $L$-parameter $\phi_{\psi_v}$ for almost all }v\}.
\end{align*}
We then have a map
\begin{align*}
\calJ_{\psi_F}:\Pi_\psi^\theta(G)&\lra\wh{{\calS_\psi}},\\
\pi&\longmapsto\calJ_{\psi_F}(\pi),\\
\calJ_{\psi_F}(\pi)(x)&\coloneqq\prod_v\calJ_{\psi_{F_v}}(\pi_v)(x_v),
\end{align*}
where $x\in\calS_\psi$ and $x_v$ is the localization of $x$ at $v$. Let $\epsilon_\psi\in\wh{\calS_\psi}$ be the canonical sign character associated to $\psi$. We put
\begin{equation*}
\Pi_\psi^\theta(G,\epsilon_\psi)=\left\{\pi\in\Pi_\psi^\theta(G)~|~\calJ_{\psi_F}(\pi)=\epsilon_\psi\right\}.
\end{equation*}
The following conjecture is the ultimate goal of our series of works:
\begin{conjecture}\label{Ultimate.Goal}
Let $\psi$ be an elliptic $A$-parameter for $G$. Then we have the decomposition
\begin{equation*}
L^2_\psi(G)=\bigoplus_{\pi\in\Pi_\psi^\theta(G,\epsilon_\psi)}\pi.
\end{equation*}
\end{conjecture}
In our previous paper \cite{CZ2020AMFPIF}, we have proved that
\begin{theorem}\label{AMF.Anisotropic}
Conjecture \ref{Ultimate.Goal} holds if the Witt index of $G$ is less than or equal to one.
\end{theorem}
The author hopes that one day in the future, he can have a chance to prove Conjecture \ref{Ultimate.Goal} in full generality.

\bibliographystyle{alpha}
%\nocite{*}
\bibliography{ThetaNAPackRevisedRef}

\end{document}